\documentclass[review,onefignum,onetabnum]{siamart220329}

\usepackage{enumerate}
\usepackage{amsmath}
\usepackage{mathrsfs}
\usepackage{xcolor}	
\usepackage{amssymb}
\usepackage{bm}
\usepackage{graphicx}
\usepackage{subcaption}
\usepackage{adjustbox}
\usepackage{soul}
\usepackage{listings}
\usepackage{array}  
\usepackage{mathtools}
\usepackage{multirow}

\usepackage{mdframed}
\newmdenv[linecolor=green!50!black, fontcolor=green!50!black, backgroundcolor=green!20, linewidth=2pt, roundcorner=10pt]{gnote}

\usepackage{xcolor} 
\newcommand{\round}{\tilde{s}}

\newcommand{\red}[1]{\textcolor{black}{#1}}
\newcommand{\blue}[1]{\textcolor{black}{#1}}
\newcommand{\teal}[1]{\textcolor{black}{#1}}

\newcommand{\forward}{{\Theta}}
\newcommand{\littleforward}{{\theta}}
\newcommand{\vecx}{{\bx}}
\newcommand{\vecy}{{\by}}

\newcommand{\bigvecx}{{\bm{X}}}

\newcolumntype{L}{>{$}l<{$}} 

\newcommand{\ba}{{\bm a}}

\newcommand{\bc}{{\bm c}}
\newcommand{\bd}{{\bm d}}

\newcommand{\bw}{{\bm w}}
\newcommand{\bx}{{\bm x}}
\newcommand{\by}{{\bm y}}
\newcommand{\bz}{{\bm z}}

\newcommand{\bfz}{{\bm 0}}
\newcommand{\bfo}{{\bm 1}}

\newcommand{\cA}{\mathcal{A}}
\newcommand{\cB}{\mathcal{B}}
\newcommand{\cC}{\mathcal{C}}

\newcommand{\cI}{\mathcal{I}}

\newcommand{\cL}{L}

\newcommand{\cN}{\mathcal{N}}

\newcommand{\cP}{\mathcal{P}}

\newcommand{\cR}{\mathcal{R}}
\newcommand{\cS}{\mathcal{S}}

\newcommand{\cU}{\mathcal{U}}

\newcommand{\C}[1]{} 

\newcommand{\var}{\mathop{\mathrm{var}}}
\newcommand{\supp}{\mathop{\mathrm{supp}}}
\newcommand{\hp}{\hphantom{-}}
\newcommand{\uniform}{\ensuremath{\mathcal{U}}\xspace}

\usepackage{xspace}

\usepackage{xr-hyper}
\usepackage{hyperref}
\usepackage{cleveref}
\usepackage{cite} 

\headers{Enhancing ZFP: Tackling Compression Error Bias}{A. Fox and P. Lindstrom}

\title{Enhancing ZFP: A Statistical Approach to Understanding and Reducing Error Bias in a Lossy Floating-Point Compression Algorithm\thanks{Submitted to the editors July 1st, 2024.
		\funding{This work was
  supported by the LLNL-LDRD Program under Project No.\ 17-SI-004 and by the Office of Science, Office of Advanced Scientific Computing Research. LLNL-JRNL-858256.}}}
\author{ Alyson Fox\thanks{Lawrence Livermore National Laboratory,
		Livermore, CA (\email{fox33@llnl.gov})}
	\and Peter Lindstrom\thanks{Lawrence Livermore National Laboratory,
		Livermore, CA (\email{pl@llnl.gov})}
}

\ifpdf
\hypersetup{
  pdftitle={Enhancing ZFP: Understanding and Reducing Error Bias},
  pdfauthor={A. Fox, P. Lindstrom}
}
\fi

\begin{document}

\maketitle

\begin{abstract}
The amount of data generated and gathered in scientific simulations and data collection applications is continuously growing, putting mounting pressure on storage and bandwidth concerns. A means of reducing such issues is data compression; however, lossless data compression is typically ineffective when applied to floating-point data. Thus, users tend to apply a lossy data compressor, which allows for small deviations from the original data. It is essential to understand how the error from lossy compression impacts the accuracy of the data analytics. Thus, we must analyze not only the compression properties but the error as well. In this paper, we provide a statistical analysis of the error caused by ZFP compression, a state-of-the-art, lossy compression algorithm explicitly designed for floating-point data. We show that the error is indeed biased and propose simple modifications to the algorithm to neutralize the bias and further reduce the resulting error. 

\end{abstract}


\section {Introduction}
\label{sec:introduction} 
Data is now generated from everywhere. Advances in sensing technology have enabled massive data sets from experimental and observational facilities to be gathered. Additionally, due to the advances in processors, FLOPs are now considered free, enabling scientific simulations to produce petabyte-sized data sets. Not only are the storage requirements an issue, but these data sets frequently need to be transferred, causing additional bandwidth concerns. One way to combat these growing issues is by  reducing the number of bits, which would mitigate both storage and bandwidth concerns.

Compression algorithms have been a clear choice in reducing the size of data; there are two types of compression algorithms, lossless and lossy. Lossless data compression compresses the data with no degradation of the values. However, for applications involving floating-point data, lossless compression only gives a modest reduction in the bandwidth and storage costs. Instead, the scientific community has been more interested in lossy data compression (SZ \cite{sz}, ZFP \cite{zfp}), which inexactly reconstructs the floating-point values. Specifically, we consider the ZFP compressor that individually compresses and decompresses small blocks of $4^d$ values from $d$-dimensional data. Unlike many traditional compression algorithms that require global information, ZFP is ideal  for storing simulation data, since only the block containing a particular data value needs to be uncompressed, similar to standard random access arrays.
 
Typically the error caused by any lossy compression algorithm is deemed acceptable as the data gathered is \teal{already} noisy from simulation error, such as truncation, iteration and round-off error, or observational error, such as finite precision measurements and measurement noise. 
Many data and statistical analytics assume the error is i.i.d. and, in many cases, it is further preferable that the error conforms to Gaussian white noise centered around zero. However, many of the compression algorithms in use have little to no rational or theoretical backing to ensure the error is indeed not biased or \teal{spatially} correlated, and thus, many of the conclusions from the resulting statistical analysis may be incorrect. 
 
There are many applications in which the error from a lossy compression algorithm could change the underlying phenomena. 
Recent studies have investigated the effects of lossy data compression for specific applications or data sets \cite{Laney,Treib,Baker}. 
All works indicate that in order ``to preserve the integrity of the scientific simulation data, the effects of lossy data compression on the original data should, at a minimum, not be statistically distinguishable from the natural variability of the system \cite{Baker}."  
Thus, it is clear that each lossy compression algorithm should ensure that the error is not biased. 
Consequently, it is surprising that many compression algorithms tend to discuss only the compression ratio and the accuracy of the solution from a mean square error viewpoint. 
Lindstrom \cite{lindstromdistrib} discusses from an empirical standpoint the impact of the distribution and \teal{auto}correlation of the error for a variety of compressors.
Grosset et al.~\cite{forsight}  developed Forsight, an analysis framework to evaluate different data-reduction techniques for scientific analyses. 
Tao et al.~\cite{zchecker}  and Wegener \cite{Wegener} both provide tools to analyze the error distribution for a specific data set for a variety of compressors but do not provide a general theoretical rationale. Liu et al.~\cite{Liu} acknowledge the need for additional measurements of success for scientific applications, offering methods to optimize the SZ compressor for various practical constraints. Additionally, Krasowsk et al.~\cite{Krasowska} statistically analyze how the correlation structure of the data influences the compressibility of the compressor and offers methods to predict compression performance.  
 
  While recent works have provided empirical studies of ZFP and other lossy compression algorithms on real-world data sets~\cite{zchecker,forsight,Wegener,lindstromdistrib,Laney, Treib, Baker, Hammerling}, Diffenderfer \cite{errorzfp} establishes the first closed-form expression for bounds on the error introduced by ZFP. In this paper, we extend the work from Diffenderfer \cite{errorzfp} to establish the first statistical analysis of the error caused by ZFP. Using concepts from \cite{errorzfp}, we provide \teal{a closed-form expression of the ZFP error distribution and analysis of its statistical properties}. Theoretically and \teal{empirically}, we show that the error is indeed biased and propose simple modifications to the \teal{compression} algorithm to neutralize the bias. \teal{These modifications reduce the magnitude of the resulting error.}
  	
The following outlines the remaining paper: \Cref{sec:background} provides a summary of the required definitions and notations from Diffenderfer \cite{errorzfp}. \Cref{sec:algorithm} walks through the eight compression steps of ZFP, detailing operators for each step using the definitions provided in \Cref{sec:background}. \Cref{section:biastools} analyzes the expected error caused by the ZFP \teal{compression} operators and \Cref{sec:totalcompressionerror} analyzes the bias for the \teal{composite} operator.
\Cref{sec:results} presents two numerical tests to validate our theoretical analysis  and \Cref{sec:correction} presents {simple modifications} to nullify the existing bias.
\Cref{sec:distribution} further compares theoretical and observed error distributions.
Finally, \Cref{sec:conclusion} summarizes our findings and details possible future analysis.

\section{Preliminaries: Definitions, Notation, and Theorems}
\label{sec:background}

ZFP was first introduced in \cite{zfp}, but since then, it has been further modified, details of which are documented in \cite{zfp-doc} and \cite{errorzfp}. Using notation from \cite{errorzfp}, we quickly provide the notation and preliminary theorems that are necessary for this paper. For clarity, see \cite{errorzfp} for more details. 
\red{Additionally, notation is provided in \cref{table:notation} for reference.}

\begin{table}[h!]
	\caption{Notation Table} 
	\begin{center}
		\begin{adjustbox}{width=1\textwidth}
	\begin{tabular}{|c l c|}
				\hline 
				Symbol & Description & Location \\
				\hline 
				\hline
				$\cI$ & active bit set & \S \ref{sec:background}  \\
				\teal{$\xi$} & \teal{sign bit of a signed binary representation such that $\xi \in \mathbb{B}$} & \S \ref{sec:background} \\
				$\cB^n$, $\cN^n$ & infinite binary and negabinary vector space, respectively & \S \ref{sec:background} \\
				$\cB^n_k$ & subset of $\cB^n$ and $\cN_k^n$ with finite active bit set, respectively & \S \ref{sec:background} \\
				$f_\cB$, $f_\cB^{-1}$, $f_\cN$, $f_\cN^{-1}$ & bijective maps from $\cB \rightarrow \mathbb{R}$, $\mathbb{R} \rightarrow \cB$, $\cN \rightarrow \mathbb{R}$ and $\mathbb{R} \rightarrow \cN$, respectively & \Cref{eqn:fN}\\
				$F_\cB$,$F_\cB^{-1}$, $F_{\cN}$,$F_\cN^{-1}$  & bijective maps from $\cB^n \rightarrow \mathbb{R}^n$, $\mathbb{R}^n \rightarrow \cB^n$, $\cN^n \rightarrow \mathbb{R}^n$ and $\mathbb{R}^n \rightarrow \cN^n$, respectively & \S \ref{sec:background}\\
					\teal{$t_\cS$, $T_\cS$} & \teal{truncation operator on $\cA \in \{\cB, \cN\}$ and $\cA^n$, respectively}  &   \cref{DefCOperators} \\
								\teal{$\ell := e_{max}(F_\cB^{-1}(\bx))- q+1 $ }& \teal{offset of the maximum exponent of $\bx$ in its block-floating-point representation and the precision $q$} & \cref{DefCOperators} \\
				\teal{$s_{\ell}$, $S_{\ell}$} & \teal{shift operator on $\cA \in \{\cB, \cN\}$ and $\cA^n$, respectively} & \cref{DefCOperators} \\
				$e_{min}$, $e_{max}$ & \teal{minimum and maximum exponents, applicable to both blocks and scalars} &   \cref{def:emaxemin}\\
				\hline
				$d$ & dimension of the input data & \S \ref{Step1Sec} \\
				$k$ & the number of IEEE mantissa bits, including the leading one-bit& \S \ref{Step2Sec}\\
				$q$ & the number of consecutive bits used to represent an element in the block-floating point {representation} &  \S \ref{Step2Sec} \\
				$\beta$ & number of bit planes kept in Step 8 &  \S  \ref{Step8Sec}  \\
				{$\round $} & rounding operator for two's complement representation & \S \ref{sec:step3} \\
				$L$, $L_d$ & one and $d$-dimension forward decorrelating linear transform &   \S \ref{sec:step3} \\
				$\tilde{L}$, $\tilde{L}_d$ & {integer} arithmetic approximation of $L$ and $L_d$ &  \S \ref{sec:step3}  \\
				$L_d^{-1} $,  $\tilde{L}^{-1} _d$ &  $d$-dimension backward decorrelating linear transform and the {integer} arithmetic approximation of $L^{-1}$&  \S \ref{sec:step3} \\
				$C_{\teal{i}}$, $\tilde{C_{\teal{i}}}$ & lossless/lossy operator for Step $\teal{i}$ of ZFP compression & \S \ref{sec:algorithm} \\
				$D_{\teal{i}}$, $\tilde{D_{\teal{i}}}$ & lossless/lossy operator for Step $\teal{i}$ of ZFP decompression & \S \ref{sec:algorithm} \\
				\hline
				\teal{$p$} & \teal{a general precision; usually $p \in \{k,q\}$ depending on the compression step }&  \S \ref{section:biastools} \\
				${\mathbb{E}[\Gamma]}$ & expected value  \teal{associated with the} distribution $\Gamma$ & \S\ref{sec:truncation}\\
				$\eta$ & starting index of the discarded bits from a truncation operator & \S \ref{sec:truncation}\\
				$\forward_1$ & error distribution from the \teal{$1$-dimensional} forward lossy transform operator & \cref{lemma:expectedtransform1d} \\
				$\forward_{\teal{d}}$ & error distribution from the \teal{$d$-dimensional} forward lossy transform operator& \S \ref{sec:lossytransform}\\
				${\bf E}_{\teal{d}} $ & {the expected value \teal{associated with} $\forward_{\teal{d}}$} & \Cref{eqn:ed} \\ 
				\hline
				\teal{$\Delta$} & \teal{quantization step introduced by truncating a negabinary number} & \S \ref{sec:totalcompressionerror}\\
				\hline
				\teal{$\Psi_{\Delta}$} & \teal{Bias correction operator for a negabinary vector} & \S \ref{sec:correction} \\
				\hline
			\end{tabular}
		\end{adjustbox}
	\end{center}
	\label{table:notation}
\end{table}

First, we define the necessary vector spaces used in the analysis. The \emph{infinite bit vector space} was introduced in \cite{errorzfp} to express each step of the ZFP compression algorithm as an operator on the binary or negabinary  \cite{Knuth}  representations. The negabinary representation utilizes negative two as base such that positive and negative numbers are represented without a designated sign bit.  Each element in the vector space is an infinite sequence of zeros and ones that is restricted, such that each real number has a unique representation; see Section 3.1 in \cite{errorzfp} for specific details. 
Accordingly, let $\mathbb{B} = \{ 0, 1 \}$ and define  
\blue{$
\mathcal{C} := \left\{ \{ c_i \}_{i \in \mathbb{Z}} : c_i \in \mathbb{B} \ \text{for all} \ i \in \mathbb{Z} \right\} \label{eqn:infinitebitvector}.
$}
For $c \in \mathcal{C}$, we define the \emph{active bit set of $c$} by $\mathcal{I} (c) := \{ i \in \mathbb{Z} : c_i = 1 \}.$ 
 Given $x \in \mathbb{R}$, there exist $c, d \in \mathcal{C}$ and $\teal{\xi} \in \mathbb{B}$ such that $x$ can be represented in signed binary and negabinary as 
\blue{\begin{align*}
{\text{Signed Binary: }} x = (-1)^{\teal{\xi}} \sum_{i \in \mathbb{Z}} c_i 2^i \ \ \ \text{ and}  \ \ \ {\text{ Negabinary: }} x = \sum_{i \in \mathbb{Z}} d_i (-2)^i. \label{CrepX}
\end{align*}}%
\noindent
The infinite bit vector spaces for signed binary and negabinary representations, denoted by $\mathcal{B}$ and $\mathcal{N}$, are formed by placing certain restrictions to ensure uniqueness on the choice of $c$, $d$, and $\teal{\xi}$. 
In particular, we define $\mathcal{N}\subset \mathcal{C}$ and $\mathcal{B} \subset \{ (\teal{\xi}, a) \in  \mathbb{B} \times \mathcal{C} \}$, where $\teal{\xi}$ represents the sign bit and $a\in \mathcal{C}$ represents the unsigned infinite bit vector. 
\blue{The signed binary representation uses a sign-magnitude approach, where the sign of the number is determined by the sign bit $\teal{\xi}$. 
This is similar to the representation used in IEEE 754 floating-point numbers, where the sign bit indicates the sign of the number, and the magnitude is represented separately.}
\teal{See \cite{errorzfp} for details on the uniqueness of the infinite bit vector spaces.} 

To imitate floating-point representations, \cite{errorzfp} defines subspaces $\cB_k$ and $\cN_k$ of $\mathcal{B}$ and $\mathcal{N}$, where $k$ represents the maximum number of consecutive nonzero bits allotted for each representation, excluding the sign bit in the signed binary representation. 
\blue{We define each element $c \in \cB_k$ or $c \in \cN_k$ to have the range of the active bit such that $( \max{\cI(c)} - \min{\cI(c)}+1 )\leq k$.}
Note that there exist invertible maps that map the infinite bit vector spaces to the reals defined as $f_\cB : \mathcal{B} \to
\mathbb{R}$ and $f_\cN : \mathcal{N} \to \mathbb{R}$ by  
\blue{\begin{equation}
	f_\cB (b) = (-1)^{\teal{\xi}} \sum_{i \in \mathbb{Z}} a_i 2^i 
	\ \
	\text{and} 
	\ \
	\label{eqn:fN}
	f_\cN (d) = \sum_{i \in \mathbb{Z}} d_i (-2)^i, 
\end{equation}}%
for all $b = (\teal{\xi}, a) \in \mathcal{B}$ and $d \in \mathcal{N}$, respectively.  
\teal{Additionally, throughout the paper, we will drop the sign bit $\xi$ when referring to  $\cB$, with the understanding that the sign bit is implied.}
By our choice of $\mathcal{B}$ and $\mathcal{N}$, $f_\cB$ and $f_\cN$ are bijections and with inverses denoted by $f_\cB^{-1} : \mathbb{R} \to \mathcal{B}$ and $f_\cN^{-1} : \mathbb{R} \to \mathcal{N}$, respectively. 
\red{The operators $\oplus : \mathcal{A} \times \mathcal{A} \to \mathcal{A}$ and $\ominus : \mathcal{A} \times \mathcal{A} \to \mathcal{A}$, are defined by
\begin{align}
\alpha \oplus  \beta = f_\cA^{-1} \left( f_\cA( \alpha ) + f_\cA(\beta) \right) \ \ \ \ \ \ \text{and} \ \ \ \ \ \ \alpha \ominus \beta = f_\cA^{-1} \left( f_\cA(\alpha) -  f_\cA(\beta) \right) \label{oplustimesdef}
\end{align}
for all $\alpha, \beta \in \mathcal{A}$ for $\cA \in \{\cB, \cN\}$. To simplify the notation, we will use $+$ and $-$ when referring to addition and subtraction in the infinite bit vector space.}

The maps $f_\cB$ and $f_\cN$ can be generalized to vector-valued functions by defining $F_\cB : \mathcal{B}^n \to \mathbb{R}^n$ and $F_\cN : \mathcal{N}^n \to \mathbb{R}^n$ as $F_\cB(\bm{c}) = [f_\cB({\bc}_1), \ldots, f_\cB({\bc}_n)]^t$ and $F_\cN(\bm{d}) =[\ f_\cN({\bd}_1), \cdots, f_\cN({\bd}_n)]^t$, where $\bc \in \cB^n$ and $\bd \in \cN^n$, respectively. Note that $F_\cB$ and $F_\cN$ are invertible with inverses $F_\cB^{-1}$ and $F_\cN^{-1}$.

Additionally, we define truncation and shift operators on $\mathcal{A} \in \{\mathcal{B}, \mathcal{N}\}$ that are necessary for the analysis to imitate floating-point representations with a finite number of nonzero bits.
\begin{definition}\label{DefCOperators}
	Let $\mathcal{S} \subseteq \mathbb{Z}$. The \emph{truncation operator},
	$t_{\mathcal{S}} : \mathcal{A} \to \mathcal{A}$,  is defined by  
	\begin{align*}
	t_{\mathcal{S}} (c)_i = \left\{ \begin{array}{ccc} 
	c_i &: & i \in \mathcal{S} \\
	0 &: & i \not\in \mathcal{S} \\
	\end{array} \right., \ \ \ \text{for all} \ c \in \mathcal{A} \ \text{and all} \ i \in \mathbb{Z}.
	\end{align*}
	Let $\ell \in \mathbb{Z}$. The \emph{shift operator}, $s_{\ell} :
	\mathcal{A} \to \mathcal{A}$, is defined by 
	$
	s_{\ell} (c)_i = c_{i + \ell}, \ \text{for all} \ c \in \teal{\mathcal{A}} \ \text{and all} \ i \in \mathbb{Z}.
$
\end{definition}
\noindent From these definitions, it follows that $t_{\mathcal{S}}$ is a
nonlinear operator and $s_{\ell}$ is a linear operator.
\teal{Note that when $\ell$ is positive, the shift operator results in a right shift, and a left shift if $\ell$ is negative.}
 These operators
can be extended to operators on $\mathcal{A}^n$ 
by defining
$T_{\mathcal{S}} : \mathcal{A}^n \to \mathcal{A}^n$ and $S_{\ell} : \mathcal{A}^n \to \mathcal{A}^n$ by applying the respective operator componentwise. Then we can define the maximum(minimum) exponent as the maximum(minimum) nonzero index of the respective infinite bit vector space. 
\begin{definition}
	\label{def:emaxemin}
	Let $\bm{x} \in \mathbb{R}^n$. The \emph{maximum exponent of $\bm{x} $, such that $\bm{x} \neq\bfz$,
		with respect to $\mathcal{A}\in \{\mathcal{B}, \mathcal{N}\}$} is 
	\begin{align*}
	e_{max, \mathcal{A}}(\bm{x}) = \max_{1 \leq i \leq n} \max_{j} \left\{ j
	\in \mathcal{I} \left( f_{\mathcal{A}}^{-1} \left(\bm{x}_i \right)
	\right) \right\} \label{eqn:emax}, 
	\end{align*}
	and \emph{minimum exponent of $\bm{x}$ with respect to $\mathcal{A}\in \{\mathcal{B}, \mathcal{N}\}$} is
		\begin{align*}
	e_{min, \mathcal{A}}(\bm{x}) = \min_{1 \leq i \leq n} \min_{j} \left\{ j \in \mathcal{I} \left( f_{\mathcal{A}}^{-1} \left(\bm{x}_i \right) \right) \right\}. \label{eqn:emin}
	\end{align*}
	\end{definition}
When it is clear from context which space, $\cB$ or $\cN$, the vector $\bm{x}$ will be represented in, we will simply write $e_{max}$ or  $e_{min}$. 
	Using the tools defined above, we will now describe each of the eight (de)compression steps and define the corresponding operator as given by \cite{errorzfp} to accurately describe the compression error.

\section{ZFP: The Algorithm} 
\label{sec:algorithm}
Our approach to analyze the bias is to utilize operators for each step of the algorithm to determine the expected value of the error caused by ZFP \blue{at each grid point}. Though the compression operator is the source of the compression error, it is the decompression operator that maps the compressed representation, and thus the error, back to the original space. The error is acceptable for many data analysis tasks as long as the error is \teal{bounded and} centered around zero, and we will show that in the current form, the ZFP decompression operator results in the expected value of the error of the transform coefficients to be biased and provide modifications in \Cref{sec:correction} to mitigate the bias. 

ZFP is comprised of eight (de)compression steps. 
We outline the ZFP compression algorithm as documented in \cite{zfp-doc} and define a lossless and lossy operator determined by \cite{errorzfp} for each step. 
Our discussion focuses on Steps 2, 3, and 8, as these steps are the only sources of error. 
The magnitude of the error caused by Steps 2 and 3 can be shown to be of the order of machine precision. 
While Step 8 is the main source of error, the error is mapped back to the original space through a combination of Steps 5 and 3, \teal{resulting in a spatially dependent error pattern that causes errors to be autocorrelated.}
Once we have defined the operators, we discuss the expected value of the error caused by each step.  
However, as the error caused by each step is dependent on the previous steps, we attempt to compose the operators and the resulting error to estimate the expected value of the composed error accurately.

\subsection{Step 1} 
\label{Step1Sec}
\blue{The first step of ZFP takes a \(d\)-dimensional array and partitions it into smaller arrays of \(4^d\) elements each, called \emph{blocks}.}
{This idea was mainly derived from the motivation for random access; however, similar to the compression of 2-d image data techniques, data that tends to be smooth within a block should be relatively easy to compress.}  
If the $d$-dimensional array cannot be partitioned exactly into blocks, then the boundary of the array is padded until an exact partition is possible. Following this initial partitioning step, the remaining steps are performed on each block independently. 
Note that as Step 1 is lossless, it is omitted from the remaining analysis. 

\subsection{Step 2}
\label{Step2Sec}
Step 2 takes the floating-point values from each block, denoted as $\bx \in \cB_k^{4^d}$, and converts them into a block floating-point representation \cite{MitrablockRoundingError} using a common exponent. 
\blue{Each block, consisting of $4^d$ values, uses the maximum exponent from within the block as the common exponent.}
Each value is then shifted and rounded to a two's complement signed integer. 
Let $q$ denote the number of bits used to represent the significand bits for block floating-point representation; this means that the integer significand of each element in the block lies within the interval $[ 1-2^q,2^q-1]$. 
An error can occur in Step 2 if the exponent range within the block is greater than what can be accommodated {by the significand of the block floating-point representation}. The truncation operator, defined by \cref{DefCOperators}, is used to truncate the least significant bits. A lossless operator, in which the bits are not truncated, will also be defined. The lossy and lossless operators for Step 2 are then defined by 
the maps $\tilde{C}_2, C_2:\mathbb{R}^{4^d} \rightarrow \cB^{4^d}$, respectively, where 
$
\tilde{C}_2 (\bm{x}) := T_{\mathcal{S}} \: S_{\ell} \: F_{\cB}^{-1}(\bm{x}) \text{ and }  C_2(\bx ) := S_{\ell}\: F_{\cB}^{-1} (\bx), 
\text{ for all } \bm{x} \in \mathbb{R}^{4^d}, 
$
where $\mathcal{S} = \{ i \in \mathbb{Z} : i \geq 0\}$ for \cref{DefCOperators} and  $\ell = e_{max}(F_\cB^{-1}(\bx))- q+1$. Note that the lossless
operator, $C_2$, is defined by removing all noninvertible maps, i.e., the truncation operator \teal{$T_{\mathcal{S}}$}.
The decompression operator for Step 2 converts the block floating-point representation back to its original floating-point representation that is representable in $\cB_k$ for $k \in \mathbb{N}$. In IEEE, the $q\in \{30,62\}$ consecutive bits for the block floating-point representation must be converted back to $k \in \{24,53\}$ bits with its respective exponent information. The lossy and lossless decompression operators for Step 2 are then defined by the maps $\tilde{D}_2, D_2: {\cB}^{4^d}\rightarrow \mathbb{R}^{4^d}$, where 
$
	\tilde{D}_2(\ba) := F_{\cB} \: S_{-\ell} \: \mathrm{fl}_k (\ba) \ \text{and} \ D_2(\ba) := F_{\cB} \: S_{-\ell}(\ba), \ \text{for all} \ \bm{a} \in	\mathcal{{\cB}}^{4^d},
$
	where $\mathrm{fl}_k (\ba)_i  = t_{\cR_{ik}} (\ba_i)$ with $\cR_{ik} = \{ j \in \mathbb{Z}: j > e_{max, \cB} (\bm{a}_i) - k \}$, for all $1 \leq i \leq 4^d$. The $\mathrm{fl}_k(\cdot)$ operator converts each component of $\bm{a}$ to a floating-point representation with $k$ bits to represent the significand. Note that the $\mathrm{fl}_k(\cdot)$ operator depends on the IEEE rounding mode.
\teal{Also,} for certain choices of $m\in \mathbb{Z}$, the constant term $2^{1-m}$ will regularly occur in the discussion of this paper. Hence, we will let $\epsilon_m := 2^{1-m}$ \red{for} any $m \in \mathbb{Z}$. For
example, machine epsilon \cite{higham2002accuracy} is defined as $\epsilon_k = 2^{1-k}$ for precision $k$.
Assuming \red{the} decompression operator is lossless,\footnote{{As shown in \cite{errorzfp} the decompression operator can result in an additional error if the index of the leading bit of the error term from Step 8 is less than $q-k \in \{6,9\}$, for single and double precision. 
It should be noted that typical uses of ZFP will result in a lossless decompression step.}} it can be shown using Prop. 4.1 from \cite{errorzfp} that the relative error caused by this step is bounded by machine precision
$
	\|D_2\tilde{C_2}(\bx) - D_2C_2(\bx)\|_\infty \leq \epsilon_q \|\bx\|_\infty  \leq \epsilon_k \|\bx\|_\infty,
	$ as $q \geq k$.
  
\subsection{Step 3}
\label{sec:step3}
The integers from Step 2 are then decorrelated using a custom, high-speed, near orthogonal transform, $L\in \mathbb{R}^{4 \times 4}$, that is similar to the discrete cosine transform. In $d$-dimensions, the transform operator is applied to each dimension separately and can be
represented as a Kronecker product. 
Then, the total forward transform operator for ZFP is defined as 
$
\cL_d = \underbrace{\cL\otimes \cL \otimes \cdots \otimes \cL}_\text{$(d - 1)$-products},
$
where $\cL$ and $L^{-1}$ are defined by
\begin{align*}
\cL = \frac{1}{16} \begin{bmatrix}
\begin{array}{rrrr}
4 & 4 & 4 & 4 \\
5 & 1 & -1 & -5 \\
-4 & 4 & 4 & -4 \\
-2 & 6 & -6 & 2
\end{array}
\end{bmatrix} \quad \text{and} \quad \cL^{-1}= \frac{1}{4} \begin{bmatrix}
\begin{array}{rrrr}
4 & 6 & -4 & -1 \\
4 & 2 & 4 & 5 \\
4 & -2 & 4 & -5 \\
4 & -6 & -4 & 1
\end{array}
\end{bmatrix}. \label{eqn:T}
\end{align*}
Note that $\mathcal{B}_q$ is not closed under addition and multiplication, and therefore, operations within this space may result in round-off error. We define $\tilde{L}_d$ as the lossy operator used in the ZFP implementation. The lossy and lossless compression operator for Step 3 are then defined as $\tilde{C}_3, C_3:
\cB^{4^d} \rightarrow \cB^{4^d}$, where
$\label{c3}
\tilde{C}_3(\ba) = F_\cB^{-1} \: \tilde{\cL}_d \: F_\cB(\ba) \  \text{ and } \ C_3(\ba) = F_\cB^{-1} \: \cL_d \: F_\cB(\ba), \text{ for all } \bm{a} \in \mathcal{B}^{4^d}.
$
Similarly, the lossless and lossy 
decompression operators are defined by $\tilde{D}_3$, $D_3: \cB^{4^d} \rightarrow \cB^{4^d}$, where 
$  
\tilde{D}_3(\ba) = F_\cB^{-1} \: \tilde{\cL}_d^{-1} \: F_\cB(\ba) \ \text{and} \ D_3(\ba) = F_\cB^{-1} \: \cL_d^{-1} \: F_\cB(\ba), \text{ for all } \bm{a} \in \mathcal{B}^{4^d},  
$
where $\tilde{\cL}_d^{-1}$ is an approximation of $\cL_d^{-1}$. To define the exact lossy operators, we first define a rounding operator $\round(\cdot)$ that rounds a right bit shift toward negative infinity as $\round : \mathcal{B} \to \mathcal{B}$ by  
	\begin{align}\label{eq:rightbitround}
	\round (a) := \begin{cases}
	t_\cS \: s_{1}(a) & :  f_{\mathcal{B}} (a) \geq 0, \\
	t_\cS \: s_{1} \left( a - 1_\cB \right) & : f_{\mathcal{B}} (a) < 0
	\end{cases},
	\end{align}
with $\mathcal{S} = \{ i \in \mathbb{Z} : i \geq 0\}$.  Note that the rounding operator applies only to infinite bit vectors that represent an integer, rounding a division by two to another respective integer. The exact lossy operators, denoted $\tilde{\cL}$ and $\tilde{\cL}^{-1}$, used in ZFP are then outlined in
\cref{table:actionTlossy}.
See Section 4.5 in \cite{errorzfp} for details.{As in \cite{errorzfp}, we assume that \red{the} backward transform operator is lossless},\footnote{{If we assume that at least $2d$ bit planes are discarded in Step 8, the resulting backward transform operator results in a linear operator and we can assume $\tilde{D}_3 = D_3$. Note that the first two steps of the backwards transform operator may result in round-off error; however, if the remaining compression steps remain lossless then each step of the backward transform, in bit arithmetic, undoes the associated step of the forward transform. Additionally, if at least $2d$ bit planes are discarded at Step 8, then the first two steps of the backwards transform do not introduce additional error. If between $1$ and $2d - 1$ bit planes are discarded, additional error may occur in the decompression step. However, this is an uninteresting case as ZFP results in a low compression ratio if only between $1$ and $2d-1$ bit planes are discarded. Thus, the remainder of the paper assumes at least $2d$ bit planes are discarded, simplifying the analysis. See \cite{errorzfp} Appendix B for details.}} i.e., $ \tilde{D}_3 = D_3$. 
 \teal{The backward transform operator, as denoted in Table 2, is considered reversible in the ZFP implementation due to both the guard bit, which prevents overflow from the left bit shifts, and the rounding operations \( \tilde{s}(\ba_i) \), which are the exact reversals of the last two steps in the forward transform operator. However, reversibility does not necessarily imply that the potentially lossy steps provided in  Table 2 are mathematically equivalent to the linear lossless operator.
 	To ensure linearity, the coefficients must satisfy the precondition of being integer multiples of four. This ensures that the rounding operations $ \tilde{s}(\ba_i) $ and the backward transform operator are mathematically equivalent to the linear lossless operator. Zeroing at least \(2d\) bit-planes in  Step~8 (see below) is one way to enforce this precondition ensuring the mathematical equivalence of \( \tilde{D}_3 = D_3\).}
 
\begin{table}[h]
	\begin{adjustbox}{width=\textwidth} 
		\begin{tabular}{| L  L   L || L   L  L | }
			\hline
			\multicolumn{3}{| c ||}{$\tilde{\cL}$} & \multicolumn{3}{ c |}{$\tilde{\cL}^{-1}$} \\ \hline 
			\ba_1 \leftarrow \ba_1+ \ba_4& \ba_1 \leftarrow (\ba_1) &\ba_4 \leftarrow \ba_4 - \ba_1 &\ba_2 \leftarrow \ba_2 +  \teal{\tilde{s}}(\ba_4)&\ba_4 \leftarrow \ba_4-  \teal{\tilde{s}}(\ba_2) & \\
			\ba_3 \leftarrow \ba_3+\ba_2 & \ba_3 \leftarrow  \teal{\tilde{s}}(\ba_3)&\ba_2 \leftarrow \ba_2- \ba_3 &\ba_2 \leftarrow \ba_2+ \ba_4 &\ba_4 \leftarrow s_{-1} (\ba_4) &\ba_4 \leftarrow \ba_4- \ba_2 \\
			\ba_1 \leftarrow \ba_1+\ba_3 & \ba_1 \leftarrow  \teal{\tilde{s}}(\ba_1)&\ba_3 \leftarrow \ba_3- \ba_1 &\ba_3 \leftarrow \ba_3 + \ba_1&\ba_1 \leftarrow s_{-1}(\ba_1) &\ba_1 \leftarrow \ba_1- \ba_3  \\
			\ba_4 \leftarrow \ba_4+\ba_2 & \ba_4 \leftarrow  \teal{\tilde{s}}(\ba_4) &\ba_2 \leftarrow \ba_2 - \ba_4 &\ba_2 \leftarrow \ba_2 + \ba_3 &\ba_3 \leftarrow s_{-1} (\ba_3) &\ba_3 \leftarrow \ba_3- \ba_2  \\
			\ba_4 \leftarrow \ba_4 +  \teal{\tilde{s}}(\ba_2) & \ba_2  \leftarrow \ba_2 -  \teal{\tilde{s}}(\ba_4) & &\ba_4 \leftarrow \ba_4 + \ba_1   &\ba_1 \leftarrow s_{-1} (\ba_1) & \ba_1\leftarrow \ba_1 - \ba_4 \\ \hline
		\end{tabular}
	\end{adjustbox}
	\caption{\label{table:actionTlossy} Bit arithmetic steps for the lossy implementation of ZFP's forward (left) and backward (right) transform.
	} 
\end{table}
From Lemma 4.4 in \cite{errorzfp}, we can show that the relative error caused by the compression operator from Step 3 is within \teal{a small constant multiple} of machine \teal{epsilon}, i.e., 
\[
\|D_3\tilde{C_3}(\bx) - D_3C_3(\bx)\|_\infty \leq \left(\frac{15}{4}\right)^d k_L \epsilon_q \|\bx\|_\infty,
\]
where \(k_L = \frac{7}{4}(2^d - 1)\). Depending on the values of \(k\) and \(q\), the right-hand side of the inequality bounding the error can be expressed as a constant multiple of machine epsilon, \(\epsilon_k\). For example, in double precision, with \(q = 62\) and \(k = 53\), we have
\teal{\(
\|D_3\tilde{C_3}(\bx) - D_3C_3(\bx)\|_\infty < 2\epsilon_k \|\bx\|_\infty,
\)
for \(d \leq 3\), and
\(
\|D_3\tilde{C_3}(\bx) - D_3C_3(\bx)\|_\infty < 11\epsilon_k \|\bx\|_\infty,
\)
for \(d = 4\).}

\subsection{Step 4} The coefficient magnitude from Step 3 tends to correlate {with the index of the elements}. Step 4 applies an invertible deterministic permutation on the \teal{coefficients}. 
The permutation roughly places the transform coefficients in order of decreasing magnitude, facilitating compression as the encoder in \teal{Step 7} tests groups of bits from consecutive coefficients and \teal{small coefficients have leading zeros} {in the binary representation.} 
 This step is not considered for the analysis since no error occurs as the (de)compression step only applies a permutation. 

\subsection{Step 5} Step 5 converts the two's complement signed integers (the standard integer representation) into their negabinary representation, first introduced in \cite{Knuth} and defined in \Cref{CrepX}, {as the negabinary representation uses leading zeros when representing small values.}   
As we are representing values using a signed binary representation instead of a two's complement representation for our analysis, we need to convert each signed binary representation to its negabinary representation. Define the operator $C_5 : \mathcal{B}^{4^d} \to \mathcal{N}^{4^d}$ and $D_5 : \mathcal{N}^{4^d} \to \mathcal{B}^{4^d}$  by
$
C_5 (\bm{a}) := F_{\cN}^{-1} \: F_\cB (\ba),$ for all $\bm{a} \in \mathcal{B}^{4^d}$
and 
 $D_5 (\bm{a}) := F_{\cB}^{-1} \: F_\cN(\ba),$ for all $ \bm{a} \in \mathcal{N}^{4^d}.
$
{In the ZFP implementation, Step~5 is lossless.\footnote{As the ZFP implementation uses a guard bit for the two's complement representation to safeguard against overflow when applying the forward transform, Step 5 is lossless as there is a one to one mapping between the signed binary representation and the negabinary representation; see Section 4.5 in \cite{errorzfp}.} }
\subsection{Step 6}
In Step 6, the bit vectors are reordered by their bit index instead of by coefficient, {allowing the leading zeros of the negabinary representation to be grouped together for small valued coefficients.} When using the infinite bit vector space, this step corresponds to a transpose of a binary matrix and, as such, does not result in an error. Thus, for simplicity, {as this step does not result in altering the representation of any element in the block, we do not denote an operator here. }

\subsection{Step 7} In Step 7, each bit plane of $4^d$ bits are individually coded with a variable-length code that is one to one and reversible (see \cite{zfp-doc} for details). This idea exploits the property that the transform coefficients tend to have many leading zeros. As this step is lossless, it again is ignored for our analysis.

\subsection{Step 8} 
\label{Step8Sec}
The embedded coder emits one bit at a time until the stopping criterion is satisfied. Specifically, ZFP has three modes that determine the stopping criteria: either fixed rate, fixed precision, or fixed accuracy. The fixed rate mode compresses a block to a fixed number of bits, the fixed precision \teal{mode} compresses to a variable number of bits while retaining a fixed number of bit planes, and the fixed accuracy mode {encodes enough bit planes to satisfy an absolute error tolerance.} For our purposes, we investigate only the fixed precision mode. Thus, Step 8 is dependent only on one parameter, denoted $\beta \geq  0$, which represents the number of most significant bit planes to keep during Step 8, and any discarded bit plane is replaced with all-zero bits. An index set, denoted as $\mathcal{P}$, is used to define the truncation and is dependent on $\beta$. The lossy operator for Step 8 is given by $\tilde{C}_8:
\cN^{4^d} \rightarrow \cN^{4^d}$ and defined as 
$
\tilde{C}_8(\bm{d}) := T_{\mathcal{P}} (\bm{d})\text{ for all } \bm{d} \in \mathcal{N}^{4^d},
$
where $\mathcal{P} = \{ i \in \mathbb{Z} : i > q + 1 - \beta \}$, $q \in \mathbb{N}$ is the value from Step 2, and $T_\cP$ is the truncation operator with respect to set $\cP$. The lossless compression and decompression operators are then 
defined by $C_8 := I_\cN$\teal{, where $I_\cN$ denotes  the identity operator with respect to the space $\cN$}
\subsection{Defining the ZFP Compression Operator}
\label{ZFPCompOperator}
To conclude this section, we define the ZFP fixed precision compression
and decompression operators, as defined in \cite{errorzfp}.  Note $C_7$, $D_7$, and $C_8$ are here omitted from the composition, as
they were defined to be the identity operator $I_{\mathcal{N}}$. 

{\begin{definition}{(Definition 4.7 \cite{errorzfp})}
		The lossy fixed precision compression operator, $\tilde{C}: \mathbb{R}^{4^d} \to \mathcal{N}^{4^d}$, is defined by 
		\begin{align*}
		\tilde{C} (\bm{x}) = \left( \tilde{C}_{8} \circ {C}_{5} \circ {C}_{4} \circ \tilde{C}_{3} \circ \tilde{C}_{2} \right) (\bm{x}), \ \ \ \text{for all} \ \bm{x} \in \mathbb{R}^{4^d},
		\end{align*}
		where $\circ$ denotes the usual composition of operators. The lossless fixed precision compression operator, $C: \mathbb{R}^{4^d} \to \mathcal{N}^{4^d}$, is defined by 
		\begin{align*}
		C(\bm{x}) = \left(C_{5} \circ C_{4} \circ C_{3} \circ C_{2} \right) (\bm{x}), \ \ \ \text{for all} \ \bm{x} \in \mathbb{R}^{4^d}.
		\end{align*} 
		Lastly, the lossy fixed precision decompression operator, $\tilde{D}:  \mathcal{N}^{4^d} \to \mathbb{R}^{4^d}$, is defined by 
		\begin{align*}
		\tilde{D}(\bm{d}) = \left( \tilde{D}_{2} \circ {D}_{3} \circ D_{4} \circ D_{5} \right) (\bm{d}), \ \ \ \text{for all} \ \bm{d} \in \mathcal{N}^{4^d},
		\end{align*} 
		and the lossless fixed precision decompression operator ${D} : \mathcal{N}^{4^d} \to \mathbb{R}^{4^d}$ is defined by 
		\begin{align*}
		{D}(\bm{d}) = \left( D_{2} \circ {D}_{3} \circ D_{4} \circ D_{5} \right) (\bm{d}), \ \ \ \text{for all} \ \bm{d} \in \mathcal{N}^{4^d}.
		\end{align*} 
\end{definition}}

\section{Understanding  Bias in ZFP}
\label{section:biastools}
The goal of this section is to analyze the {expected value of the} error caused by each compression step, as well as their composition.  
A few assumptions are made to analyze the error statistically. First, we will assume that the bits after the leading \teal{one-bit} in both the signed binary and negabinary representation are uniformly random.
 This assumption is common in floating-point analysis \cite{higham2002accuracy}. {To validate this assumption for negabinary, we conducted an exploratory study in \Cref{apd:random}, which concludes that it is a reasonable assumption for bit-plane indices greater than three.}\footnote{{Due to the block-floating point transform in Step 2, there is a high probability that the inputs into the transformation have trailing zeros. This is due to the precision differences between the input data type and the block floating point representation, i.e., in the current implementation of ZFP we have $q>k$. The transformation propagates the zero bits through arithmetic operations. However, if the block has a small dynamic range, it is likely that not all the trailing zero bits will be operated on. Thus, the least significant bits have a high probability of being zero.}} Additionally, we will assume the input $\bx \in \mathbb{R}^{4^d}$ is representable in $\cB_{\teal{p}}^{4^d}$ for some precision $\teal{p}$.


From \Cref{sec:algorithm}, it can be seen that the compression algorithm is constructed by utilizing operators that act on the bit representation. Note that the only operators that introduce error are Steps 2, 3, and 8, which are comprised of either the truncation operator, $t_\cS(\cdot)$ or the lossy transform operator, $\tilde{\cL_d}(\cdot)$. Otherwise, the remaining operators either shift the index of the leading bit or change the mapping from the binary representation to the real number space,  i.e., $f_\cN(\cdot)$ or $f_\cB(\cdot)$. 

In the following, we will assume that the input distribution for each operator follows a discrete uniform distribution within the infinite bit vector space $\mathcal{C}_{\teal{p}} \teal{\in}  \{\mathcal{B}_{\teal{p}}, \mathcal{N}_{\teal{p}}\}$, with a finite precision denoted by ${\teal{p}}$. 
The location of the non-zero \teal{bit} of the infinite vector will change depending on the step of the ZFP operation. 
\teal{ \begin{definition} \label{def:dist} 
Define a discrete uniform distribution $A_{\{\cC_p,\iota\}}$ such that a random variable $a \sim A_{\{\cC_p,\iota\}}$ implies that the distribution has support $\supp(a) = \{ a : a \in \cC_p \text{ and } \cI(a) \subset \{ \iota, \iota+1, \ldots, \iota+p-1\} \}$, with precision $p$. 
\end{definition} }
Thus, if $\iota = 0$, $A_{\{\cB_3, 0 \}}$  depicts a discrete uniform distribution of integers from \teal{from negative seven to positive seven}, i.e.,  $A_{\{\cB_3, 0 \}} $ is synonymous with \teal{$ \cU(-7,7)$} in the real space. 
We will also define a vector version of \cref{def:dist}. 
\teal{\begin{definition}  \label{def:vectdist}
Define a discrete uniform vector distribution ${\bf A}^n_{\{\cC_p, \iota \}}$ such that for a vector $\ba \sim  {\bf A}^n_{\{\cC_p, \iota \}}$, each element $\ba_i$ is a random variable satisfying $\ba_i  \sim  A_{\{\cC_p, \iota \}}$ for all $i = 1, \ldots, n$. 
\end{definition} }
When dealing with a specific data set, it is possible to introduce additional assumptions regarding the input distribution. 
However, for the following analysis, the conclusion holds for symmetric distributions such as the uniform and normal distributions. 
A symmetric distribution is where the mean, median, and mode typically coincide at a single point. 
Additional relaxations of the assumptions may enable the conclusions to hold. 
However, it's important to consider certain edge cases. For example, distributions in which variable values are exclusively even integers could lead to the failure of our established findings. 
To ensure broad applicability, we will use \cref{def:dist} and  \cref{def:vectdist} in our subsequent analysis. 

\subsection{The Truncation Operator}
\label{sec:truncation}
First, we will discuss the error caused by the truncation operator, $t_{\cS}(\cdot)$. 
In the following section we will show that when the input $a \in A_{\{\cB_{\teal{p}},0\}}$  is an integer such that $\cI(a) \subset \teal{\mathbb{N}}$, the expected value of the error caused by $t_{\cS}(\cdot)$ \teal{is} centered around zero. 
However, if the leading bit is truncated by $t_{\cS}(\cdot)$, then $t_{\cS}(a) = 0_\cB$. Thus, the expected value of the error is dependent on the input distribution and the index of the leading truncated bit.  
For some vector distribution $\Gamma$, such that for \teal{a random variable} $\gamma \teal{\ \sim\ } \Gamma$ we have $\gamma \in \mathbb{R}^n$, define $\mathbb{E}(\teal{\gamma}) \in \mathbb{R}^n$ as the expected value of the  distribution element-wise. 
Note that for any operator that is applied to the distribution  $\Gamma$, the operator is applied element-wise, i.e., let $\gamma \teal{\ \sim\ } { \Gamma }$ such that  $t_{\cS}({\gamma}) =  [t_{\cS}(\gamma_1),\dots,t_{\cS}(\gamma_n)]^t$. 
\teal{Similarly, we can define the expected value for a vector of infinite bit-vector distributions using the previously defined mappings $F_\mathcal{C}(\cdot)$ and $F_\mathcal{C}^{-1}(\cdot$), where  $\mathcal{C} \teal{\in}  \{\mathcal{B}, \mathcal{N}\}$, which map between the real space and the infinite bit vector space. When it is obvious, we omit the mappings for the sake of readability.} 

\cref{lemma:truncbinary} presents the expected value of the error caused by the truncation operator for a bounded distribution comprised of elements from $\cB_{\teal{p}}$. 
With respect to Step 2, the block floating-point representation, define $\eta = {\teal{p}} - l -1 \in  \teal{\mathbb{N}}  $ as the starting index of the bits that will be discarded when the truncation operator is applied, where $\teal{p} \in \teal{\mathbb{N}}  $ is the number of allotted bits used in the signed binary representation, i.e., $\cB_{\teal{p}}$, and $l\in  \teal{\mathbb{N}}  $ is the number of bits that are kept. In other words, we define the truncation operator $t_\cS(a)$ such that $\cS = \{i \in \mathbb{Z}:   i > \eta\}$ for all $a\in \cB_{\teal{p}}$. 
See \cref{fig:truncation} for a simple example of applying the truncation operator for a 32-bit integer. Recall $e_{max,\cB}(a)$ is the index of the leading nonzero bit and determines the magnitude of element $a$.  

\begin{figure}
	\centering
	\includegraphics[width=.5\textwidth]{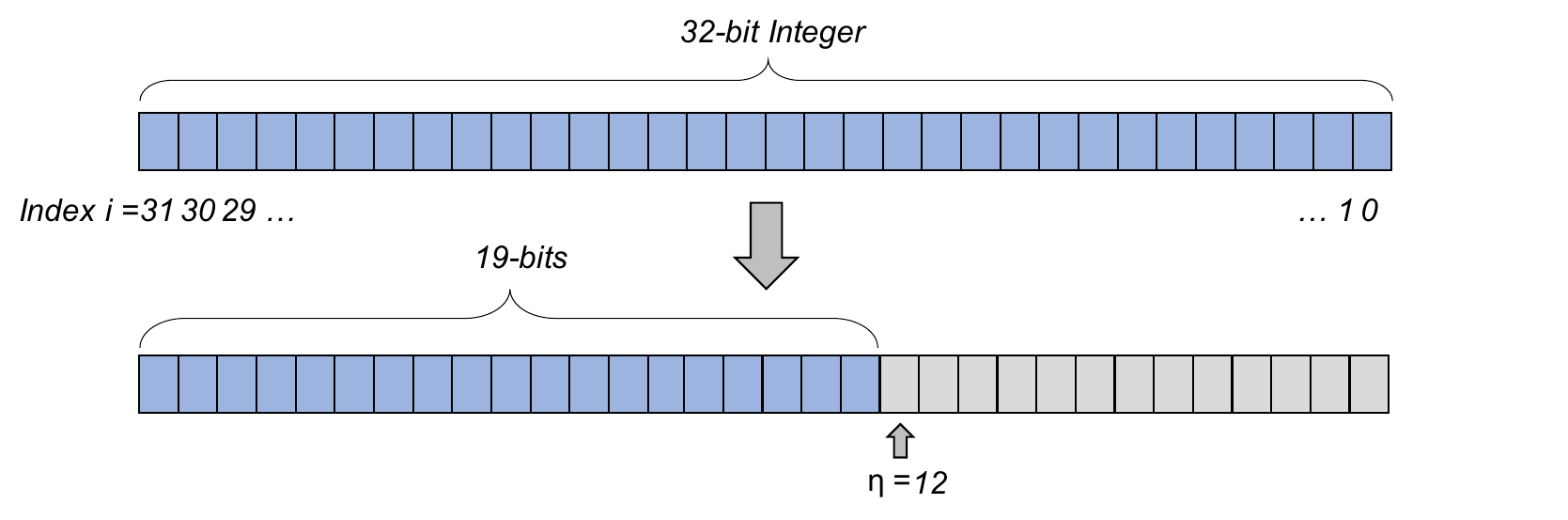}
	\caption{Applying the truncation operator $t_\cS$, where $\teal{p} = 32$, $l=19$, and $ \eta = 12$, such that $\cS =\{i \in \mathbb{Z}:   i >\eta\}.$ The truncated bits are grayed out to represent their replacement by zero bits. }
	\label{fig:truncation}
\end{figure}

\begin{lemma}\label{lemma:truncbinary} 
Assume ${\teal{p}}, l \in \teal{\mathbb{N}}$ \teal{such that $p > l$}, ensuring $\eta = {\teal{p}} - l - 1 \in \teal{\mathbb{N}}$. Define $\cS = \{i \in \mathbb{Z}: i > \eta\}$.
	Define the distribution $ A:= A_{\{\cB_{\teal{p}},0\}}$.  \teal{Let $a\teal{\ \sim\ } A$, then }
	\begin{itemize}
		\item[(i)] \teal{Assume} $e_{max,\cB}(a)>\eta $, then  $f_\cB(\teal{t_{\cS}(a)-a}) \in [1-2^{\eta+1},2^{\eta+1}-1]$ and \\$\teal{f_\cB} (\mathbb{E}[\teal{t_{\cS}(a)-a}]) = 0 $.
		\item[(ii)] \teal{Assume} $e_{max,\cB}(a) \leq \eta $, then $\mathbb{E}[\teal{t_{\cS}(a)-a}]= - \mathbb{E}[\teal{a}] $.
	\end{itemize}
\end{lemma}
\begin{proof}
	\begin{itemize}
		\item[(i)]	Let $e_{max,\cB}(a) >\eta $.  Observe that 
		\begin{align*}
		f_{\mathcal{B}} \left(\teal{t_{\cS}(a)-a}\right) = (-1)^{\teal{\xi}} \sum_{i \in \mathcal{I} \left( \teal{ t_{\cS}(a) \red{-}a } \right)} 2^i. 
		\end{align*}
		where $\teal{\xi} = \mathrm{sign}(a)$. 
Then
		\begin{align*}
		1-2^{\eta+1} \teal{=}  - \sum_{i =0}^{\eta} 2^i &\leq 	f_{\mathcal{B}} \left( t _{\mathcal{S}} (a) -a \right) \leq \sum_{i =0}^{\eta} 2^i   \teal{=}  2^{\eta+1}-1,
		\end{align*}
		implying the error is bounded by $[-2^{\eta+1}+1,2^{\eta+1}-1] $ and is distributed uniformly. Thus, $\teal{f_\cB} (\mathbb{E}[t _{\mathcal{S}} (\teal{a}) - \teal{a}])= 0$.
		
		\item[(ii)] Let $e_{max,\cB}(a) \leq \eta $. Then $ t _{\mathcal{S}} (a)  =  0_\cB$ and $\teal{t_{\cS}(a)-a} =-a $, implying   \\$\mathbb{E}[\teal{t_{\cS}(a)-a} ] = - \mathbb{E}[\teal{a}] $.
		\end{itemize}
	\end{proof}

Similarly, the distribution of the error of the truncation operator is also affected by the error caused by using the negabinary representation instead of the binary representation. 
\begin{lemma}\label{lemma:truncneg}
Assume ${\teal{p}}, l \in \teal{\mathbb{N}}$ \teal{such that $p > l$}, ensuring $\eta = {\teal{p}} - l - 1 \in \teal{\mathbb{N}}$. Define $\cS = \{i \in \mathbb{Z}: i > \eta\}$.
	Define the distribution $ A:= A_{\{\cN_{\teal{p}},0\}}$.   \teal{Let $a\teal{\ \sim\ } A$, then } 
	\begin{itemize}
		\item[(i)] Assume $e_{max,\cN}(a) >\eta$. If $\eta$ is even, then $f_\cN(\teal{t_{\cS}(a)-a}) \in 2^{\eta+1} \left( -\frac{1}{3},\frac{2}{3}\right)$, such that $\teal{f_\cN} (\mathbb{E}[\teal{t_{\cS}(a)-a} ])= \frac{2^{\eta+1}}{6}$,
		 Otherwise, if $\eta$ is odd, \newline $f_\cN(\teal{t_{\cS}(a)-a}) \in 2^{\eta+1} \left( -\frac{2}{3},\frac{1}{3}\right)$, such that $\teal{f_\cN} (\mathbb{E}[\teal{t_{\cS}(a)-a}]) =-  \frac{2^{\eta+1}}{6}$.
		\item[(ii)]Assume $e_{max,\cN}(a) \leq \eta $. Then $\mathbb{E}[ \teal{t_{\cS}(a)-a} ] = -  \mathbb{E}[\teal{a} ] $.
	\end{itemize}
\end{lemma}

\begin{proof}
	\begin{itemize}
		\item[(i)]	Let  $e_{max,\cN(a)} >\eta$ for all $a \teal{\ \sim\ }  A$. Observe that 
	\begin{align*}
	f_{\mathcal{N}} \left(\teal{t_{\cS}(a)-a } \right) = - \left(  \sum_{i \in \mathcal{I} \left(\teal{ t_{\cS}(a) \red{-} a }\right)} (-2)^i \right) , 
	\end{align*}
	where $\mathcal{I} (\teal{t_{\cS}(a) \red{-} a})$ is the index set of the truncated least significant bits.  Depending on if $\eta$ is even or odd, the error is uniformly bounded in either  $2^{\eta+1} \left( -\frac{1}{3},\frac{2}{3}\right)$ or $2^{\eta+1} \left(-\frac{2}{3},\frac{1}{3}\right), $  respectively. To demonstrate, first assume $\eta$ is even, then observe that 
	\begin{align*}
	     \sum_{j = -\infty}^{\eta/2 } (-2)^{2j}   &\leq  - \- \left(  \sum_{i \in \mathcal{I} \left(\teal{ t_{\cS}(a) \red{-} a }\right)} (-2)^i \right )\leq  \sum_{j =  -\infty}^{\teal{\eta/2-1} } (-2)^{2j+1}  ,  \\
	 - \frac{1}{3} 2^{\eta +1}&\leq - \left(  \sum_{i \in \mathcal{I} \left(\teal{ t_{\cS}(a) \red{-} a }\right)} (-2)^i \right ) \leq  \frac{2}{3} 2^{\eta +1}.
	\end{align*}
	Similarly, if $\eta$ is odd, then 
	\begin{align*}
	-\frac{2}{3} 2^{\eta+1}  &\leq- \left(  \sum_{i \in \mathcal{I} \left(\teal{ t_{\cS}(a) \red{-} a }\right)} (-2)^i \right) \leq      		\frac{1}{3} 2^{\eta+1}.
	\end{align*}
	This phenomenon is due to the alternating sign in the negabinary representation, implying that the expected value is either $$\teal{f_\cN} (\mathbb{E}[ \teal{ t_{\cS}(a)-a }])\teal{\in} \{\frac{1}{6}, -\frac{1}{6}\}2^{\eta+1}$$ depending on whether the index $\eta$ is even or odd.
	\item[(i)] Let $e_{max,\cN}(a) \leq \eta $, then $ t _{\mathcal{S}} (a)  =  0_\cN$. Then, we have $ \teal{t_{\cS}(a) - a = -a} $, implying $\mathbb{E}[\teal{ t_{\cS}(a)-a}] = -  \mathbb{E}[\teal{a}] $.
\end{itemize}

\end{proof}	

From \cref{lemma:truncbinary} and \cref{lemma:truncneg}, assuming the leading \teal{one-bit} is not truncated, one can see that the error caused by the truncation operator on a signed binary representation results in the error distribution to be centered around zero. 
In contrast, the error distribution caused by truncation operator on a negabinary representation is biased based \teal{on the parity of the index of the first truncated bit-plane,} as the last compression step truncates bit-planes while in a negabinary representation.
Before the complete discussion of the ZFP compression error, we discuss the error caused by the lossy transform operator, $\tilde{\cL}_d$.

\subsection{Lossy Transform Operator} 
\label{sec:lossytransform}
Using \cref{table:actionTlossy}, we can write the action of $\tilde{L}_1$ as a composite operator of each step. 
Define $\ba = \teal{[ \ba_1,\ba_2, \ba_3, \ba_4 ]}^t  = F_\cB^{-1}(\bx) \in \cB^4$, such that $\cI(\teal{\ba}_i) \subset \{0,\cdots,q-1\}$ for all $i$, to be the representation of $\bx$ in $\mathcal{B}^4$. \teal{The active bit set is normalized so that the input represents an integer, ensuring it corresponds to the input for Step 3. }
Let $\tilde{\cL}_{\cB,1}$ and ${\cL}_{\cB,1}$ denote the action of $\tilde{\cL}_1$ and $\cL_1$ in the vector space $\mathcal{B}^4$, respectively. 
Then 
	\begin{align*}
	\tilde{L}_{\cB,1}(\ba) &= \begin{bmatrix}
	\round(\round(\teal{\ba}_1 +\teal{\ba}_4) +\round(\teal{\ba}_3 +\teal{\ba}_2))\\ \teal{\by}_2 -\round(\teal{\by}_1  +\teal{\by}_2) -\round(\round( \teal{\by}_1   +\teal{\by}_2) +\round(\teal{\by}_2 -\round(\teal{\by}_1   +\teal{\by}_2)) )\\ \round(\teal{\ba}_3 +\teal{\ba}_2) - \round(\round(\teal{\ba}_1 +\teal{\ba}_4) +\round(\teal{\ba}_3 +\teal{\ba}_2)) \\\round(\teal{\by}_1   +\teal{\by}_2) +\round(\teal{\by}_2 -\round(\teal{\by}_1   +\teal{\by}_2)) 
	\end{bmatrix},
	\end{align*}
where $\teal{\by}_1 = \teal{\ba}_4 - \round(\teal{\ba}_4+\teal{\ba}_1)$ and $\teal{\by}_2 = \teal{\ba}_2 - \round(\teal{\ba}_2+\teal{\ba}_3)$. 
For details on forming $\tilde{L}_{\cB,1}$, see Section 4.3 in \cite{errorzfp}. 
The operator $\cL_{\cB,1}$ is formed by replacing the rounding operator $\round(\cdot)$ with the shift operator, $s_1(\cdot)$. 
Note that an error \teal{may} occur because $\round(\cdot)$ is not bijective.
Using the composite operators of the lossless and lossy forward transform operators, the error between \teal{the} two operators is defined in the following lemma. 

\begin{lemma}\label{lemma:expectedtransform1d} 
	Define the distribution ${\bf A}:={\bf A}^4_{\{\cB_p,0\}}$ with precision $p$.
	Define the operator $\littleforward(\cdot): \cB_p \rightarrow  {\{-\frac{1}{2},0 \}}$ as the error caused by rounding towards  \teal{negative infinity through a right bit-shift, as described in Equation \cref{eq:rightbitround}} i.e., $\littleforward(\cdot):= \round(\cdot) -s_1(\cdot)$.  
	Under the assumption that the trailing least significant bits are uniformly random, we have $\theta_j(\cdot) \in \{-1/2 , 0\}$ with equal probability\footnote{
		Note that $\mathcal{I}(\teal{\ba_i}) \subseteq \teal{\mathbb{N}}$. 
		Suppose $f_\mathcal{B}(\teal{\ba_i})\geq 0$, then $\mathcal{I}(\round(\teal{\ba_i}))=\mathcal{I}(t_\mathcal{S}(s_1(\teal{\ba_i}))) = \mathcal{I}(s_1(\teal{\ba_i})) \setminus \{-1\}$, implying   $\littleforward_j(\teal{\ba_i}) \in  \{-\frac{1}{2},0 \}$ with equal probability \teal{under the assumption that the least significant bits of $\ba_i$ are uniformly random}. 
		Now suppose $f_\mathcal{B}(\teal{\ba_i})< 0$. 
		If  $f_\mathcal{B}(\teal{\ba_i})$ is even such that $f_\mathcal{B}(\teal{\ba_i}) = 2k$ then $f_\mathcal{B}(\round(\teal{\ba_i})) =k$. 
		If $f_\mathcal{B}(\teal{\ba_i})$ is odd such that $f_\mathcal{B}(\teal{\ba_i}) = 2k-1$, then $f_\mathcal{B}(\round(\teal{\ba_i})) = k-1$, implying  $\littleforward_j(\teal{\ba_i}) \in \{-\frac{1}{2},0 \} $ with equal probability. 
}. Then the error of the lossy forward transform operator \teal{associated with a random variable $\ba  \sim  {\bf A }$} is   
\begin{align*}\label{eqn:transformlossyerror}
\forward_1: = \left \{  \tilde{L}_{\cB,1}\ba - {L}_{\cB,1}\ba \in 
\begin{bmatrix}
\frac{1}{2}\left(\littleforward_1+\littleforward_2\right )  +\littleforward_4\\
\frac{1}{8}\left( 5\littleforward_1-\littleforward_2 \right) -\frac{5}{4}\littleforward_3-\frac{1}{2}\littleforward_5 -\littleforward_6\\
\frac{1}{2}\left(\littleforward_2-\littleforward_1\right) -\littleforward_4 \\ 
-\frac{1}{4}\left( \littleforward_1 +3 \littleforward_2\right) +\frac{1}{2}\littleforward_3+\littleforward_5
\end{bmatrix} \: | \;  \ba \teal{\ \sim \ }{{{\bf A}}} \right \}, 
\end{align*}
where each $\littleforward_j:=\littleforward_j(\cdot)$ is associated with \teal{a} unique operation from \cref{table:actionTlossy}. 
The expected value of the error is then 
\begin{align*}
 { {\bf E}_1 = \teal{F_\cB}(\mathbb{E}[\forward_1]}) = \begin{bmatrix} - \frac{1}{2}, \frac{9}{16}, \frac{1}{4}, - \frac{1}{8} \end{bmatrix}^t.
\end{align*}
\end{lemma}
\begin{proof}
Observe that each term involving $\round(\cdot)$ in \cref{eqn:transformlossyerror} can written in terms of $\littleforward_i$ and the shift operator $s_1$. 
For example, 
$
	\round(a_1+a_4) =  s_1\left(a_1+a_4\right) +\littleforward_i(a_1+a_4).
$
	Assuming each $\littleforward_i$ is defined by \teal{a} specific operation from \cref{table:actionTlossy}, and the fact that $ {L}_{\cB,1}$ is formed by replacing $\round(\cdot)$ with the shift operator $s_1(\cdot)$, the proof follows. 
\end{proof}
\cref{sec:demoEd} demonstrates the validity of the estimated error predicted in \cref{lemma:expectedtransform1d} with respect to the the expected value caused by the lossy transform operator.
Generalizing \cref{lemma:expectedtransform1d} to higher dimensions, \teal{$d$}, \cref{thm:expectedtransform}  presents the expected value of error $\mathbb{E}\teal{[}\forward_d\teal{]}$ for input distribution ${\bf A}:={\bf A}^d_{\{\cB_p,0\}}$, the error between $\tilde{L}_d$  and $L_d$.
 As we traverse the $x$ dimension before the $y$ dimension and so forth and the error is nonlinear, the resulting error matrix will not be symmetric, as can be seen in \cref{thm:expectedtransform}.


\teal{\begin{theorem}\label{thm:expectedtransform} 
	Let ${\bf A}:={\bf A}^{4^d}_{\{\cB_p,0\}}$ be a distribution over $\cB_p$ with precision $p$.  
	Define the operator $\littleforward(\cdot): \cB_p \rightarrow \{-\frac{1}{2},0 \}$ as the error caused by rounding towards negative infinity through a right bit-shift, as described in Equation \cref{eq:rightbitround}, i.e., $\littleforward(\cdot):= \round(\cdot) -s_1(\cdot)$.  
	Assume that the rounding error $\littleforward_j(\cdot)$ is uniformly distributed in $\{-\frac{1}{2}, 0\}$ and independent for all components $j$, as established in \cref{lemma:expectedtransform1d}. 
	Under these assumptions, the expected value of the error caused by the lossy forward transform for $d$ dimensions is given by:
	\begin{align} \label{eqn:ed}
		{\bf E}_d = \mathbb{E}[\forward_d] = \mathrm{vec} \left( [ {\bf E}_{d-1}, {\bf E}_{d-1}, {\bf E}_{d-1}, {\bf E}_{d-1} ] + L_{d-1} \left[ \underbrace{{\bf E}_1, {\bf E}_1, \ldots, {\bf E}_1}_{4^{d-1}} \right]^t   \right),
	\end{align}
	where $\forward_d = \{  \tilde{L}_d\ba - L_d\ba  \; | \; \ba \sim {\bf A} \}$, and ${\bf E}_{d-1}$ is the expected error between the lossless and lossy transform operators for $d-1$ dimensions. 
\end{theorem}
\begin{proof}
The base case for this recursion, ${\bf E}_1$, is given in \cref{lemma:expectedtransform1d}.
	For \(d = 2\), let \(X = \mathrm{vec}^{-1}(\ba)\) be the inverse vectorization of \(\ba\). 
	Then, the lossless and lossy transforms are given by \(L_2\ba = \mathrm{vec}\left[L(LX^t)^t\right]\) and \(\tilde{L}_2\ba = \mathrm{vec}\left[\tilde{L}(\tilde{L}X^t)^t\right]\), respectively. Then the error is 
	\begin{align}
		\mathrm{vec}^{-1}\left[ \tilde{L}_2\ba -  L_2 \ba  \right]  &=   \tilde{L}(\tilde{L}X^t)^t - L(LX^t)^t , \\
		&=  \tilde{L}(\tilde{L}X^t)^t  - L(\tilde{L}X^t)^t  + L[(\tilde{L}X^t)- (LX^t) ]^t.
	\end{align}
	Taking the expected value of each term, the first term, \( \tilde{L}(\tilde{L}X^t)^t - L(\tilde{L}X^t)^t  \), corresponds to the error introduced by \(\tilde{L}\). Note that the nonlinear operator $\tilde{L}$ ensures that the output remains in the same space, satisfying $\tilde{L}\ba \in \cB^4_p$ for any $\ba \in \cB^4_p$ such that $\cI(\ba) = \{i : i \geq 0\}$.  By \cref{lemma:expectedtransform1d}, the rounding errors introduced by \(\tilde{L}\) are independent of the distribution of the random variables \(Y = (\tilde{L}X^t)^t\), as they are confined to uniformly random least significant bits.  Thus, 
	\begin{align}
		\mathbb{E}[   \tilde{L}(\tilde{L}X^t)^t -  L (\tilde{L}X^t)^t  ] &= \mathbb{E}[ \tilde{L}Y -  LY ]  = [E_1, E_1, E_1, E_1].
	\end{align}
	For the second term, \(L[  (\tilde{L}X^t)- (LX^t)  ]^t  \), we use the linearity of \(L\) and the vectorization operator, and we have 
	\begin{align}
		\mathbb{E}\left [ L[\tilde{L}X^t - LX^t]^t  \right] &= L \mathbb{E}[\tilde{L}X^t - LX^t]^t  = L [E_1, E_1, E_1, E_1]^t
	\end{align}
	Combining these results, the expected error for \(d = 2\) is 
	\begin{align}
		\mathbb{E}[\tilde{L}_2\ba - L_2\ba] &= \mathrm{vec} \left( [E_1, E_1, E_1, E_1] + L [E_1, E_1, E_1, E_1]^t \right).
	\end{align}
	Thus, \({\bf E}_2 = \mathrm{vec} \left( [E_1, E_1, E_1, E_1]+ L_1[E_1, E_1, E_1, E_1]^t\right)\), where \(L_1 = L\). By induction, the same reasoning applies for higher dimensions, where the error propagates recursively. Therefore, the expected error for dimension \(d\) is
	\begin{align}
		{\bf E}_d = \mathrm{vec} \left(  [ {\bf E}_{d-1}, {\bf E}_{d-1}, {\bf E}_{d-1}, {\bf E}_{d-1} ] +L_{d-1} \left[ \underbrace{{\bf E}_1, {\bf E}_1, \ldots, {\bf E}_1}_{4^{d-1}} \right]^t \right).
	\end{align}
\end{proof}}

Additionally, a similar analysis can be done for the decompression operator; however, if we assume $\beta \geq q - 2d+2$, where $q$ bits are used to represent the significand for the block-floating point representation, i.e., the integer coefficients of each element in the block, then no additional error will occur when applying the decorrelating linear transform operator (see~\cite[\S{4.3}]{errorzfp} for details).\footnote{The additional error that may occur from the decorrelating linear transform operator depends on $\beta$, the fixed-precision parameter. If $\beta \geq q - 2d+2$, where $q$ bits are used to represent the significand for the block-floating point representation, then no additional error will occur.}

Lastly, we will now discuss how the expected value is affected by the shift operator, $S_\ell(\cdot)$, a critical operator used by ZFP. 

\subsection{Shift Operator}
Even though the shift operator is lossless, it changes the magnitude of elements. Note that the shift operator, $s_\ell(\cdot)$, is linear; thus, we have the following simple lemma that describes how the expected value is shifted. 
\begin{lemma}\label{lemma:shift} 
	Define the {distribution ${ A}:={ A}_{\{\cB_p,\iota\}}$} with some precision $p$.  \teal{Let $ a \sim A$. Then}
$
	\teal{f_\cB}({\mathbb{E} [s_\ell(\teal{a})]})= 2^{-\ell}  \teal{f_\cB}( {\mathbb{E}[\teal{a} ]}).
$
\end{lemma}
\begin{proof} Let $a \teal{\ \sim \ } A$. Then 
		\begin{align*}
				\teal{f_\cB}(	{\mathbb{E} [s_\ell(a)]} )=	{\mathbb{E} [f_\cB(s_\ell(a))]}   = 	{\mathbb{E} \left[(-1)^{\teal{\xi}}  \sum_{i \in \cI(a)}2^{i-\ell}\right]} = 2^{-\ell} 	{\mathbb{E} [f_\cB(a)]}=	2^{-\ell}\teal{f_\cB}(	{\mathbb{E} [a]}). 
		\end{align*}
	 \end{proof}
	
In the next section we look at the composite operator of the ZFP compression steps and discuss the resulting error using the tools derived in \Cref{section:biastools}.

\newcommand{\tone}{{\bm z}}
\newcommand{\ttwo}{{\bm y}}
\newcommand{\tthree}{{\bm w}}

\section{ZFP Compression Error} 
\label{sec:totalcompressionerror}
In the current framework, the ZFP \newline(de)compression operators that introduce error are inherently nonlinear; however, to analyze the expected value of the error distribution, we decompose the full ZFP operator into four terms, each representing a nonlinear error caused by the truncation operator. 
Using the tools derived in \Cref{section:biastools}, the expected value of the total error distribution can be expressed as a sum of the expected value of each nonlinear term associated with a lossy operator. 
To begin our discussion, let $\tone =  {D}_3 D_4 D_5 \tilde{C} \bx $.  Using the distributive property of linear operators, the total compression error can be decomposed  as
\begin{align*}
  	\tilde{D}(\tilde{C}(\bx)) - D(C(\bx))
	&=  \tilde{D}_{2}  {D}_{3}  D_{4}  D_{5} (\tilde{C}(\bx)) -  D(C(\bx))  , \\
	&=  \tilde{D}_{2} \tone -  {D}_{2} \tone  +  {D}_{2}{D}_3 D_4 D_5 \left(    \tilde{C} \bx - {C} \bx  \right).
\end{align*} 
Continuing in the same manner, let $\ttwo = C_5 C_4 \tilde{C}_3 \tilde{C}_2 \bx$ and $\tthree = \tilde{C}_2 \bx$. 
The total compression error is decomposed as  
\begin{align}\label{eqn:decomposed}
D&(C(\bx))-\tilde{D}(\tilde{C}(\bx)) 
=  \left(\tilde{D}_2\tone- D_2\tone\right)+ D_2 D_3 D_4 D_5 \left(  \tilde{C}_8 \ttwo- C_8\ttwo\right) \\
&+  D_2 D_3 D_4 D_5 C_5 C_4 \left(\tilde{C}_3 \tthree- C_3 \tthree \right ) +D_2 D_3 D_4 D_5 C_5 C_4C_3\left ( \tilde{C}_2 \bx - {C}_2 \bx  \right). \notag
\end{align} 
Note, only the lossy operators, $\tilde{C}_2,$ $\tilde{D}_2$, $\tilde{C}_3$ and $\tilde{C}_8$, are nonlinear.
Thus, we have four sources of error: $\tilde{C}_2 \bx - {C}_2 \bx $,  $\tilde{C}_3 \tthree -C_3 \tthree$, $\tilde{C}_8 \ttwo  - C_8\ttwo$, and $\tilde{D}_2\tone - D_2\tone$.
Each term is propagated back to a floating-point representation of the original magnitude by applying the lossless decompression operators.
Additionally, note the dependencies, i.e. $\tone$ is dependent on $\ttwo$, $\ttwo$ is dependent on $\tthree$, and $\tthree$ is dependent on $\bx$.
To understand the \teal{error} bias, we will first look at each portion independently and examine the expected value of each term.
\cref{lemma:c2} presents the expected value for the last error term in \cref{eqn:decomposed} caused by the second compression step.
\cref{lemma:c3} and \cref{lemma:c8} present the expected value of the error caused by the third and eighth compression step, respectively.
Lastly, \cref{lemma:d2} presents the expected value for the first error term in \cref{eqn:decomposed} caused by the second decompression step.
 Note that each of the following lemmas assume a non-zero block, $\bx \neq \bfz$, as it is a special case since ZFP can represent it exactly with minimal bits.

\begin{lemma} \label{lemma:c2}

 Let $\bf{X} =F_{\cB}\left( \bf{A}^{4^d}_{\cB_k, \iota}\right)$ \teal{be a distribution such that for every realization $\bx$ of a random vector drawn from $\bf{X}$, it holds that $\bx_i \neq 0$ for all $i$} and $F^{-1}_{\cB}\left(\bx\right) \in \cB_k$, for some precision $k$. 
 \teal{Here, \( k \) represents the precision of the original input into the compressor; typically $k \in \{24,53 \}$ and $\iota_i$,  the starting index of the mantissa bits for each element $\bx_i$ in the vector, is determined by the input representation. }
 Let $0\leq \beta \leq q - 2d+2$, where $q \in \mathbb{N}$ is the precision for the block-floating-point representation such that $q \geq k$. Assume $\rho \leq q-2$, where $\rho = e_{max,\cB}({\bx})-e_{min,\cB}({\bx})+1$ is the exponent range for $\bx $ \teal{drawn from} $\bf{X}$. Then 
	\begin{align*}
	{\mathbb{E} \left[D_2 D_3 D_4 D_5 C_5 C_4C_3\left ( \tilde{C}_2 \teal{\bx} - {C}_2 \teal{\bx} \right)\right]}  =  \bfz. 
	\end{align*}
\end{lemma}

	\begin{proof}First let us look at the expected value of the nonlinear term, i.e., $ \tilde{C}_2 \teal{\bx} - {C}_2 \teal{\bx} $, 
	\begin{align*}
		{\mathbb{E}  \left[  \tilde{C}_2 \teal{\bx} - {C}_2 \teal{\bx} \right]}
		&=	{\mathbb{E} \left[   T_{\mathcal{S}} S_{\ell} F_{\cB}^{-1}(\teal{\bx} )-  S_{\ell}F_{\cB}^{-1} (\teal{\bx} ) \right]}= 	{\mathbb{E} \left[  T_{\mathcal{S}} \teal{\hat{\bx}} - \teal{\hat{\bx}} \right]}, 
	\end{align*}
	where $\teal{\hat{\bx}} = S_{\ell} F_{\cB}^{-1}({\teal{\bx} } )$ and $\cS = \mathbb{N}$. \teal{Recall that  in \Cref{sec:algorithm}, Step 2 defines $\ell = e_{max}(F_\cB^{-1}(\bx))- q+1$. Then,} $ e_{max,\cB}(\teal{F^{-1}_{\cB}(\hat{\bx} )}) = e_{max,\cB}(\teal{F^{-1}_{\cB}(\bx)} )-\ell = q- 1 \geq 0 $ and $ e_{min,\cB}(\teal{F^{-1}_{\cB}(\hat{\teal{\bx}})}) = e_{min,\cB}(\teal{F^{-1}_{\cB}( \bx)})-\ell ={-\rho}+q-1 \geq 0$. Thus, \teal{applying \cref{lemma:truncbinary} component-wise}, we have
	\begin{align}
		{\mathbb{E} [  \tilde{C}_2 {\teal{\bx}} - {C}_2 {\teal{\bx}}   ]} = F_\cB^{-1} (\bfz), \label{eqn:c2}
	\end{align}	
	\teal{as both $e_{max,\cB}(\teal{F^{-1}_{\cB}(\hat{\bx})})$ and $e_{min,\cB}(\teal{F^{-1}_{\cB}(\hat{\teal{\bx}})})$ are bounded away from the index $0$, where the truncation operator begins zeroing out entries.}
	Combining \cref{eqn:c2} with \cref{lemma:shift} and the linearity of expectation, the observation follows	
	\begin{align*}
			\mathbb{E} [D_2 D_3 D_4 D_5 C_5 C_4C_3&\left ( \tilde{C}_2 {\teal{\bx}} - {C}_2 {\teal{\bx}} \right)] \\  
			&={ \mathbb{E}\left[	F_\cB S_{-\ell} \left(D_3 D_4 D_5 C_5 C_4C_3\left(  \tilde{C}_2 {\teal{\bx}} - {C}_2 {\teal{\bx}}  \right) \right)\right]},\\
		& = 2^{\ell}L_d^{-1} {\mathbb{E}\left[ L_d F_\cB \left(\tilde{C}_2 {\teal{\bx}}  -  {C}_2 {\teal{\bx}}  \right) \right]},\\
		& = 2^{\ell} \teal{F_\cB}\left( {\mathbb{E}\left[  \tilde{C}_2 {\teal{\bx}}- {C}_2 {\teal{\bx}}   \right]} \right) = \bfz.
	\end{align*}
	
	\end{proof}
Next, \cref{lemma:c3} presents the expected value of the error caused by the third compression step, i.e., the forward transform operator.

\begin{lemma}
	\label{lemma:c3}
	Let  $\bf{W}:={\bf A}^{4^d}_{\{\cB_q,0\}}$ \teal{be a distribution such that for every realization $\bw$ of a random vector drawn from $\bf{W}$, it holds that $f_\cB(\tthree_i) \neq 0$ for all $i$}, for some precision $\teal{{q}}$. Let $0\leq \beta \leq \teal{{q}} - 2d+2$, where $q \in \mathbb{N}$ is the precision for the block-floating-point representation. Then 
\begin{align*}
	{\mathbb{E}  \left[  D_2 D_3 D_4 D_5 C_5 C_4 \left( \tilde{C}_3 {\teal{\bw}}- C_3 {\teal{\bw}} \right)\right]}  =  2^{\ell}  {\cL}_d^{-1}    {\bf E}_d. 
\end{align*}	
	\end{lemma}
\begin{proof}
	Similar to \cref{lemma:c2}, the expectation of the nonlinear term is
	\begin{align}
		{\mathbb{E}  \left[ \tilde{C}_3 {\teal{\bw}} - C_3 {\teal{\bw}} \right]}& = F_\cB^{-1} \left({\mathbb{E}\left [  \tilde{L}_d F_\cB({\teal{\bw}}) - L_dF_\cB({\teal{\bw}}) \right]}  \right)= F_\cB^{-1}\left( {{\bf E}_d} \right), \label{eqn:c3}
	\end{align}
where ${{\bf E}_d}$ is defined by \cref{lemma:expectedtransform1d} when $d = 1$ or \cref{thm:expectedtransform} when $d = 2,3$. Combining (\ref{eqn:c3}), \cref{lemma:shift}, and the linearity of expectation, the observation follows
	\begin{align*}
	{\mathbb{E} \left[  D_2 D_3 D_4 D_5 C_5 C_4 \left( \tilde{C}_3 {\teal{\bw}} - C_3 {\teal{\bw}} \right)\right]} & = {\mathbb{E}\left[   F_{\cB} S_{-\ell}F_\cB^{-1}{\cL}_d^{-1}F_\cB  \left( \tilde{C}_3 {\teal{\bw}}- C_3 {\teal{\bw}} \right ) \right]}\\
	&= 2^{\ell}  {\cL}_d^{-1}  {{\bf E}_d}.
	\end{align*}
\end{proof}
 \cref{lemma:c8} presents the expected value of the error caused by the eighth compression step, i.e., the truncation of the transform coefficients. 
\teal{The precision of the negabinary value in Step 8 is $\tilde{q} = q + 2$, as the sign bit and guard bit are utilized when mapping from the signed binary representation in Step 5 (see \cite{errorzfp}). Define $\Delta$ as the quantization step introduced by truncating a negabinary number in Step 8. 
In our case, it is defined as $\Delta = 2^{q-\beta+2}$, where $q$ is the precision of the block floating-point representation, and $\beta$ is the negabinary truncation parameter, which represents the number of bit planes kept in Step 8.
}
\begin{lemma} \label{lemma:c8}	
Let  $\bf{Y}:= {\bf A}^{4^d}_{\{\cN_{\teal{\tilde{q}}},0\}}$ \teal{be a distribution such that for every realization $\by$ of a random vector drawn from $\bf{Y}$, it holds that $f_\cN(\by_i) \neq 0$ for all $i$}, for some precision $\teal{\tilde{q}} = q +2$ and \teal{$0\leq \beta \leq q - 2d+2$.}  Let $e_{min,\cN}(\ttwo) > \teal{\tilde{q}-\beta-1}$. Then 
	\begin{align*}
{\mathbb{E}\left[D_2 D_3 D_4 D_5 \left(  \tilde{C}_8 \teal{\by} - \teal{C_8}\teal{\by} \right)\right ]}&= \teal{\pm \Delta  \frac{2^{\ell}}{6}L^{-1}_d\bfo , }
\end{align*}
\teal{where the expected value is positive/negative if $ \tilde{q}-\beta-1$ is even or odd, respectively. }
	\end{lemma}
\begin{proof}
If $e_{min,\cN}(\ttwo) > \teal{\tilde{q}-\beta -1 }$ , then $e_{max,\cN}(\ttwo_i) > \teal{ q - \beta +1 }$ for all $i$. \teal{Recall that, $	\tilde{C}_8 = T_{\mathcal{P}},$ where,  $\mathcal{P} = \{ i \in \mathbb{Z} : i > q - \beta +1 \}$.}
	  \teal{	  	By applying \cref{lemma:truncneg}, the quantization error is given by
	  	\begin{align} F_\cN\left( \tilde{C}_8 \teal{\by} - \teal{C_8}\teal{\by} \right) = F_\cN\left( {T}_\cP\teal{\by} - \teal{\by} \right).
	  	\end{align}
From \cref{lemma:truncneg}, this quantization error for each element of the vector lies in the interval $(\pm 2/3, \mp 1/3)\Delta$ if $q - \beta +1 $ is even or odd, respectively.   The inverse decorrelating transform then scales and transforms this error. Thus, the expected error is given by }
\begin{align}
	{\mathbb{E}\left[ F_\cN\left(  \tilde{C}_8 \teal{\by} -  \teal{C_8}\teal{\by}   \right)\right]} 
	& =  \frac{(-2)^{q+1-\beta}}{6} \bfo \teal{=  \pm\frac{\Delta}{6} \bfo.} \label{eqn:c8}
 \end{align}
Combining (\ref{eqn:c8}), \cref{lemma:shift}, and the linearity of expectation, we have 
\begin{align*}
{\mathbb{E}\left[D_2 D_3 D_4 D_5 \left( \tilde{C}_8 \teal{\by} -   \teal{C_8}\teal{\by}   \right)\right]} 
= 2^{\ell}L^{-1}_d {\mathbb{E}\left[  F_\cN  \left(  \tilde{C}_8\teal{\by}-   \teal{C_8}\teal{\by}   \right) \right]} \teal{= \pm \Delta  \frac{ 2^{\ell}}{6} L^{-1}_d \bfo.}
\end{align*}
\end{proof}
\teal{Furthermore, we note that $L_d^{-1} \bfo$ frequently appears in our expressions below. For 
	$d=1$, the corresponding values are explicitly provided in \cref{tab:diststats}.}
Finally, \cref{lemma:d2} presents the expected value for the first error term in \cref{eqn:decomposed} caused by the \teal{decompression operator associated with Step 2}, mapping the values \red{within} the block back to an IEEE floating-point representation.
\begin{lemma}\label{lemma:d2}
         Let $k \in \mathbb{N}$ be the precision for the floating-point representation.
        Let $\bf{Z}:= {\bf A}^{4^d}_{\{\cN_q,\iota\}}$  \teal{be a distribution such that for every realization $\bz$ of a random vector drawn from $\bf{Z}$, it holds that $f_\cB(\bz_i) \neq 0$ for all $i$ }, for some precision $q$. 
         Then 
	\begin{align*}
	{\mathbb{E} \left[\tilde{D}_2 \teal{\bz} - {D}_2 \teal{\bz}\right]}  =  \bfz. 
	\end{align*}
\end{lemma}
\begin{proof}
	Combining the definition of $\tilde{D}_2$ and $D_2$ with \cref{lemma:truncbinary} and \cref{lemma:shift}, we observe 
	\begin{align*}
{\mathbb{E} \left[ \tilde{D}_2 (\teal{\bz})  - D_2( \teal{\bz}) \right] }
	& ={\mathbb{E} \left[  F_\cB \left(S_{-\ell} fl_k (\teal{\bz})\right) - F_\cB\left(S_{-\ell}(\teal{\bz}) \right) \right]} \\
	&=  2^{\ell} { \mathbb{E}\left[ F_\cB \left(fl_k (\teal{\bz})-  \teal{\bz}   \right) \right]} = \bfz,
	\end{align*}
	where $fl_k (\tone)_i  = t_{\cR_{ik}} (\tone_i)$ with $\cR_{ik} = \{ j \in \mathbb{Z}: j > e_{max, \cB} (\tone_i) - k \}$ for all $i$. 
\end{proof}
\teal{Note that the above lemma can also be deduced logically, as the operator $fl_k (\cdot)_i$ mimics rounding to the nearest value, resulting in an expected error of zero when converting to floating-point representation. However, the same cannot be said for Step 2 compression, as it involves truncation rather than rounding to the nearest value, simply zeroing out bits.} 
Now that the expected value of each error term is explicitly defined, using the linearity of expected values, we can determine the expected value of the total compression error.
\begin{theorem}
	\label{thm:bias}
\teal{For some precision $k$,} let $\bf{X}:= F_{\cB}\left({\bf A}^{4^d}_{\{\cB_k,\iota\}}\right)$ \teal{be a distribution such that for every realization $\bx$ of a random vector drawn from $\bf{X}$, it holds that $\bx_i \neq 0$ for all $i$}. 
Let $0\leq \beta \leq q - 2d+2$, where $q \in \mathbb{N}$ is the precision for the block-floating-point representation.
Assume the respective assumptions from \cref{lemma:c2}--\cref{lemma:d2} defined by the distributions from \cref{eqn:decomposed}.
Then 
\begin{align} \label{eqn:bias}
{\mathbb{E}\left[  \tilde{D}(\tilde{C}(\teal{\bx}))- D(C(\teal{\bx}))   \right]} 
	&=  \teal{ \pm \Delta \frac{2^{\ell}}{6} \left( L^{-1}_d \bfo + {{\bf E}_d} \right)}, 
\end{align}
\teal{where $\ell$ is defined as the difference between the maximum exponent of $\bx$ in its block-floating-point representation and the precision $q$, specifically,
$ \ell = e_{max}(F_\cB^{-1}(\bx)) - q + 1. $ The operator $L_d^{-1}$ is defined in \Cref{sec:step3}, and ${\bf E}_d$ is defined in \cref{thm:expectedtransform}. }
\end{theorem}
\begin{proof}
Using the distributive property of linear operators and adding by zero, the total compression error is decomposed as
\begin{align}
 \tilde{D}&(\tilde{C}(\teal{\bx})) -  D(C(\teal{\bx}))
 =   \left( \tilde{D}_2{\teal{\bz}} - D_2{\teal{\bz}} \right)+ D_2 D_3 D_4 D_5 \left(  \tilde{C}_8 {\teal{\by}} -  C_8{\teal{\by}} \right) \label{eqn:totalerror}\\
&+  D_2 D_3 D_4 D_5 C_5 C_4 \left( \tilde{C}_3 {\teal{\bw}} - C_3 {\teal{\bw}}  \right )+D_2 D_3 D_4 D_5 C_5 C_4C_3\left ( \tilde{C}_2 {\teal{\bx}} -  {C}_2 {\teal{\bx}}  \right), \notag 
\end{align} 
where ${\teal{\by}} =  {D}_3 D_4 D_5 \tilde{C} {\teal{\bx}} $, ${\teal{\bz}} = C_5 C_4 \tilde{C}_3 \tilde{C}_2 {\teal{\bx}}$, and ${\teal{\bw}} = \tilde{C}_2 {\teal{\bx}}$. Even though the terms are dependent,  expectation is linear, i.e., regardless of whether the sum of random variables are independent, the expected value is equal to the sum of the individual expected values. Combining \cref{eqn:totalerror} with \cref{lemma:c2}--\cref{lemma:d2}, and the linearity of expectation, the result follows.
\end{proof}
As can be seen, the expected value in \cref{eqn:bias} is not zero. As it is assumed that  $0\leq \beta \leq q - 2d+2$, the magnitude of the leading order term in the theoretical expected value is   
\begin{align}
\teal{\mathcal{O}\left( \pm \Delta \frac{2^{\ell}}{6}L_d^{-1}\bfo\right),}
\end{align} 
which is from the truncation of a negabinary representation that is magnified by the backwards transform operator.
\teal{For $1\leq d \leq 4$, we can bound the term $\|L_d^{-1}\bfo\|_\infty \leq \left(\frac{15}{4} \right)^d$. }
Note that the assumptions in \cref{thm:bias} are relatively strict. 
Mainly, when applying $\tilde{C}_8$, the negabinary truncation, we require the index of the leading bit in each element to be greater than $q -\beta+1$, which will most likely be unsatisfied for blocks whose elements are smooth; this is because the forward transform operator pushes all the energy into the low frequency components.
For example, if the forward transform operator is applied to the constant vector with white noise, \teal{$\omega$, where $\omega_i \ll 1$}, we have, $
L_1(\bfo + \omega) \approx \begin{bmatrix}1 &0 &0 &0\end{bmatrix}^t.
$
In this case, when the truncation operator is applied, only the first element satisfies the assumption causing the predicted mean for the remaining elements after Step 8 to be an overestimate  when, in reality, it is closer to zero. 
Once the decorrelating transform, $L_d^{-1}$, is applied, the \teal{bias is magnified; however, in practice, this error is typically smaller.
Nevertheless,} this does not necessarily mean that the results from \cref{thm:bias} are not useful for approximating the bias.
Even if the theoretical estimate of the mean is an over or underestimate, the relative magnitude of the leading order bias term is captured, which is associated with the error caused by the compression Step~8 and the decompression Step~3.
In the next section, we provide test results on a simulated dataset to test the accuracy of our theoretical bias estimation. 

\section{Numerical Results}
\label{sec:results}
As observed in \Cref{sec:totalcompressionerror}, there is indeed a bias in the compression error in the current implementation of the ZFP algorithm. The first numerical test studies the effectiveness of our theoretical results on generated $4^d$ blocks. The second test is on real-world data from a climate application; \cite{Hammerling} has a more in-depth study of the bias of ZFP for this particular data set. 

\subsection{Synthetic $4^d$ Block}
\label{sec:generatedBlock}
In the first numerical test, we wish to mimic the worst possible input for ZFP for a chosen exponent range,
 \begin{align}
 	\label{eqn:rho}
\rho = e_{max}-e_{min},  
\end{align}
where $e_{max} = e_{max,\cB}(\bx)$ and $e_{min} = e_{min,\cB}(\bx)$ for block $\bx \in \mathbb{R}^{4^d}$. In each example, a $4^d$ block was formed with absolute
values ranging from $2^{e_{min}}$ to $2^{e_{max}}$. The exponent
$e_{min}$ remains stationary while $e_{max}$ varies, depending on the
chosen exponent range.
The interval $[e_{min}, e_{max}]$ was divided
into $4^d$ evenly spaced subintervals. Each value of the block was randomly 
selected from a uniform distribution in the range
$[2^{e_{min}+(h-1)\frac{e_{max}-e_{min}}{4^d}},
2^{e_{min}+h\frac{e_{max}-e_{min}}{4^d}}]$ with {sub}interval
$h\in \{1,\dots,4^d \}$ and uniform randomly assigned sign. Using the C$++$ standard library function
{\it{random\_shuffle}}, the block was
then randomly permuted to remove any bias in the total sequency order. The block was then compressed and decompressed with \red{precision $\beta$.}

For \cref{fig:1drelative error},
the data is represented and compressed in single precision (32-bit IEEE standard), i.e., $k= 24$,  with $e_{min}  = -20$, while $e_{max} $ varies with respect to the required exponent range. Note that similar results can be produced for any value of $e_{min}$ as the block-floating-point representation converts the block to signed integers. The only difference in the results occurs when the exponent range, $\rho$, increases. In this example we let $d=1$. One million blocks were generated using the above routine for a single $\beta$, which were then compressed and decompressed. The {\it average compression error} of the one million \teal{synthetic} blocks was recorded, denoted by the vector $\bar{\bx}$. The {\it theoretical expected value}, represented as a vector $\bm{\mu}$, is defined by \cref{eqn:bias}. For this particular example, $q = 30$ and $\ell = (e_{min}+\rho)-q+1 = -49 +\rho$. Note that $2^{e_{max} - k}$ is the minimal representable magnitude for the decompression Step 2, i.e., the minimal representable magnitude for the conversion from a block-floating-point representation to an IEEE representation. As the error was calculated in double precision, any experimental or theoretical mean whose absolute value is less than $2^{e_{max} - k}$ is essentially zero; thus,  for our demonstration, we rounded such values to zero.

For $ \rho = 0$, meaning that the
magnitude of the absolute values of the four element block are similar, $\ell$ is as small as it can be resulting in the theoretical bias to have a magnitude of 
$
\mathcal{O}\left( 2^{e_{min}-q+1  } \frac{(-2)^{q-\beta+1 }}{6}L_1^{-1}\bfo\right).
$
 As the
exponent range increases, {fewer bits are used to represent the smaller values in each block during Step 2, which will result in a larger value for $\ell$, increasing the magnitude of the expected value further away from zero, resulting in a magnitude of 
$
	\mathcal{O}\left(2^{{e_{max}}-q+1  } \frac{(-2)^{q-\beta+1 }}{6}L_1^{-1}\bfo\right).
$
The top row in \cref{fig:1drelative error} presents the following results for $\rho = 14$. The leftmost plot depicts the ratio of the experimental and theoretical mean, i.e., $\bar{x}_i/\mu_i$ for each element $i$. 
The vertical axis represents the element index, $i$. For varying $\beta$, represented along the horizontal axis, the magnitude of the ratio is from 0.96 to 1.04, represented as a variation of red to blue, respectively. 
For all $\beta$ values, one can see that the ratio is approximately one, which means that the theoretical prediction is correct in sign as well as in magnitude. 
Note that the white blocks represent when the experimental mean for the element index was less than $2^{e_{max} - k}$, where $k= 24$. To see if the theoretical prediction of the expected value is mimicking the experimental mean,  the middle figure is a side-by-side comparison where the top plot is the experimental mean, and the bottom is the \teal{predicted error bias}. 
As $\beta$ decreases, the theoretical mean follows the same pattern as the experimental mean for all elements $i$. Lastly, the rightmost figure depicts the relative error between the \teal{predicted error bias} and the experimental mean. 
The vertical axis represents the relative error, while the horizontal axis represents the fixed precision parameter $\beta$. Each element index is represented as a different color. 
At most, our \teal{predicted error bias} is off by $4\%$ from the experimental mean; however,  in most cases, it is much less. \cref{fig:1drelative error_new} depict the same results but for when $\rho=0$. 
Similar conclusions can be seen; however, the magnitude of the expected value has increased as \cref{eqn:bias} is a function of $\rho$. 

	\vspace{5pt}
	\begin{figure}
	\centering
	\begin{subfigure}[b]{\textwidth}
		\includegraphics[width=\textwidth]{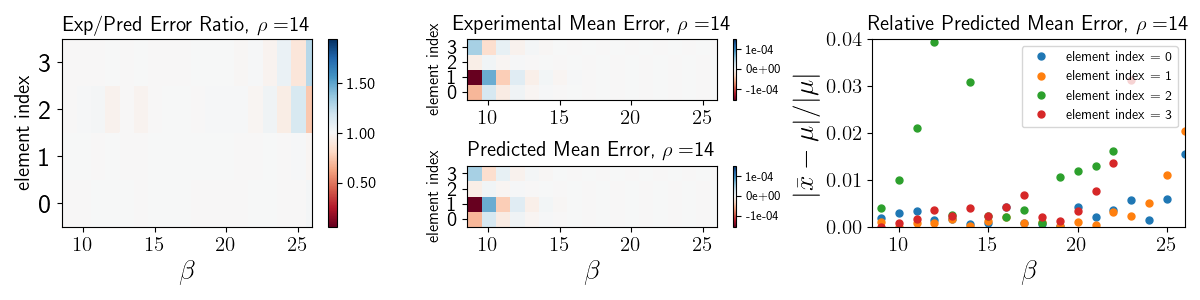}
		\label{fig:rho141d}
	\end{subfigure}
	\caption{1-d Simulated Example: Each row depicts the ratio, a side-by-side comparison, and the relative error of the experimental and \teal{predicted error bias} error for $\rho=14$ , where $\rho$ is the \teal{exponent} range of values in a block defined by \cref{eqn:rho}.} 
	\label{fig:1drelative error}
\end{figure}

\cref{fig:2drelative error} and \cref{fig:3drelative error} depict the same test as described above for $d = 2$ and $d= 3$. ZFP typically can compress more effectively as $d$ increases. This is partly because the forward transform operator, $L_d$, is applied to each dimension, pushing more of the energy into the lower frequency components. Thus, the index of the leading non-zero bit after applying the forward transform tends to decay with respect to the ordering of the transform coefficients due to the total sequency applied in Step 4. 
Thus, as $\beta$ decreases, our \teal{predicted error bias} tends to degrade as some of the assumptions in  \cref{thm:bias} are no longer valid. Mainly, the assumption that ${e_{max}\left( \tilde{C}_8C_5C_4\tilde{C}_3\tilde{C}_2(\bx)\right)} \geq q-\beta+2 $ is violated. 
\teal{This violation implies that the truncation operator in Step 8 leads to the truncation of leading one-bits, as described in \cref{lemma:truncneg}. }
The error that is caused by the truncation operator, when $\beta$ is small, is \teal{magnified} when the backward decorrelating transform operator is applied, causing the mean to be overestimated, which can be seen in  the rightmost figures in \cref{fig:2drelative error} and \cref{fig:3drelative error}. 

\subsection{Climate Data Real-World Example} 
\label{section:ncar}
As climate model simulations produce large volumes of data, climate scientists have been interested in adopting lossy compression schemes. 
It was noted in \cite{Hammerling}  that ZFP has a bias with respect to the element index within a block. 
In this section, we perform the same test as in \cite{Hammerling}, but also include results concerning the \teal{predicted error bias} from \cref{thm:bias} to show the accuracy of our predicted bias within a real application area. 
For this application, we test the surface temperature (TS) data from the  CESM Large Ensemble Community Project (CESM-LENS).  
The publicly available CESM-LENS data set contains 40 ensemble runs for the period 1920-2100. As in \cite{Hammerling}, we use only the historical forcing period, i.e., from 1920-2005, for more details see \cite{Hammerling}. At this resolution, the CAM grid contains $288 \times 192$ grid points and 31,390 time slices. 
\red{From left to right, \cref{fig:ncaroldzfp} presents the experimental and theoretical mean error over time, along with the minimum of the cumulative distribution function (CDF) and complementary cumulative distribution function (CCDF) of the relative error between the two, presented as a unitless measure with respect $\Delta$. 
The leftmost figure shows the grid cell-level observed errors averaged across time for the daily TS data at $\beta = 20$. 
The middle figure illustrates the theoretical mean error calculated using \cref{thm:bias}, based on the maximum exponent $e_{max,\cB}$  for each block. 
Finally, the rightmost figure displays the minimum of the CCDF and CDF of the relative error between the actual mean (\(\bar{\mathbf{x}}\)) and the predicted mean (\(\bm{\mu}\)), calculated as \(\frac{\bar{x}_i - \mu_i}{\hat{\Delta}}\), where \(\hat{\Delta} = \Delta \cdot 2^{e_c - q + 1}\). Here, \(e_c\) is the common exponent of the dataset, which for this dataset is \(e_c = 8\).}
One can see that the theoretical mean is of the same magnitude as the true mean. 
Additionally, the pattern in the bias is similar, if not the same. 
\red{In this dataset, which contains many smooth regions, ZFP compression incurs little to no error in these areas for large $\beta$ values. As a result, while we do observe some large relative errors  these are primarily due to our predicted error overestimating the true error. This is because the observed error in smooth regions is often zero, while the predicted error remains non-zero. However, more often than not, we accurately predict the error. This is evident in the rightmost graph, where small relative errors dominate, resulting in a left-skewed graph. }
We are able, on the whole, to accurately depict the bias that occurs in ZFP for the daily TS data.
%
\vspace{5pt}
\begin{figure}
	\centering
	\begin{subfigure}[b]{\textwidth}
	\includegraphics[width=\textwidth]{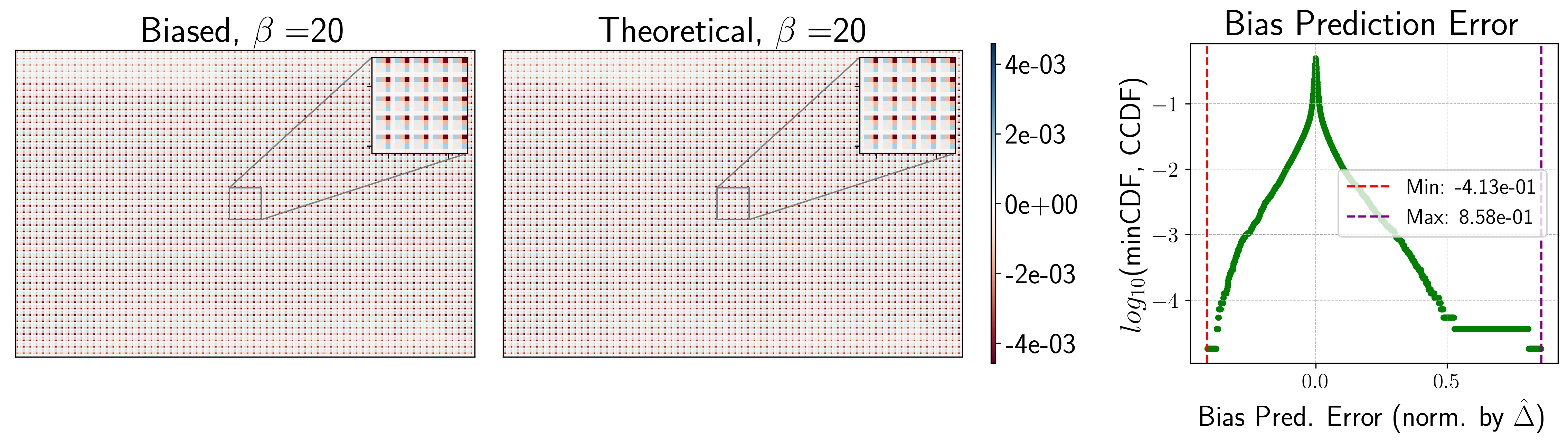}
	\label{fig:true20}
	\end{subfigure}
\caption{Climate Data Real-World Example: Each row depicts the experimental mean \teal{error}, the \teal{predicted error bias}, and \teal{the minimum of the complementary cumulative distribution function (CCDF) and the cumulative distribution function (CDF) of the relative error $(\bar{x}i - \mu_i) / \hat{\Delta}$, where $\beta$ is the number of bit-planes kept at Step 8 for $\beta = 20$, $\hat{\Delta} = \Delta \cdot 2^{e_c - q + 1}$ otherwise, and $e_c$ is the common block-floating-point exponent, and $q$ is the precision of the block floating-point representation.}}
\label{fig:ncaroldzfp}
\end{figure}

\section{ZFP Bias Correction}
\label{sec:correction}
\teal{While the magnitude of the bias is fixed relative to the magnitude of the error from the quantization step ($\Delta$), its impact depends on the ratio between $\Delta$ and the transform coefficients. In practice, $\Delta$ is often much smaller than the values being compressed, which generally makes the bias relatively small. However, the relative error in the transform coefficients may become significant in certain scenarios, such as when the coefficients are close to zero. While this relative error may seem large, it does not necessarily matter in practice, as transform coefficients that are zero (or near zero) have negligible impact on the reconstruction. Instead, the bias introduced by quantization is primarily relevant when considering the reconstructed values, where artificial correlation in the errors may affect the final output. In such cases, some application areas may benefit from reducing the bias to minimize its impact on the reconstructed values.
} 
Simple modifications to ZFP that involve rounding can be implemented, drastically reducing the bias's magnitude. We saw in \cref{thm:bias} that the truncation of the negabinary representation causes the largest source of \teal{error, which is then amplified by the application of the inverse decorrelating operator.} Simplifying the results, the error is either in $\teal{\left( -\frac{1}{3}\Delta, \frac{2}{3}\Delta \right)}$ or $\teal{\left( -\frac{2}{3}\Delta, \frac{1}{3}\Delta \right)}$, \teal{where $\Delta$ denotes the quantization step,} before it is propagated back by $L_d^{-1}$. 
\teal{ Refer to \cref{lemma:truncneg} and \cref{lemma:c8} for additional details.} 
	 Thus, if we can mitigate these errors before the propagation back to the original space, then we can reduce the bias effect. 
	 One simple modification to ZFP that was proposed in \cite{Hammerling} is to offset the decompressed values in order to center the error around zero. 
	 We center the reconstructed transform coefficients within the interval after the negabinary truncation by adding(subtracting) a scaled factor of $\frac{1}{6}$, i.e., shift the interval of the error from 
	$\teal{ \left( \mp\frac{1}{3}\Delta, \pm\frac{2}{3}\Delta \right)}$ to $\teal{ \left( -\frac{1}{2}\Delta, \frac{1}{2}\Delta \right)}$.  Note that the error interval of $\teal{ \left( \mp\frac{1}{3}\Delta, \pm\frac{2}{3}\Delta \right)}$ applies when the leading one bit is not truncated. Otherwise, the truncation operator maps the value to 0, i.e., $t_{\cS}(x) = 0$ and adding(subtracting) $\frac{1}{6}\teal{\Delta}$ would add \teal{a nonzero term}. This rounding scheme is referred to as {\it postcompression} rounding. 
	\teal{Adding a nonzero term can be problematic, particularly when the true value is very close to zero. In such cases, the coefficient may be represented as $\pm\frac{1}{6}\teal{\Delta}$, which can be significantly larger in magnitude than the true value itself. This introduces a noticeable bias, especially when the leading one-bit is truncated.}

 \teal{To address this issue, we propose a new rounding option called {\it precompression} rounding. In this approch,} the transform coefficients are modified by adding(subtracting) \teal{$\frac{1}{6} \Delta$}, \emph{before} converting to negabinary and applying the truncation step. 
This has \teal{a similar effect to {\it postcompression} rounding when the leading one-bit is not truncated but significantly reduces the bias when this bit is truncated} as it centers the error around zero but is more resilient to biasing effects from \cref{lemma:truncneg}. 
\teal{Both precompression and postcompression rounding options are available in ZFP when operating in either fixed-precision or fixed-accuracy mode, providing flexibility to reduce bias based on application needs.
Precompression rounding is analogous to mid-tread quantization, while postcompression rounding is analogous to mid-riser quantization. Both approaches achieve ``round-to-nearest'' logic, which reduces bias compared to the original, biased quantization scheme that simply truncates the negabinary representation.}
\footnote{In ZFP~1.x, precompression, postcompression, and no rounding are available through the \texttt{ZFP\_ROUND\_FIRST}, \texttt{ZFP\_ROUND\_LAST}, and \texttt{ZFP\_ROUND\_NEVER} compile-time settings.}

Let us define the \teal{bias correction} operator for {\it post} and {\it precompression} rounding as follows: 
\teal{For a scalar $ a \in \cN $, the \teal{bias correction} operator is defined as
$ \Psi_{\Delta}(a) := a + f_\cN^{-1} \left( \pm \frac{ 1 }{6} \Delta \right),$
where $\Delta = 2^{q - \beta +2}$, and the sign ($+$ or $-$) is determined by whether $q-\beta+2$ is even or odd, respectively.
This scalar definition trivially extends to vectors. For a vector $ 
 \in \cN^{4^d} $, the \teal{bias correction} operator is applied element-wise, i.e.,
$$ \Psi_\Delta( \bd ) := \bd + F_\cN^{-1}\left (\pm \frac{ 1}{6} \Delta \bfo_n\right). $$}
\teal{Here, ${\teal{{q}}} - \beta+1 \in  \teal{\mathbb{N}}  $} is the starting index of the bits that will be discarded when the truncation operator  in Step 8, $q$ is the precision of the block floating-point representation, and $\beta$ is the number of bit-planes kept during the truncation of the negabinary representation.  }
For {\it precompression} the \teal{bias correction} operator is applied between $C_7$ and $C_8$, while for {\it postcompression} the \teal{bias correction} operator is applied between $D_8$ and $D_7$. 
 For the  {\it postcompression} step we can revise \cref{lemma:truncneg}, as seen in \cref{lemma:truncneg_postcompression}.
\begin{lemma}\label{lemma:truncneg_postcompression}
	Assume $\teal{q},\teal{\beta} \in \teal{\mathbb{N}} $ \teal{such that $q> \beta-2$}, ensuring $q - \beta+1 \in \teal{\mathbb{N}}$ and $\cS = \{i \in \mathbb{Z}:   i \geq  q - \beta+2 \}$. Define the distribution $ A:= A_{\{\cN_{\teal{q+2}},0\}}$.  \teal{Let $a\teal{\ \sim\ } A$, then }
	\begin{itemize}
		\item[(i)] Assume $e_{max,\cN}(a)  \geq  q - \beta+2 $. 
		Then $f_\cN(\teal{\Psi_\Delta(t_{\cS}(a)) - a} ) \in \Delta \left( -\frac{1}{2},\frac{1}{2}\right)$, such that $\mathbb{E}[\teal{\Psi_\Delta(t_{\cS}(a))- a}]=0 $,
		\item[(ii)] Assume $e_{max,\cN}(a) <   q - \beta+2 $ . Then $\mathbb{E}[\teal{\Psi_\Delta(t_{\cS}(a))- a}  ] =  \mathbb{E}[\teal{a}] \pm \frac{ 1 }{6}  \teal{\Delta}$.
	\end{itemize}
\end{lemma}
\begin{proof}
	\begin{itemize}
		\item[(i)]	Let  $e_{max,\cN(a)} \geq q- \beta +2 $\teal{, i.e., the leading one bit is not truncated.}  Observe from \cref{lemma:truncneg} that if $q- \beta +2$ is odd, then
		\begin{align*}
		- \frac{1}{3}  \teal{\Delta}&\leq f_{\mathcal{N}} \left( \teal{t _{\mathcal{S}} (a)- a} \right)  \leq  \frac{2}{3}  \teal{\Delta},
		\end{align*}
		Thus, $ \red{f_{\mathcal{N}} \left({\Psi_\Delta(t_{\cS}(a))- a }\right)} $  is bounded by 
		\begin{align*}
	- \frac{1}{3} \teal{\Delta} - \teal{\Delta} \frac{1}{6}&\leq \red{f_{\mathcal{N}} \left({\Psi_\Delta(t_{\cS}(a))- a } \right)}   \leq  \frac{2}{3} \teal{\Delta} -  \frac{1}{6}  \teal{\Delta},  \\
			- \frac{1}{2}  \teal{\Delta} &\leq  \red{f_{\mathcal{N}} \left({\Psi_\Delta(t_{\cS}(a))- a } \right)}  \leq  \frac{1}{2}  \teal{\Delta}.
	\end{align*}
		Similarly, the same can be shown when $q- \beta +2$ is even. 
		\item[(ii)] Same as $(ii)$ in \cref{lemma:truncneg}.
	\end{itemize}
\end{proof}	

 Similarly, for the {\it precompression} step, we can revise \cref{lemma:truncneg}, as seen in \cref{lemma:truncneg_precompression}. 
\begin{lemma}\label{lemma:truncneg_precompression}
	Assume $\teal{q},\teal{\beta} \in \teal{\mathbb{N}} $ \teal{such that $q > \beta-1 $}, ensuring $\teal{q - \beta +1  } \in \teal{\mathbb{N}}$ and $\cS = \{i \in \mathbb{Z}:   i \geq  q - \beta+2\}$. Define the distribution $ A:= A_{\{\cN_{\teal{q+2}},0\}}$.  \teal{Let $a\teal{\ \sim\ } A$, then } 
\begin{itemize}
	\item[(i)] Assume $e_{max,\cN}(a) \geq  q - \beta+2 $. 
	Let $\hat{a} = \teal{\Psi}_\Delta(a),$ then $f_\cN(\teal{t_{\cS}(\hat{a})- a }) \in  \teal{\Delta} \left( -\frac{1}{2},\frac{1}{2}\right)$, such that $\mathbb{E}[\teal{t_{\cS}(\hat{a})- a }] =0 $,
	\item[(ii)] Assume $e_{max,\cN}(a)>  q - \beta+2 $. 
	Then $\mathbb{E}[\teal{ t_{\cS}(\hat{a})- a }] =  \mathbb{E}[\teal{a}] $.
\end{itemize}
\end{lemma}
\begin{proof}
\begin{itemize}
	\item[(i)]	Let  $e_{max,\cN}(a)\geq  q - \beta+2 $\teal{, i.e., the leading one bit is not truncated.} Observe that for $\cS = \{i \in \mathbb{Z}:   i \geq  q - \beta+2 \}$, we have 
	\teal{\begin{align*}
		 t_{\cS}(\hat{a}) - a & =  t_{\cS}(\hat{a})  - \Psi_{\Delta}(a) +    f_\cN^{-1} \left ( \pm \frac{1}{6}   \teal{\Delta}  \right) =t_{\cS}(\hat{a})  -\hat{a}  +   f_\cN^{-1} \left ( \pm \frac{1}{6}   \teal{\Delta} \right) 
	\end{align*}}
		It then follows that if $  q - \beta+2  $ is odd, then 
	\teal{\begin{align*}
		- \frac{1}{3}  \teal{\Delta} - \frac{1}{6}  \teal{\Delta} &\leq f_{\mathcal{N}} \left(	 t_{\cS}(\hat{a})  - a \right)  \leq  \frac{2}{3}  \teal{\Delta}  - \frac{1}{6} \Delta,\\
		- \frac{1}{2}  \teal{\Delta} &\leq f_{\mathcal{N}} \left(  t_{\cS}(\hat{a})   - a \right )  \leq  \frac{1}{2}  \teal{\Delta}.		
	\end{align*}}
	Similarly, the same can be shown when $ q - \beta+2 $ is even.
	\item[(i)] Same as $(ii)$ in \cref{lemma:truncneg}.
\end{itemize}
\end{proof}	

Both \cref{lemma:truncneg_postcompression} and \cref{lemma:truncneg_precompression} illustrate that the rounding schemes shift the negabinary values such that when the truncation is applied, the error is centered around zero \teal{provided that the truncation is applied for sufficiently small values of $\Delta$.}
For both rounding schemes, the implications for \cref{thm:bias} result in the expected error to be 
\teal{\begin{align} \label{eqn:biascorrected}
	{\mathbb{E}\left[ \tilde{D}(\tilde{C}(\teal{\bx}))  - D(C(\teal{\bx}))  \right]} 
	=  	2^{\ell}L^{-1}_d  {\bf E}_d, 
\end{align}}
\hspace{-1em} 
where \teal{${\bf E}_d$ is the expected error from the forward transform operator defined by \cref{lemma:expectedtransform1d} and  \cref{thm:expectedtransform}.}
	Note that \cref{thm:bias} assumes that the leading bit for each value within the block is not truncated in Step 8. 
	As this assumption degrades, (ii) in \cref{lemma:truncneg_postcompression} and  \cref{lemma:truncneg_precompression} will begin to be present. This is especially an issue for {\it postcompression} as the rounding constant, $\pm \frac{1}{6} \Delta$, will be present in the expected error, causing the rounding scheme to be less resilient to the biasing effects.
	\teal{However, it is worth noting that, typically, $\mathbb{E}[a] = 0$ for coefficients other than the first one (the DC term), as $a$ approximately follows a Laplace distribution with zero mean\cite{lindstromdistrib}.}
%

\subsection{Synthetic $4^d$ Blocks} 
\label{sec:unbiasedgeneratedBlock}
As in the first numerical test, we wish to mimic the worst possible input for ZFP, i.e., a not smooth, uncorrelated block of values that can not take advantage of the properties of the forward \teal{decorrelating} transform. 
Using the same setup as in \Cref{sec:generatedBlock}, we generate $4^d$ blocks of highly oscillatory elements with $d= 1$. 
We now compare experimental mean error from the generated blocks using the biased variant and the {\it postcompression} and {\it precompression} rounding variants of ZFP. \cref{fig:precomp} depicts the experimental mean error from all generated blocks using the {\it precompression} and original variant of ZFP and a side-by-side comparison of the biased and {\it precompression} rounding variant scaled by $\beta$, i.e., $\bar{x}_{i,\beta}(2^\beta)$,  \teal{for $\rho=14$, where $\rho$ is the dynamic range of values in a block defined by \cref{eqn:rho}.}
Similar results are depicted for the {\it postcompression} rounding variant of ZFP in \cref{fig:postcomp}.  \cref{fig:precomp_new} and \cref{fig:postcomp_new} depict the same results but for when $\rho=0$.  Clearly, both rounding techniques have experimental mean errors that are orders of magnitude smaller than the biased variant. This is especially apparent as $\beta$ decreases. \cref{fig:precompression2d}, \cref{fig:postcompression2d}, \cref{fig:precompression3d}, and \cref{fig:postcompression3d} present the same results for when $d = 2$ and $d = 3$, respectively, and similar conclusions can be drawn. 
\vspace{5pt} 
 \begin{figure}
 	\centering
 	\begin{subfigure}[b]{\textwidth}
 		\includegraphics[width=\textwidth]{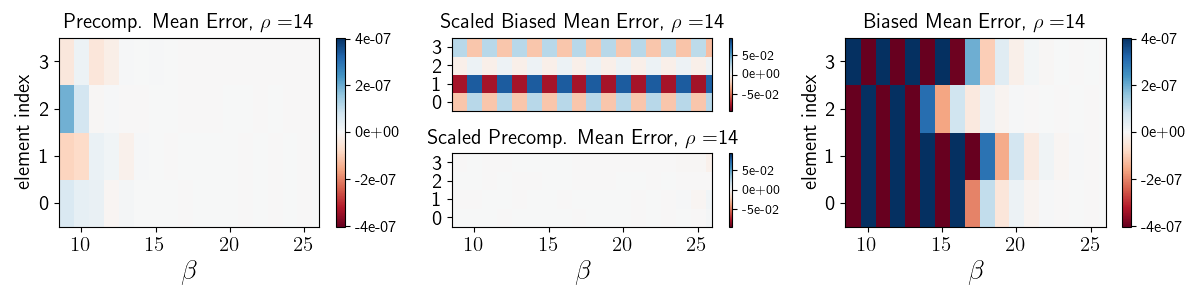}
\label{fig:precompression_1rho14}
 	\end{subfigure}

\caption{1-D Simulated {\it Precompression} Rounding Example: The left and right figures depict the unbiased and biased experimental mean error, respectively, using {\it precompression} and the original variant, while the middle figure shows a side-by-side comparison of the unbiased and biased scaled experimental mean error by $\beta$ for $\rho = 14$, where $\rho$ is the exponent range of values in a block defined by \cref{eqn:rho}.} 
 	\label{fig:precomp} 	
 \end{figure}

\vspace{5pt}
\begin{figure}
	\centering
	\begin{subfigure}[b]{\textwidth}
	\includegraphics[width=\textwidth]{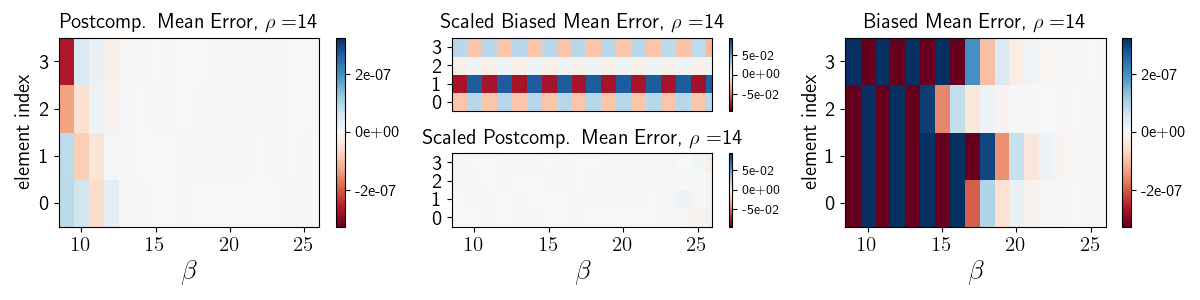}
	\label{fig:postcompression_1rho14f}
	\end{subfigure}
	
\caption{1-D Simulated {\it Postcompression} Rounding Example: The left and right figures depict the unbiased and biased experimental mean error, respectively, using {\it postcompression} and the original variant, while the middle figure shows a side-by-side comparison of the unbiased and biased scaled experimental mean error by $\beta$ for $\rho = 14$, where $\rho$ is the exponent range of values in a block defined by \cref{eqn:rho}.} 
	\label{fig:postcomp}
\end{figure}

\subsection{Climate Data Real-World Example} 
In this section, we repeat the experiments from \Cref{section:ncar} using the {\it precompression} and {\it postcompression} rounding  variants and compare them to the biased variant. The leftmost panels in \cref{fig:unbiased_ncar} present the \teal{mean compression error} of \teal{the} biased ZFP variant for $\beta = \{10,20\}$ while the middle  and right figures present the experimental mean error of the {\it precompression} variant and  {\it postcompression}  variant for each $\beta$. 
For each $\beta$ the magnitude of the mean error for the rounding variants is much smaller; however, one can see that as $\beta$ decreases, there is still indeed a bias with respect to the element index within the block\teal{, as predicted by our analysis in Equation \cref{eqn:biascorrected}.}
One interesting observation to note is the difference in the bias between the rounding schemes that can be seen in \cref{fig:unbiased_ncar} when $\beta = 10$. This difference can be explained by the difference in (ii) for \cref{lemma:truncneg_postcompression} and \cref{lemma:truncneg_precompression} when the leading bit is truncated. The {\it precompression} rounding variant is more resilient to the biasing effects as the assumptions in \cref{lemma:truncneg} are violated, and the addition(subtraction) of \teal{$\frac{1}{6} \Delta$} will cause bias for elements whose transform coefficients after truncation do not have a leading nonzero bit. 
\teal{This is because we generally expect $\mathbb{E}[a] = 0$, particularly for components other than the first one,  so any postcompression bias correction will introduce error.}
 \begin{figure}
	\centering
\begin{subfigure}[b]{\textwidth}
	\includegraphics[width=\textwidth]{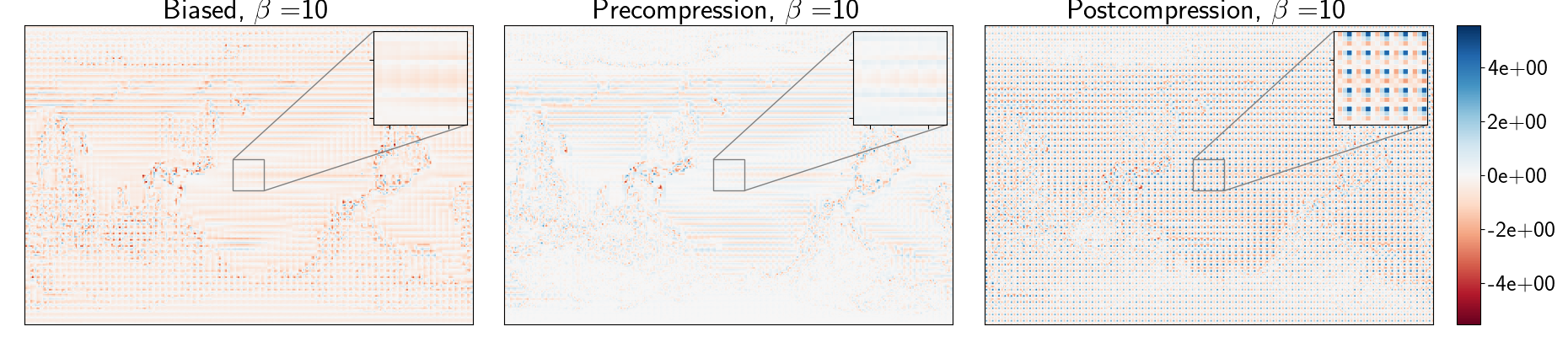}
	\label{fig:unbiasedncar10}
\end{subfigure}
\begin{subfigure}[b]{\textwidth}
	\includegraphics[width=\textwidth]{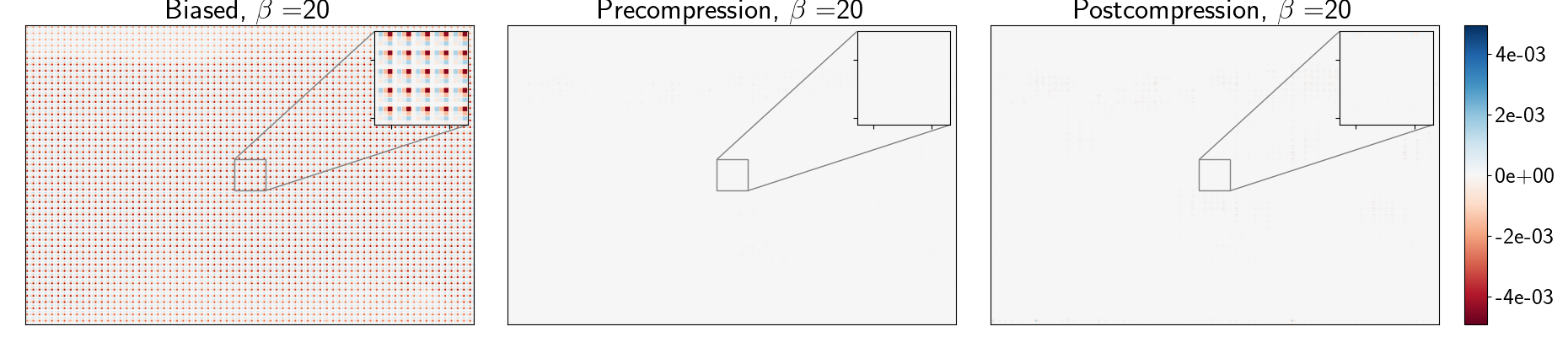}
	\label{fig:unbiasedncar20}
\end{subfigure}
	\caption{Climate Data Numerical Example: Each row depicts biased experimental mean error using ZFP 1.0.x, the unbiased experimental mean error using the {\it precompression} variant, and the unbiased experimental mean error using the {\it postcompression} variant for $\beta = \{ 10,20\}$, where $\beta$ is the number of bit planes kept in Step 8.}
	\label{fig:unbiased_ncar}
\end{figure}

\subsection{Autocorrelation Analysis} 
An additional quantity of interest that pertains to bias is the absence of autocorrelation in the error field. Autocorrelation is the correlation of a signal with a delayed copy of itself as a function of the delay. In 1-d, the delay is also known as a horizontal lag. 
\cref{fig:unbiased_ncar_auto} shows \teal{a 2-D slice of the 3-D autocorrelation function, $\mathcal{R}( \delta )$, computed for the compression error, corresponding to $\Delta t = 0$ (i.e., zero time lag) } for each rounding variant and the biased variant with zero lag. 
\teal{Here, $ \delta $ denotes the vector of integer lags in each dimension.}
In other words, the source \red{data is treated as a 3-d field}, with time on the z-axis and the autocorrelation function is computed with a zero lag in the time dimension. 
	Ideally, the autocorrelation function is a Dirac delta function at the center of the field with zero elsewhere. The optimal autocorrelation function occurs if there is no correlation between the error values and their neighboring values. 
The center pixel of each \teal{2-D slice has a value of one}, and each corresponding pixel quickly \teal{decays} to near zero, approximating the optimal autocorrelation function. 
As the precision increases to $\beta = 20$, depicted in the bottom row of \cref{fig:unbiased_ncar_auto},  the autocorrelation function for the {\it post}- and  {\it precompression} variants remain ideal, while the biased autocorrelation function begins to degrade. When $\beta = 10$, depicted in the top row of \cref{fig:unbiased_ncar_auto}, the \red{autocorrelation} function for the {\it postcompression} variant degrades. 
This is \red{again due} to the  difference in (ii) for \cref{lemma:truncneg_postcompression} and \cref{lemma:truncneg_precompression} when the leading \teal{one-}bit is truncated. The {\it postcompression} rounding variant violates  assumptions in \cref{lemma:truncneg}, and the addition(subtraction) of \teal{$\frac{1}{6} \Delta$} will cause bias for elements whose transform coefficients do not have a leading nonzero bit. Lastly, \cref{fig:autocor_beta_both} depicts the 2-norm of the autocorrelation function, $\|R(\delta)\|$, as a function of the precision, $\beta$, on a \teal{linear} and log scale for each rounding variant. 
Note that \cref{fig:unbiased_ncar_auto} and \cref{fig:autocor_beta_both} used only the first 368 days to produce the figures.  

\vspace{5pt}
\begin{figure}
	\centering
	\begin{subfigure}[b]{\textwidth}
		\includegraphics[width=\textwidth]{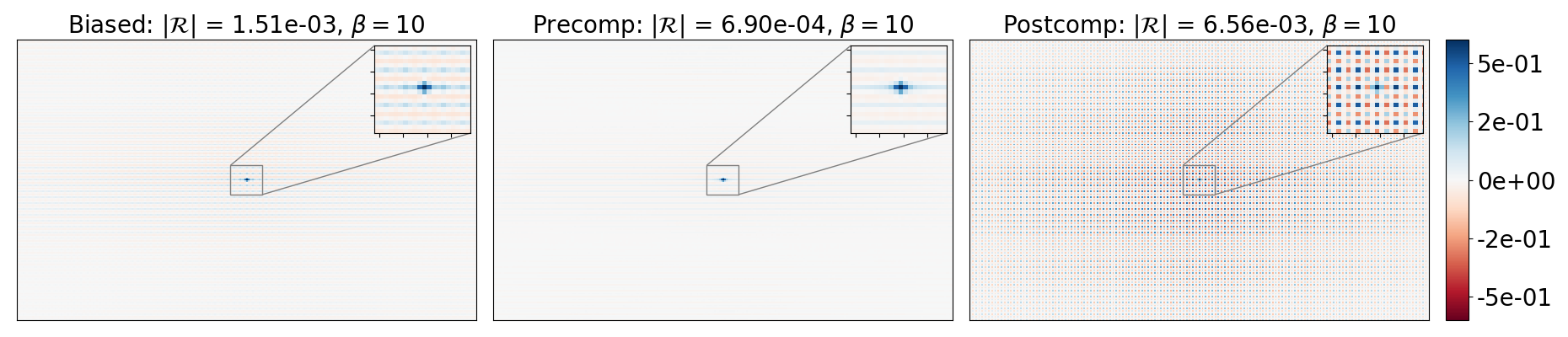}
		\label{fig:ncar10_auto}
	\end{subfigure}
	\begin{subfigure}[b]{\textwidth}
		\includegraphics[width=\textwidth]{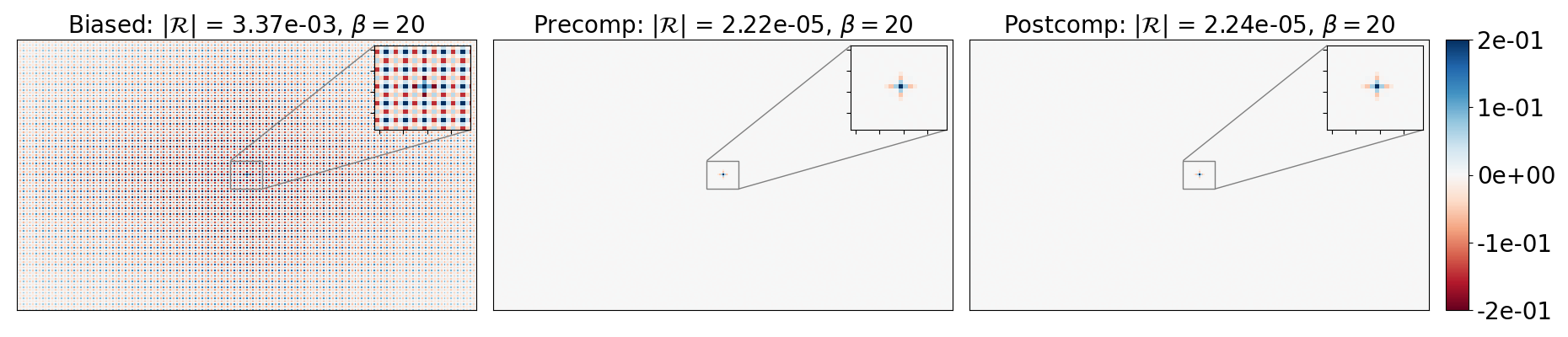}
		\label{fig:ncar20_auto}
	\end{subfigure}	
	\caption{Climate Data Numerical Example: Each figure depicts the 2-d slice of the autocorrelation function \teal{$\mathcal{R}( \delta  )$} of the mean error for $\beta = \{10,20\}$, i.e., the $\Delta t = 0$ slice of the 3-d autocorrelation field, using ZFP 1.0.x, the {\it precompression} variant, and the {\it postcompression} variant, respectively.} 
	\label{fig:unbiased_ncar_auto}
\end{figure}

\vspace{5pt}
\begin{figure}
	\centering
	\begin{subfigure}[b]{.8\textwidth}
		\includegraphics[width=\textwidth]{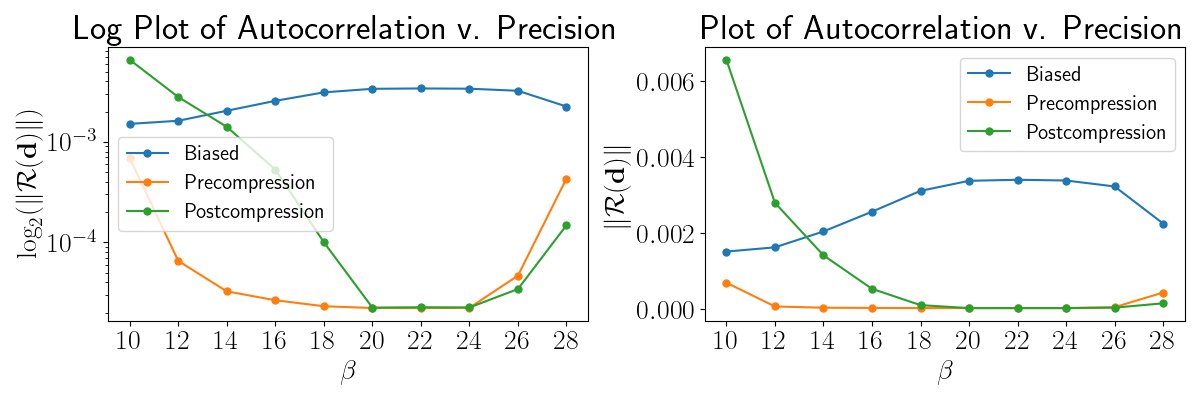}
		\label{fig:autocor_beta}
	\end{subfigure}
	\caption{Climate Data Numerical Example: Figure (\red{right}) depicts the 2-norm of the autocorrelation, $\|\teal{\mathcal{R}(\delta )}\|$, of the mean error as a function of precision, $\beta$. Figure (\red{left}) \teal{plots the same function on a vertical log axis.}}
	\label{fig:autocor_beta_both}
\end{figure}

\section{Empirical Error Distributions}
\label{sec:distribution}

We conclude our experiments with an investigation of how the compression errors due to quantization in Step~8 are distributed, both theoretically and empirically.
By far, these tend to be the dominant source of error in ZFP.
Following our assumption that discarded trailing bits of transform coefficients are uniformly random, quantization errors are thus uniform either on $\left(-\frac{2}{3}, \frac{1}{3}\right)$~ulps (unit in the last place) or $\left(-\frac{1}{3}, \frac{2}{3}\right)$~ulps.
\teal{Here, we define 1 ulp ($\Delta$) as the magnitude of the least significant bit of the quantized representation.}
The linear decorrelating inverse transform gives a weighted average of these uniform error terms that tends toward Gaussian distributions, as previously observed in~\cite{lindstromdistrib}, and in one dimension is piecewise cubic.
The closed-form expressions are easily found via convolution and are presented in \Cref{sec:diststats}.
Due to the negabinary quantization errors being biased and because of the slight nonorthogonality of the decorrelating transform, the actual error distribution in one dimension ($d = 1$) varies spatially with index $i \in \{1, 2, 3, 4\}$ within the block, and each of the four distributions gives rise to either positive or negative bias.
As discussed earlier, this bias can be corrected using proper rounding, e.g., by offsetting coefficients before truncating them.
\vspace{5pt}
\begin{figure}[tp]
\centering%
\hspace{0.5in}\includegraphics[width=0.5\linewidth]{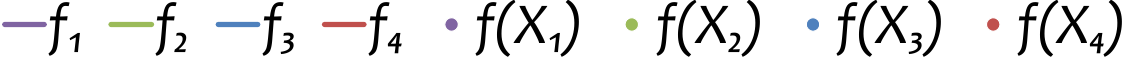}%
\\[2ex]%
\includegraphics[height=2in]{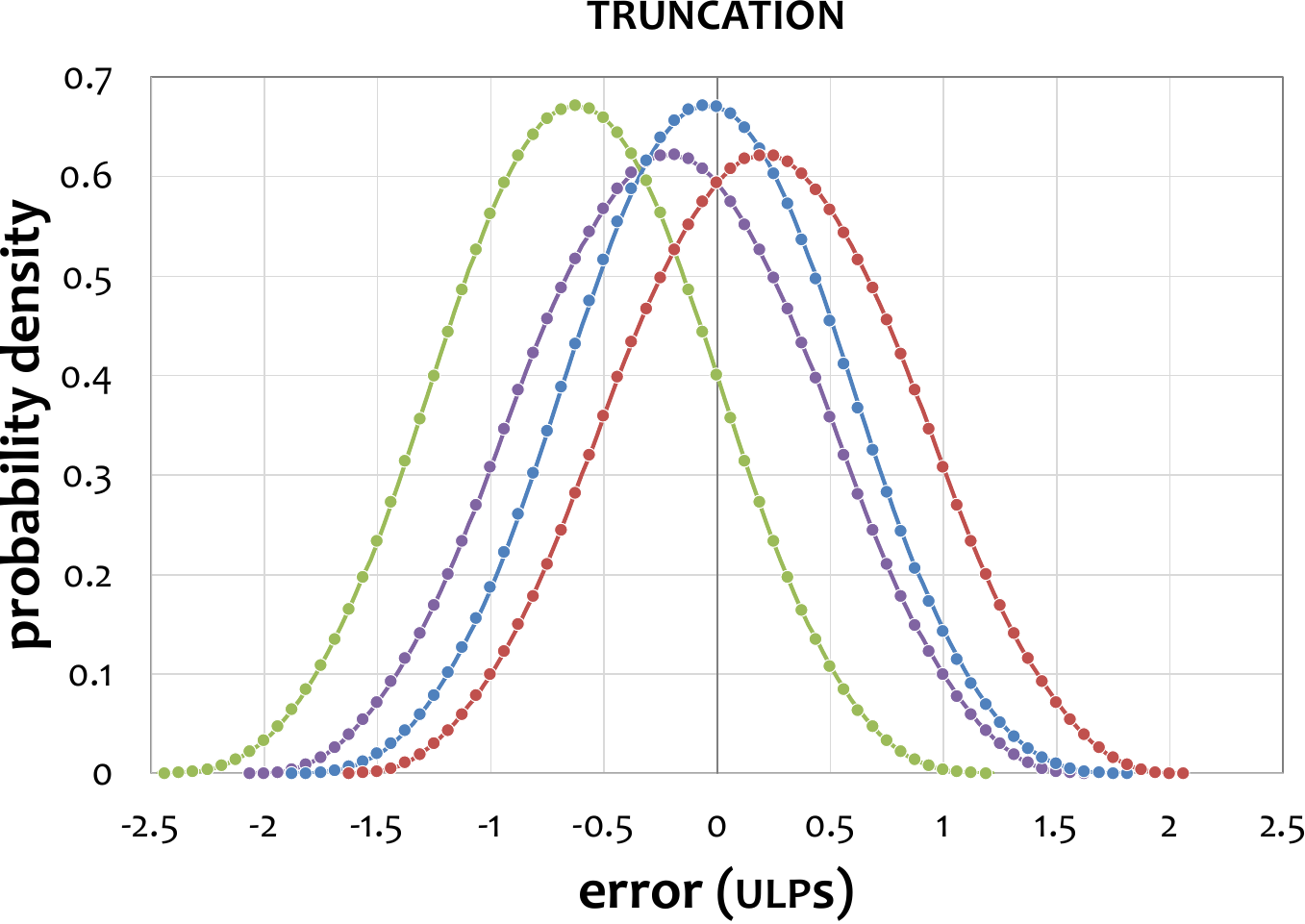}%
\hfill%
\includegraphics[height=2in]{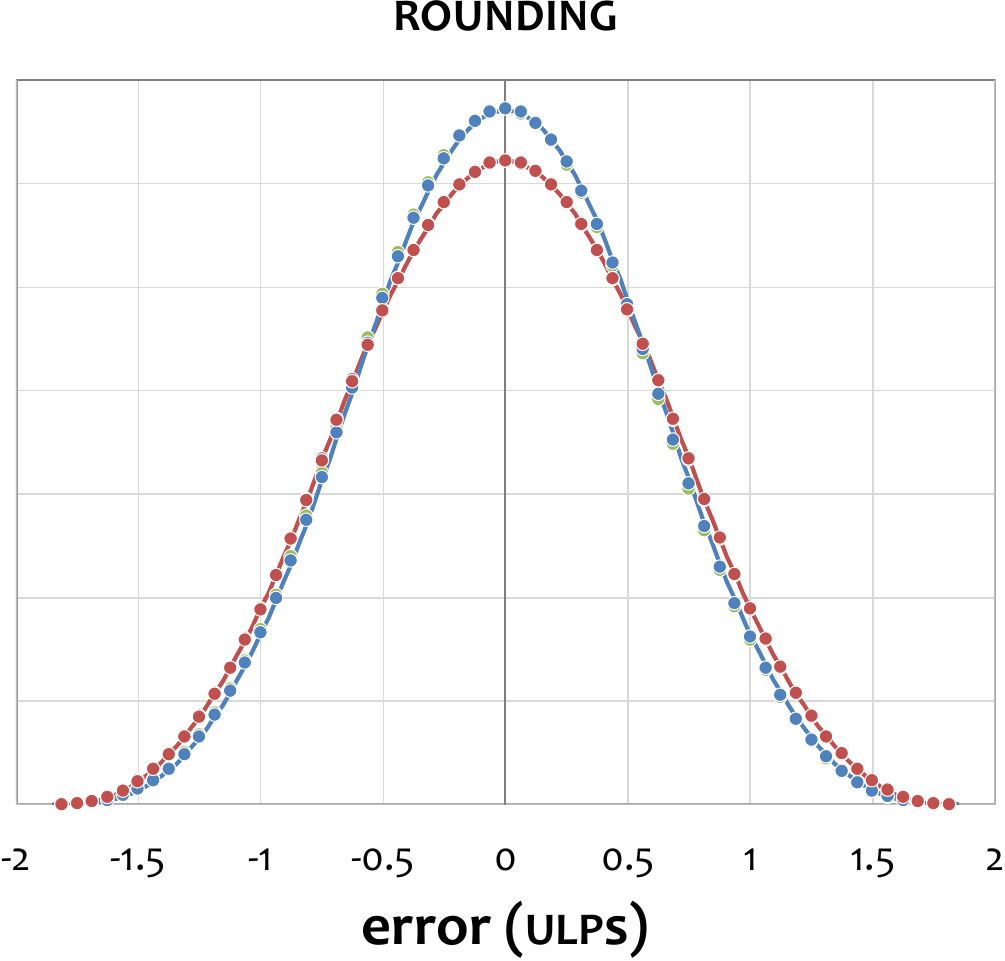}%
\caption{Error distributions due to coefficient truncation (left) and rounding (right) for 1D ZFP compression.
The four distributions each correspond to random variables associated with one of four spatial locations within a block.
\teal{The empirical distributions (shown as dots) align remarkably well with what theory predicts (curves).}
}
\label{fig:distributions}
\end{figure}

To compare theoretical and observed error distributions, we performed an experiment using the climate data first described in \Cref{section:ncar}.
This data set is composed of $31,390$ daily averages of surface temperature on a $288 \times 192$ lat/lon grid.
Because the fastest-varying dimension is a multiple of four, we may simply reshape the data as a one-dimensional \teal{vector such that each block corresponds} 
to four adjacent grid points at the same latitude.
Because the sign of the error depends on the parity of number of truncated bits, as given by the error tolerance and per-block common exponent, we considered only those blocks whose maximum value fell in $[2^8, 2^9)$~kelvins, which make up just over 80\% of all blocks.

\Cref{fig:distributions} shows excellent agreement between theory and observation, both with and without bias correction, as the dots (empirical densities given by the ratio of bin probability to bin width) coincide with the curves (theoretical densities).
This figure and \Cref{tab:diststats} further validate the efficacy of our bias correction scheme, as the observed error distributions have zero mean.
It can be seen that the biased distributions vary both in their position and shape (e.g., amplitude and variance), with $\bigvecx_1$ and $\bigvecx_4$ distributions having the same shape, and similarly for $\bigvecx_2$ and $\bigvecx_3$.
As detailed further in \Cref{sec:diststats}, $\bigvecx_i$ mean and variance are governed by the row sums and norms of the inverse decorrelating transform $L^{-1}$.
A transform $L$ with orthogonal rows (with $\frac{1}{16} \begin{bmatrix}6 & 2 & -2 & -6\end{bmatrix}$ as second row) coupled with bias correction would result in i.i.d.\ distributions, albeit at the expense of higher computational cost.

\section{Conclusion}
\label{sec:conclusion}
In this paper, we analyzed the bias of the error introduced in the use of lossy ZFP compression of floating-point data. This paper's significant contribution is the theoretical proof of the bias using the \teal{compression operators, which refer to the individual steps of the ZFP algorithm, as defined in Diffenderfer \cite{errorzfp}. These operators act on the bit vector space $\cB^n$, allowing us} to critically analyze the bias resulting from each step of the ZFP algorithm \teal{as well as the composite operator, which encompasses all compression steps.} \Cref{sec:results} presented numerical experiments to test the accuracy of the theoretical bias in a simulated example as well as a real-world example. \Cref{sec:correction} presented two correction methods \teal{to cancel the bias} that involve a simple rounding step. The {\it postcompression} rounding variant was first introduced in \cite{Hammerling}, while the more effective {\it precompression} variant is introduced in this paper. We note that the {\it precompression} variant can only be applied when the number of bit-planes during truncation is known, i.e., this mode is unavailable to the fixed-rate mode of ZFP.
\blue{While our focus was specifically on ZFP, we acknowledge that bias is a known issue in other compression algorithms as well, as discussed in \cite{lindstromdistrib} and other studies. Our work is informed by this broader context, and we are aware that similar biases may exist in other algorithms.  That said, our primary aim was to provide a detailed analysis of ZFP, an algorithm whose bias has not been extensively covered in the literature.}

 The theoretical bias determined in this paper is limited by the assumptions on the input distributions. \teal{While these limitations exist, our results provide a framework for assessing whether the magnitude of the bias will significantly impact the conclusions of a statistical analysis, depending on the application.}
 Even with this deviation from the predicted expected error, \cref{thm:bias} indicates the worst-case scenario, and the application can determine if the magnitude of the bias is acceptable. 
\teal{If the magnitude of the bias is not acceptable, the correction methods introduced in \Cref{sec:correction} offer practical solutions. These methods, {\it postcompression} and {\it precompression} rounding, build directly on the theoretical framework developed in this paper and demonstrate how the bias can be drastically reduced. }
 We modify \cref{lemma:truncneg} for each scheme and demonstrate that the bias can be drastically reduced. 
 We also showed that  {\it precompression} rounding is more resilient to the biasing effect as the assumptions in \cref{thm:bias} are no longer satisfied.
 \teal{By tying these correction methods to the theoretical analysis, this paper provides both a detailed understanding of the bias and actionable tools for mitigating it.
 	While there is room for further refinement of the theoretical framework, this paper provides a solid foundation for understanding and mitigating bias in ZFP compression.}

Our results show that it is indeed possible to statistically analyze the error caused by a compression algorithm. Our analysis is the first of its kind to attempt this theoretical approach. \blue{We hope that by using the vector space, $\cB^n$, additional research can further our understanding of the error caused by finite bit representations, including those introduced by other compression algorithms.}

\section*{Acknowledgments}
This work was performed under the auspices of the U.S. Department of Energy by
Lawrence Livermore National Laboratory under Contract
DE-AC52-07NA27344.

\newpage
\bibliographystyle{siamplain}
\bibliography{compression}

\newpage
\setcounter{page}{1}
\appendix

\section{Uniformly Random Bits}
\label{apd:random}

In \Cref{section:biastools}, it was assumed that the trailing bits after the leading non-zero bit in a negabinary representation are uniformly random, i.e., each trailing bit has an equal probability of being either a zero or a one.
	Specifically, we are interested in understanding the distributions of the trailing bits for each of the transform coefficients. 
	To validate this assumption for the transform coefficients, we empirically tested our theory. In the following, the data set that is used was formed by sampling 32 thousand 3D blocks from 32 different data sets, resulting in over 1 million total blocks. The sample data sets are from various scientific simulations. 
	Each block was then compressed by Step~2 through Step~5 so that each transform coefficient is ordered by total sequency and represented in negabinary. 
	\cref{fig:precent10} and \cref{fig:precent57} each 
represent a transform coefficient length of \teal{at most} 10 and 57, respectively, i.e., the transform coefficient $a \in \cN$ has a length of $10$, \teal{meaning $0 \leq a \leq 2^{10} - 1$, such that $\max{\cI(a)} \leq 9$ and $\min{\cI(a)} \geq 0$} 
Each column represents a coefficient index. As we are studying 3D blocks, there are 64 coefficients with a starting index of 0. 
The rows represent the trailing bits, with the least significant bits at the top. 
The color map and value represent the percentage that \teal{$a_i$ (where $i$ denotes the $i$-th row of the color map)} is a one-bit. 
It can be seen in \cref{fig:precent10} that the \teal{most} significant bits have a much higher probability of being zero. 
Due to the block-floating point transform in Step~2, there is a high probability that the inputs into the transformation have trailing zeros. 
This is due to the precision differences between the input data type and the block floating point representation, i.e., we typically have $q>k$. 
The transformation propagates the zero bits through arithmetic operations. However, if the block has a small dynamic range, it is likely that not all the trailing zero bits will be operated on. 
Thus, the least significant bits have a high probability of being zero. 
As the width of the coefficient increases, this phenomenon is less likely, as can be seen in \cref{fig:precent57}. 
Our assumption in Step~8 is that at least $2d$ bit planes are discarded, removing these bits from the analysis.
From our empirical results, we assume for our analysis that the resulting bits that could be truncated are uniformly random. 

\begin{figure}[h!]
	\centering
	\begin{subfigure}[b]{0.4\textwidth}
	\includegraphics[width=\textwidth]{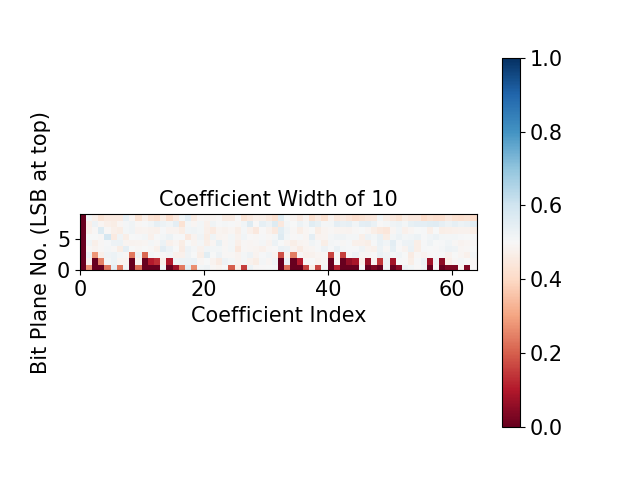}
		\caption{Coefficient Width of 10}
	\label{fig:precent10}
	
	\end{subfigure}
	\begin{subfigure}[b]{0.4\textwidth}
		\includegraphics[width=\textwidth]{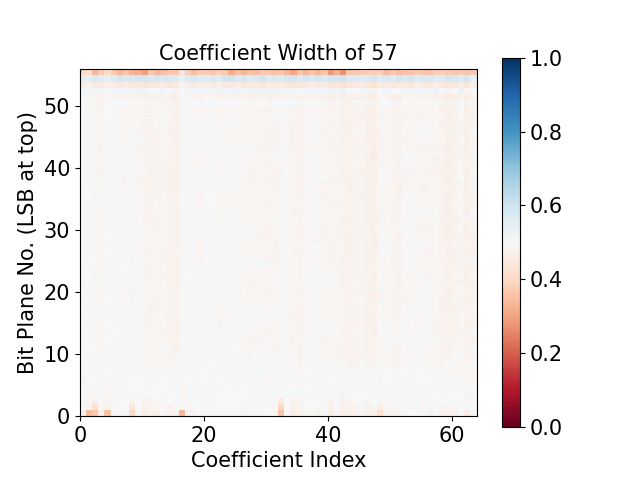}
	\caption{Coefficient Width of 57}
\label{fig:precent57}

\end{subfigure}
	\caption{The row of each color-map represents the trailing bits, with the most significant bits at top, while each column represents a coefficient index.
		The color map and value represent the percentage that the transform coefficient is a one-bit. }
\end{figure}
%
%
%
%
%
%
%

\section{Demonstration of \cref{lemma:expectedtransform1d}}
\label{sec:demoEd}
In this section, we demonstrate the validity of  \Cref{lemma:expectedtransform1d} to predict the estimated error caused by the lossy transform operator. 
Define the distribution ${\bf X }$ such that $\{\bx \in \mathbb{R}^4\} \teal{\ \sim \ } {\bf X }$ is a vector of integers whose elements are drawn from a uniform distribution $\mathcal{U}{(-2^{30}, 2^{30})}$. Let ${\bf A}$ be a uniform distribution that maps the elements of ${\bf X}$ to the infinite bit-vector space, i.e.,  for every  $\bx \teal{\ \sim\ } {\bf X}$ we have $\ba \teal{\ \sim \ } {\bf A}$ such that  $\bx = F_\cB(\ba)$. 
\cref{fig:Delta} demonstrates the accuracy of our defined $\forward_1$. 
The y-axis depicts the probability mass.
The bottom row is the experimental distribution of the error from 10,000 trials.  
For each $\ba \teal{\ \sim \ } {{{\bf A}}}$, the lossy and lossless forward transform operator is applied, and the difference for each element is stored. 
The bottom row depicts the histogram of the  distribution of the error for each component from 10,000 trials.  
The top row depicts the exact distribution as defined in \cref{eqn:transformlossyerror}. 
As can be seen, the experimental distribution follows the exact distribution as defined by \cref{lemma:expectedtransform1d}. 
\teal{Similarly, \cref{fig:LinvDelta} demonstrates the error resulting from the composite process of applying the forward transform followed by the inverse transform for both the lossy and lossless cases.}
The bottom row is the experimental distribution of the error from 10,000 trials and the top row depicts the exact distribution. 
\vspace{5pt}
\begin{figure}
	\centering
	\begin{subfigure}[b]{0.24\textwidth}
		\includegraphics[width=\textwidth]{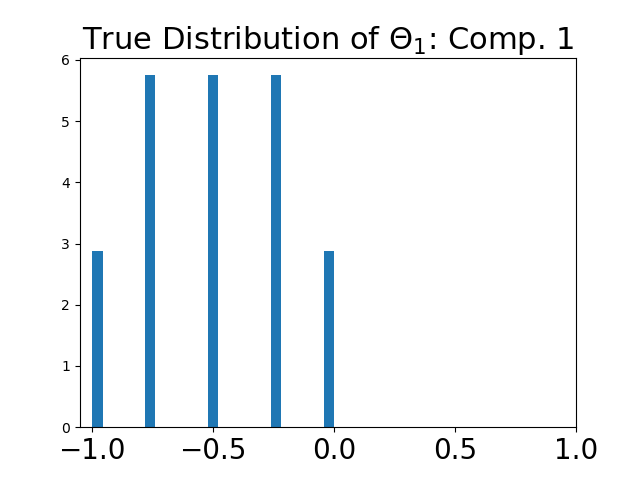}
	\end{subfigure}
	\begin{subfigure}[b]{0.24\textwidth}
		\includegraphics[width=\textwidth]{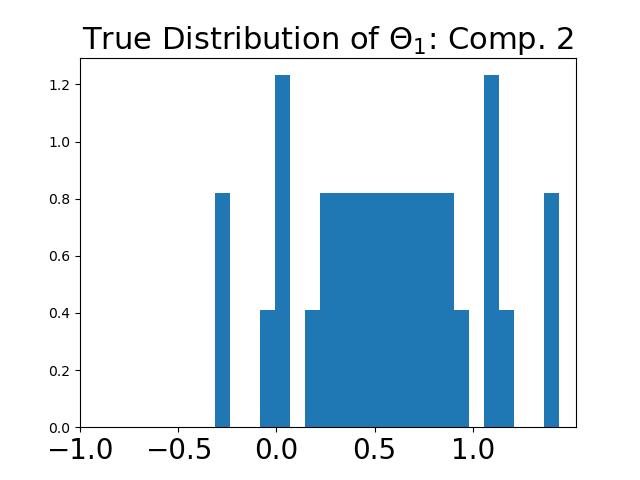}
	\end{subfigure}
	\begin{subfigure}[b]{0.24\textwidth}
		\includegraphics[width=\textwidth]{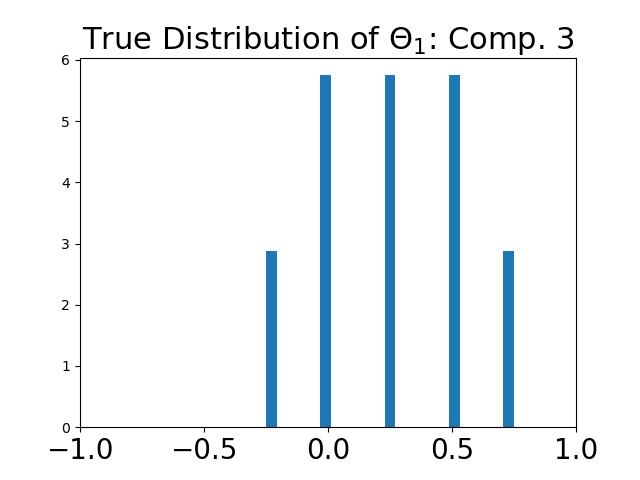}
	\end{subfigure}
	\begin{subfigure}[b]{0.24\textwidth}
		\includegraphics[width=\textwidth]{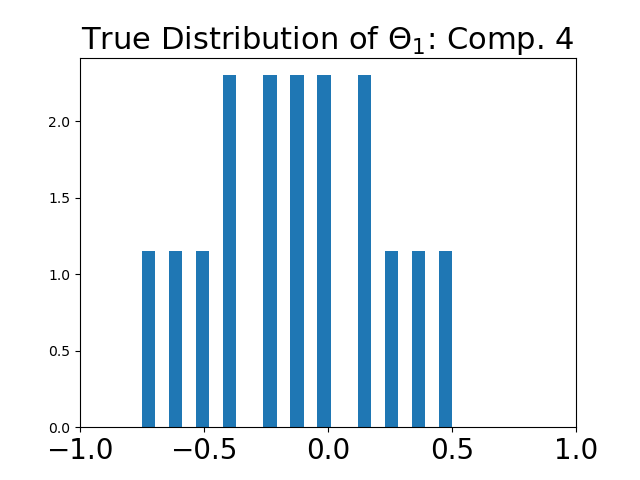}
	\end{subfigure}
	\begin{subfigure}[b]{0.24\textwidth}
		\includegraphics[width=\textwidth]{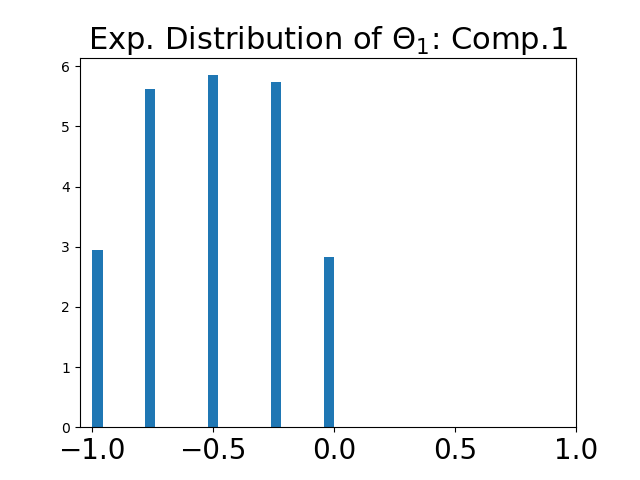}
	\end{subfigure}
	\begin{subfigure}[b]{0.24\textwidth}
		\includegraphics[width=\textwidth]{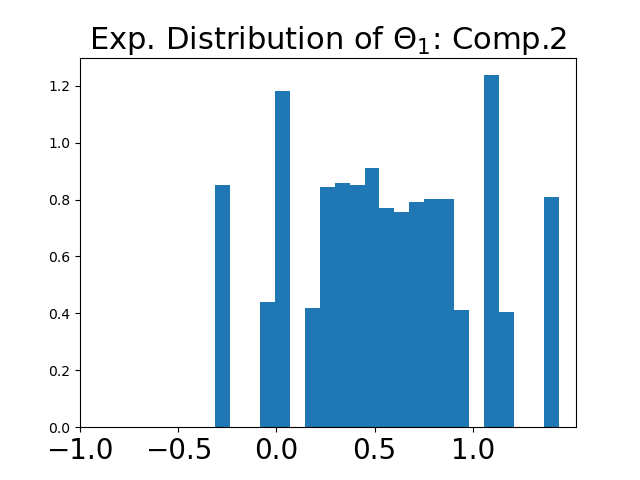}
	\end{subfigure}
	\begin{subfigure}[b]{0.24\textwidth}
		\includegraphics[width=\textwidth]{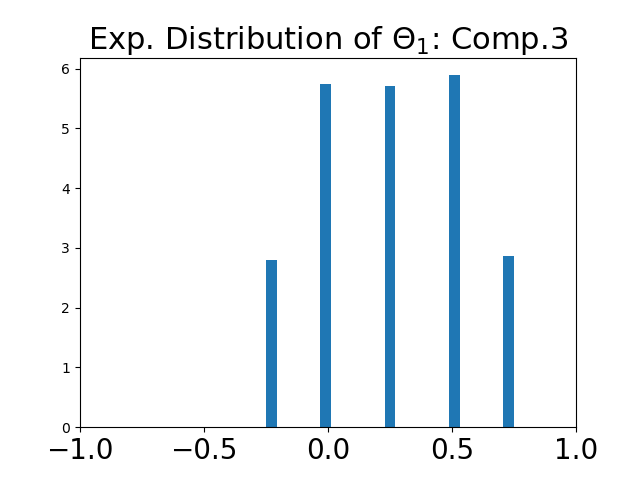}
	\end{subfigure}
	\begin{subfigure}[b]{0.24\textwidth}
		\includegraphics[width=\textwidth]{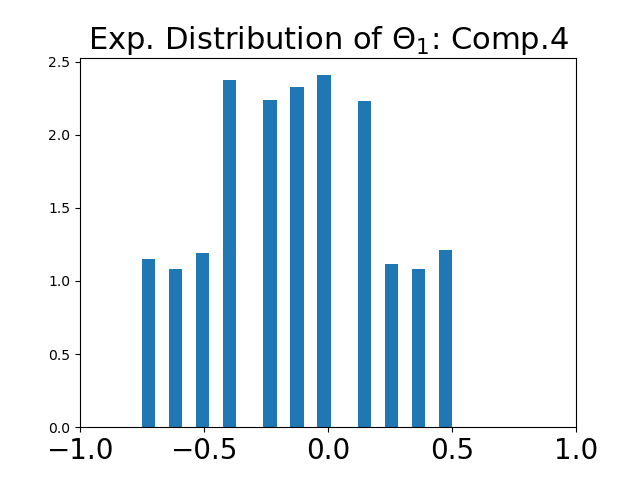}
	\end{subfigure}

	\caption{Demonstration of $\forward_1$. 
		The top row depicts the true distribution of the error as defined in \cref{eqn:transformlossyerror} for each element $i = \{1,2,3,4\}$, from left to right. 
		The bottom row depicts the experimental error distribution from 10,000 trials for each element, respectively.}
	\label{fig:Delta}
\end{figure}
\vspace{5pt}
\begin{figure}
	\centering
	\begin{subfigure}[b]{0.24\textwidth}
		\includegraphics[width=\textwidth]{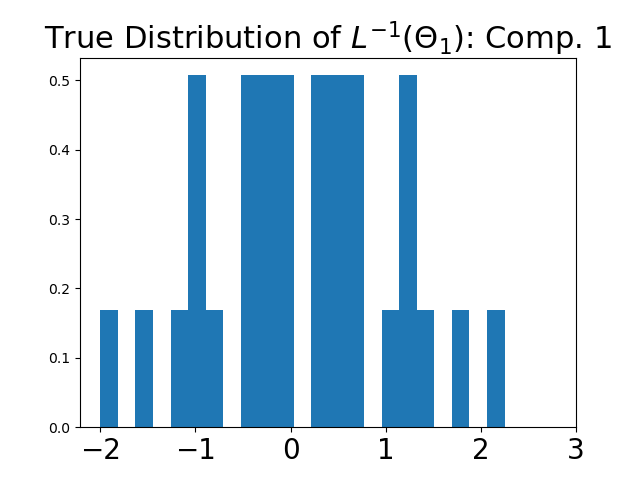}
	\end{subfigure}
	\begin{subfigure}[b]{0.24\textwidth}
		\includegraphics[width=\textwidth]{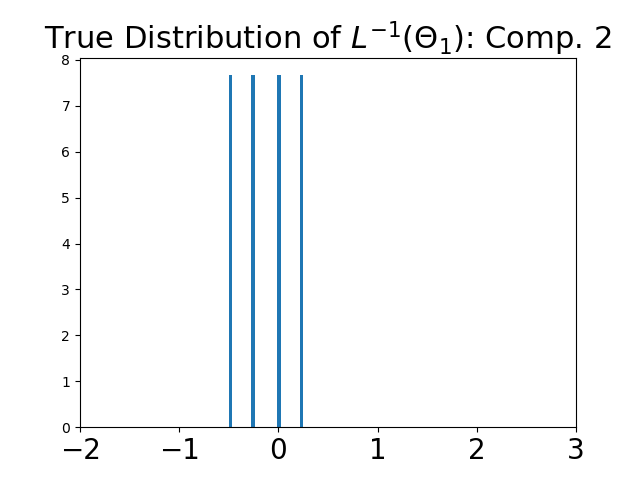}
	\end{subfigure}
	\begin{subfigure}[b]{0.24\textwidth}
		\includegraphics[width=\textwidth]{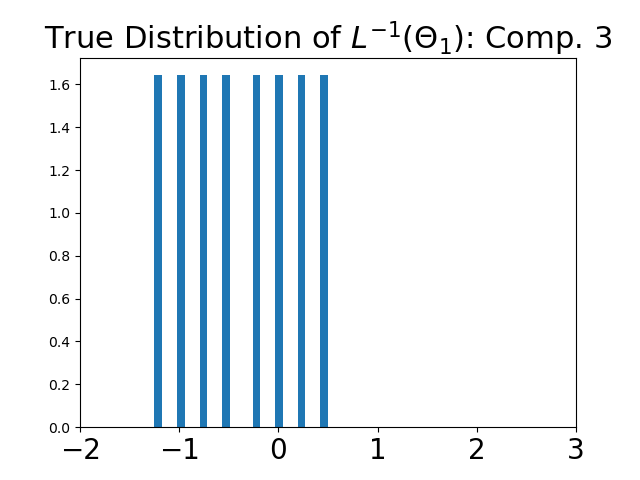}
	\end{subfigure}
	\begin{subfigure}[b]{0.24\textwidth}
		\includegraphics[width=\textwidth]{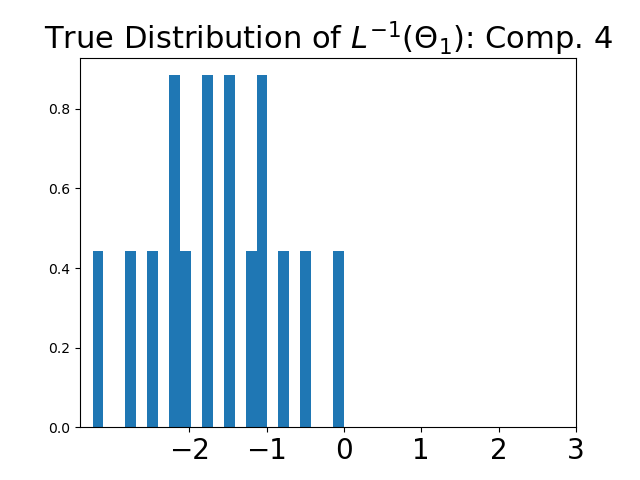}
	\end{subfigure}
	\begin{subfigure}[b]{0.24\textwidth}
		\includegraphics[width=\textwidth]{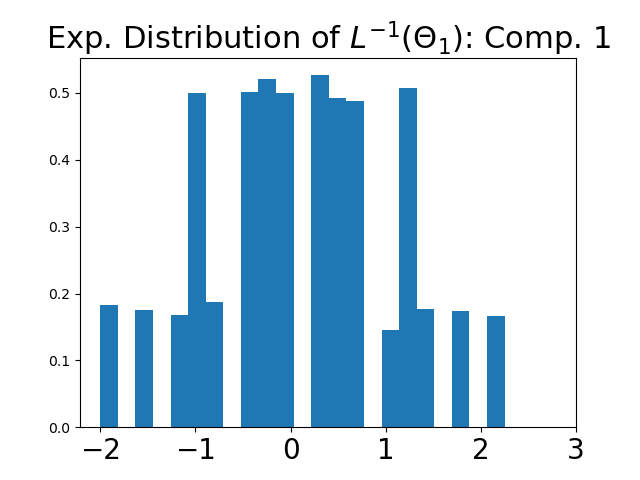}
	\end{subfigure}
	\begin{subfigure}[b]{0.24\textwidth}
		\includegraphics[width=\textwidth]{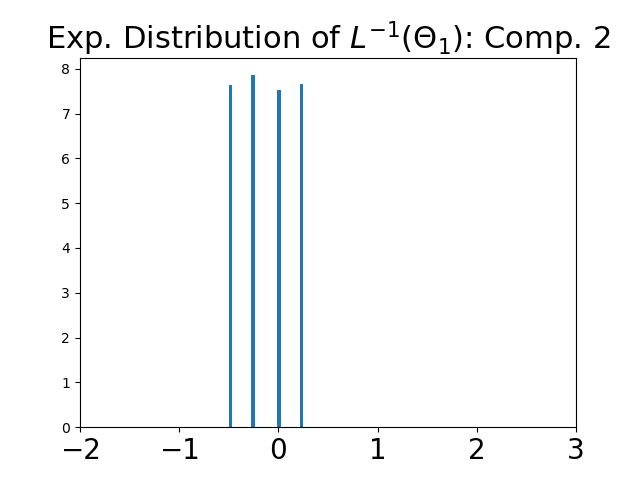}
	\end{subfigure}
	\begin{subfigure}[b]{0.24\textwidth}
		\includegraphics[width=\textwidth]{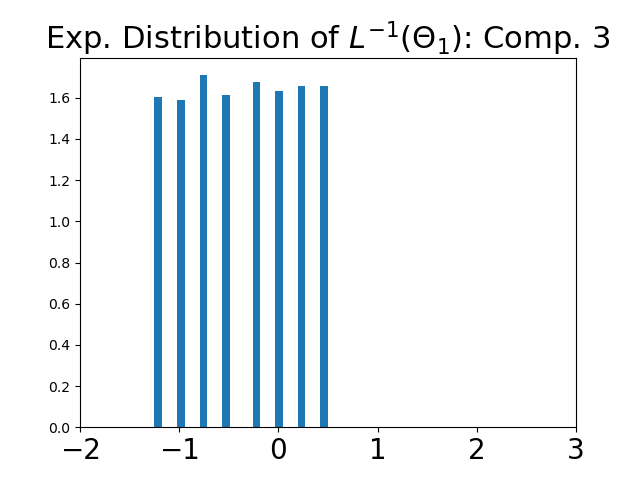}
	\end{subfigure}
	\begin{subfigure}[b]{0.24\textwidth}
		\includegraphics[width=\textwidth]{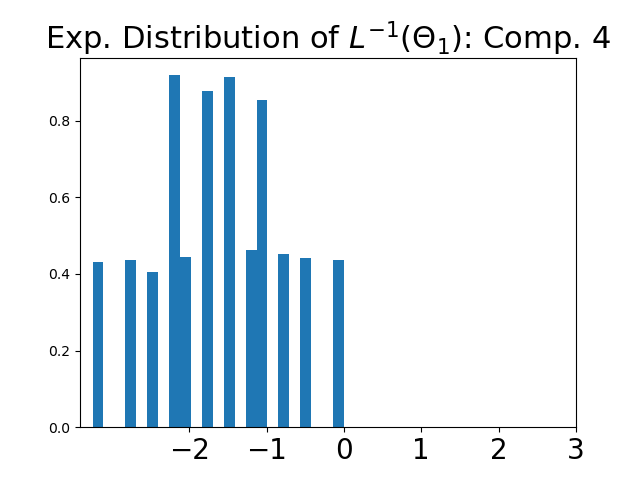}
	\end{subfigure}

	\caption{Demonstration of $L^{-1}(\forward_1)$. 
		The top row depicts the lossless backwards transform operator applied true distribution of the error as defined in \cref{eqn:transformlossyerror} for each element $i = \{1,2,3,4\}$, from left to right. 
		The bottom row depicts the lossless backwards transform operator applied the experimental error distribution from 10,000 trials for each element, respectively.}
	\label{fig:LinvDelta}
\end{figure}

\section{Quantization Error Distributions}
\label{sec:diststats}
\teal{In this appendix, we analyze the error distributions resulting from the quantization of negabinary transform coefficients. The error distributions resulting from the quantization of negabinary transform coefficients are presented here in closed form.}
We focus primarily on the one-dimensional ($d = 1$) case, where we obtain different error distributions for each spatial location $i \in \{1, 2, 3, 4\}$ within a block.
As before, we assume that coefficient quantization errors are uniform i.i.d.\ random variables $\vecy$, such that $\vecy \in \mathbb{R}^{4^d}$. 
Let $\Delta$ denote the unit in the last place---the quantization step---and $\uniform(a, b)$ denote the uniform distribution on the interval $(a, b)$.
Without bias correction, we have two cases: either ${\vecy_i}^{\text{even}} \sim \uniform(-\frac{2}{3} \Delta, \frac{1}{3} \Delta)$ or ${\vecy_i}^{\text{odd}} \sim \uniform(-\frac{1}{3} \Delta, \frac{2}{3} \Delta)$ for all $i \in \{1,...,4^d\}$, depending on whether an even or odd number of least significant bits are discarded, respectively.
Because the even and odd cases are symmetric and differ only in sign, we will focus only on the even case and drop the superscript.
Additionally, we present results for $\Delta = 1$  as $\vecy_i$ are scaled uniformly by $\Delta$.

The errors, $\vecy_i$, in coefficients are mixed by the inverse decorrelating linear transform, $L^{-1}$ (see \Cref{eqn:T}), resulting in piecewise cubic error distributions, in the canonical basis (i.e., in the decompressed field values).
Let
\begin{equation*}
\vecx = \begin{bmatrix}\vecx_1 & \vecx_2 & \vecx_3 & \vecx_4\end{bmatrix}^t 
\end{equation*}
Due to linearity of expectation,
\begin{equation*}
  {\mathbb{E}}[\vecx] = {\mathbb{E}}[\vecy] L^{-1} \mathbf{1} \qquad
  \var(\vecx) = \var(\vecy) (L^{-1} \circ L^{-1}) \mathbf{1},
\end{equation*}
where ${\mathbb{E}}[\vecy_i] = \teal{\pm}\frac{1}{6}$, $\var(\vecy_i) = \frac{1}{12}$, and $\circ$ denotes Hadamard (element-wise) product.
\teal{Here, we assume $\tilde{L}^{-1} = L^{-1}$, which holds when the least significant zero-bits as the right bit shifts then occur only on least significant zero-bits, preserving the transform.}
That is, the expected value of the $i$-th element in the vector distribution, ${\mathbb{E}}[\vecx_i]$, is proportional to the $i^{\text{th}}$ row sum of $L^{-1}$; the variance of the $i$-th element in the vector distribution, $\var(\vecx_i)$, is proportional to the square 2-norm of the $i^{\text{th}}$ row of $L^{-1}$.
Of course, when bias correction is applied, {$\vecy_i \sim \uniform(-\frac{1}{2}, \frac{1}{2})$ for all $i$} and ${\mathbb{E}}[\vecx_i] ={\mathbb{E}}[\vecy_i] = 0$, though $\var(\vecx_i)$ remains the same as in the biased case.
The probability density of $\vecx_i$ can be parameterized as
\teal{
\begin{equation}
\begin{split}
  f_i(x) = s_i \sum_{j=1}^4 \Bigl[
    &\Bigl(|x - (c_i - u_{i,j})|^3 + |x - (c_i + u_{i,j})|^3\Bigr) - \\
    &\Bigl(|x - (c_i - v_{i,j})|^3 + |x - (c_i + v_{i,j})|^3\Bigr)
  \Bigr],
\label{eqn:error-distribution}
\end{split}
\end{equation}
}%
%
where $s_i$ is a scaling factor, $c_i = {\mathbb{E}}[\vecx_i]$ is the center of $\vecx_i$, and $\{u_{i,j}\}$ and $\{v_{i,j}\}$ are ``knots'' that define the intervals of the piecewise cubic $f_i(x)$, which is zero outside the support $\supp(\vecx_i) = (c_i - u_{i,4}, c_i + u_{i,4})$.

\begin{table}[tp]
\centering
\setlength{\tabcolsep}{3pt}
\renewcommand*{\arraystretch}{1.25}
\begin{tabular}{rc|cccccccc|cccr@{,}l}
& $i$ & $s_i$ & $c_i$
  & $\hat{u}_{i,1}$ & $\genfrac{}{}{0pt}{1}{\hat{u}_{i,2}}{\hat{u}_{i,3}}$ & $\hat{u}_{i,4}$
  & $\hat{v}_{i,1}$ & $\genfrac{}{}{0pt}{1}{\hat{v}_{i,2}}{\hat{v}_{i,3}}$ & $\hat{v}_{i,4}$
  & ${\mathbb{E}}[|\vecx_i|]$ & ${\mathbb{E}}[\vecx_i^2]$ & $\var(\vecx_i)$ & \multicolumn{2}{c}{$\supp(\vecx_i)$} \\
\hline
\multirow{4}*{\rotatebox{90}{\emph{biased}}}
& $1$
  & $\frac{2}{9}$
  & $-\frac{5}{24}$
  & $1$ & $5$ & $15$
  & $3$ & $7$ & $13$
  & $\frac{577445}{1119744}$ & $\frac{29}{72}$ & $\frac{23}{64}$ & $(-\frac{25}{12}$&$\frac{5}{3})$ \\
& $2$
  & $\frac{2}{15}$
  & $-\frac{5}{8}$
  & $1$ & $3$ & $15$
  & $5$ & $7$ & $11$
  & $\frac{7}{10}$ & $\frac{17}{24}$ & $\frac{61}{192}$ & $(-\frac{5}{2}$&$\frac{5}{4})$ \\
& $3$
  & $\frac{2}{15}$
  & $-\frac{1}{24}$
  & $1$ & $3$ & $15$
  & $5$ & $7$ & $11$
  & $\frac{713183}{1555200}$ & $\frac{23}{72}$ & $\frac{61}{192}$ & $(-\frac{23}{12}$&$\frac{11}{6})$ \\
& $4$
  & $\frac{2}{9}$
  & $\hp\frac{5}{24}$
  & $1$ & $5$ & $15$
  & $3$ & $7$ & $13$
  & $\frac{577445}{1119744}$ & $\frac{29}{72}$ & $\frac{23}{64}$ & $(-\frac{5}{3}$&$\frac{25}{12})$ \\
\hline
\multirow{4}*{\rotatebox{90}{\emph{unbiased}}}
& $1$
  & $\frac{2}{9}$
  & $\hp0$
  & $1$ & $5$ & $15$
  & $3$ & $7$ & $13$
  & $\frac{90199}{184320}$ & $\frac{23}{64}$ & $\frac{23}{64}$ & $(-\frac{15}{8}$&$\frac{15}{8})$ \\
& $2$
  & $\frac{2}{15}$
  & $\hp0$
  & $1$ & $3$ & $15$
  & $5$ & $7$ & $11$
  & $\frac{70259}{153600}$ & $\frac{61}{192}$ & $\frac{61}{192}$ & $(-\frac{15}{8}$&$\frac{15}{8})$ \\
& $3$
  & $\frac{2}{15}$
  & $\hp0$
  & $1$ & $3$ & $15$
  & $5$ & $7$ & $11$
  & $\frac{70259}{153600}$ & $\frac{61}{192}$ & $\frac{61}{192}$ & $(-\frac{15}{8}$&$\frac{15}{8})$ \\
& $4$
  & $\frac{2}{9}$
  & $\hp0$
  & $1$ & $5$ & $15$
  & $3$ & $7$ & $13$
  & $\frac{90199}{184320}$ & $\frac{23}{64}$ & $\frac{23}{64}$ & $(-\frac{15}{8}$&$\frac{15}{8})$ \\
\hline
\end{tabular}
\caption{Parameters and statistics describing the biased (top half) and bias corrected (bottom half) theoretical error distributions $\vecx_i$ for a 1D block.
Here $i$ indicates the spatial position within the block,
\teal{with $s_i$ scaling the amplitude of the distribution;}
\teal{see \Cref{eqn:error-distribution}}.
$\hat{u}_{i,j} = 8 u_{i,j}$ and $\hat{v}_{i,j} = 8 v_{i,j}$.
The mean, or bias in error (in ulps), is given by ${\mathbb{E}}[\vecx_i] = c_i$.
}
\label{tab:diststats}
\end{table}

\C{
\begin{table}[htp]
\centering
\setlength{\tabcolsep}{3pt}
\begin{tabular}{c|cccccccc|cccr@{,}l}
$i$ & $s_i$ & $c_i$
  & $u_{i,1}$ & $u_{i,2} = u_{i,3}$ & $u_{i,4}$
  & $v_{i,1}$ & $v_{i,2} = v_{i,3}$ & $v_{i,4}$
  & ${\mathbb{E}}[|\vecx_i|]$ & ${\mathbb{E}}[\vecx_i^2]$ & $\var(\vecx_i)$ & \multicolumn{2}{c}{$\supp(\vecx_i)$} \\[3pt]
\hline
$1$
  & $\frac{2}{9}$
  & $-\frac{5}{24}$
  & $\frac{1}{8}$ & $\frac{5}{8}$ & $\frac{15}{8}$
  & $\frac{3}{8}$ & $\frac{7}{8}$ & $\frac{13}{8}$
  & $\frac{577445}{1119744}$ & $\frac{29}{72}$ & $\frac{23}{64}$ & $(-\frac{25}{12}$&$\frac{5}{3})$ \\[3pt]
$2$
  & $\frac{2}{15}$
  & $-\frac{5}{8}$
  & $\frac{1}{8}$ & $\frac{3}{8}$ & $\frac{15}{8}$
  & $\frac{5}{8}$ & $\frac{7}{8}$ & $\frac{11}{8}$
  & $\frac{7}{10}$ & $\frac{17}{24}$ & $\frac{61}{192}$ & $(-\frac{5}{2}$&$\frac{5}{4})$ \\[3pt]
$3$
  & $\frac{2}{15}$
  & $-\frac{1}{24}$
  & $\frac{1}{8}$ & $\frac{3}{8}$ & $\frac{15}{8}$
  & $\frac{5}{8}$ & $\frac{7}{8}$ & $\frac{11}{8}$
  & $\frac{713183}{1555200}$ & $\frac{23}{72}$ & $\frac{61}{192}$ & $(-\frac{23}{12}$&$\frac{11}{6})$ \\[3pt]
$4$
  & $\frac{2}{9}$
  & $\hp\frac{5}{24}$
  & $\frac{1}{8}$ & $\frac{5}{8}$ & $\frac{15}{8}$
  & $\frac{3}{8}$ & $\frac{7}{8}$ & $\frac{13}{8}$
  & $\frac{577445}{1119744}$ & $\frac{29}{72}$ & $\frac{23}{64}$ & $(-\frac{5}{3}$&$\frac{25}{12})$ \\[3pt]
\hline
$1$
  & $\frac{2}{9}$
  & $\hp0$
  & $\frac{1}{8}$ & $\frac{5}{8}$ & $\frac{15}{8}$
  & $\frac{3}{8}$ & $\frac{7}{8}$ & $\frac{13}{8}$
  & $\frac{90199}{184320}$ & $\frac{23}{64}$ & $\frac{23}{64}$ & $(-\frac{15}{8}$&$\frac{15}{8})$ \\[3pt]
$2$
  & $\frac{2}{15}$
  & $\hp0$
  & $\frac{1}{8}$ & $\frac{3}{8}$ & $\frac{15}{8}$
  & $\frac{5}{8}$ & $\frac{7}{8}$ & $\frac{11}{8}$
  & $\frac{70259}{153600}$ & $\frac{61}{192}$ & $\frac{61}{192}$ & $(-\frac{15}{8}$&$\frac{15}{8})$ \\[3pt]
$3$
  & $\frac{2}{15}$
  & $\hp0$
  & $\frac{1}{8}$ & $\frac{3}{8}$ & $\frac{15}{8}$
  & $\frac{5}{8}$ & $\frac{7}{8}$ & $\frac{11}{8}$
  & $\frac{70259}{153600}$ & $\frac{61}{192}$ & $\frac{61}{192}$ & $(-\frac{15}{8}$&$\frac{15}{8})$ \\[3pt]
$4$
  & $\frac{2}{9}$
  & $\hp0$
  & $\frac{1}{8}$ & $\frac{5}{8}$ & $\frac{15}{8}$
  & $\frac{3}{8}$ & $\frac{7}{8}$ & $\frac{13}{8}$
  & $\frac{90199}{184320}$ & $\frac{23}{64}$ & $\frac{23}{64}$ & $(-\frac{15}{8}$&$\frac{15}{8})$ \\[3pt]
\hline
\end{tabular}
\caption{Parameters and statistics describing the biased (top half) and bias corrected (bottom half) theoretical error distributions $\vecx_i$ for a 1D block.
Here $i$ indicates the spatial position within the block.
The mean, or bias in error (in ulps), is given by ${\mathbb{E}}[\vecx_i] = c_i$.
}
\label{tab:diststats}
\end{table}
}

\Cref{tab:diststats} gives the parameters of the $\vecx_i$ error distributions (plotted in \Cref{fig:distributions}) both for the biased and unbiased case.
The rational knots have for clarity been scaled by their common denominator, $8$, e.g., $u_{1,4} = \frac{1}{8} \hat{u}_{1,4} = \frac{15}{8}$.
Note the multiplicity of knots $u_{i,2} = u_{i,3}$ and $v_{i,2} = v_{i,3}$.
The nonzero $c_i$ indicate spatially dependent bias in errors that is eliminated using our bias correction.
In addition to this variation in bias of $\vecx_i$, the distribution shapes also vary as a result of the differences in 2-norms of rows of $L^{-1}$, with equal shapes for $\vecx_1$ and $\vecx_4$ and for $\vecx_2$ and $\vecx_3$.
The error distributions for higher-dimensional data ($d \geq 2$) are obtained via convolution of these four base distributions, e.g., $\vecx_{i,j} = \vecx_i * \vecx_j$ when $d = 2$.

\section{Additional Figures}
\label{sec:additionalfigs}
\teal{This appendix presents additional figures comparing experimental and theoretical mean errors for ZFP compression across various scenarios, including different dimensions (1D, 2D, and 3D) and dynamic ranges (\( \rho \), the exponent range of values in a block). 
	The source data consists of synthetic \( 4^d \)-dimensional blocks with values ranging from \( 2^{e_{\text{min}}} \) to \( 2^{e_{\text{max}}} \), where \( \rho = e_{\text{max}} - e_{\text{min}} \). One million blocks were generated, compressed, and decompressed for each scenario, and the resulting errors were analyzed.
	The plots include:
	\begin{itemize}
		\item \textbf{Exp/Pred Error Ratio}: The ratio of experimental to theoretical mean error for each block element.
		\item \textbf{Side-by-Side Comparison}: Experimental (top) vs. theoretical (bottom) mean errors.
		\item \textbf{Relative Error}: The relative difference between experimental and theoretical mean errors as a function of \( \beta \) (bit planes kept).
	\end{itemize}
	These figures validate the theoretical framework and show how compression bias evolves with \( \rho \), \( \beta \), and dimensionality.}
\begin{figure}[h!]
	\centering
	\begin{subfigure}[b]{\textwidth}
		\includegraphics[width=\textwidth]{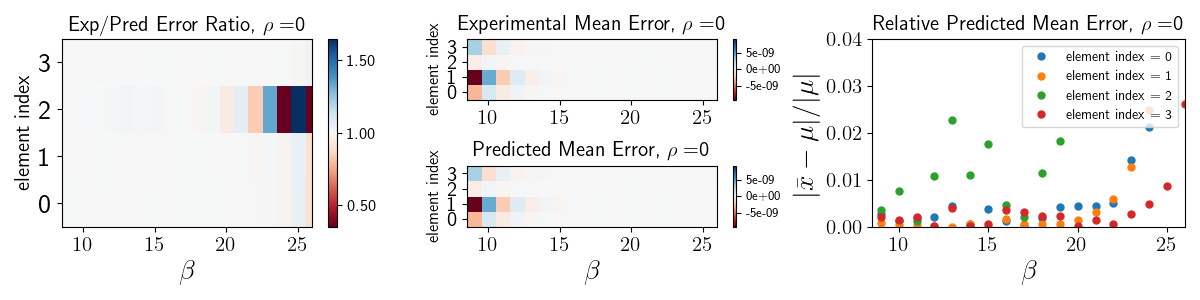}
	\end{subfigure}
	\caption{1-d Simulated Example: Each row depicts the ratio, a side-by-side comparison, and the relative error of the experimental and \teal{predicted error bias} error for $\rho =0 $, where $\rho$ is the \teal{exponent} range of values in a block defined by \cref{eqn:rho}. \teal{This figure uses the original biased variant of ZFP.}} 
	\label{fig:1drelative error_new}
\end{figure}

\begin{figure}[h!]
	\centering
	\begin{subfigure}[b]{\textwidth}
		\includegraphics[width=\textwidth]{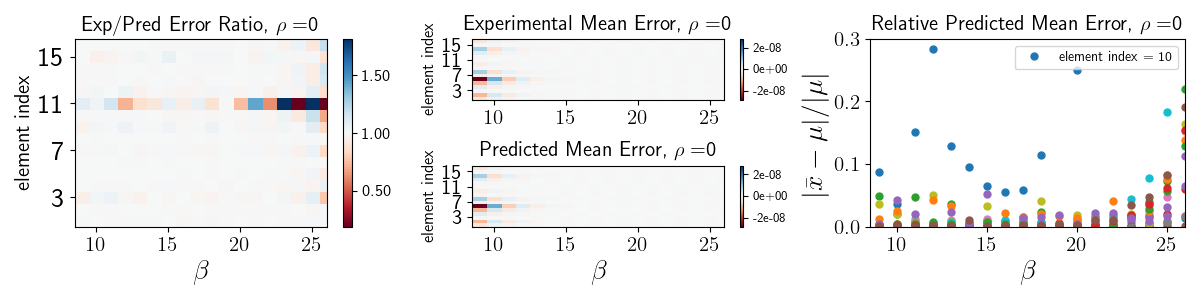}
		\label{fig:rho02d}
	\end{subfigure}
	\begin{subfigure}[b]{\textwidth}
		\includegraphics[width=\textwidth]{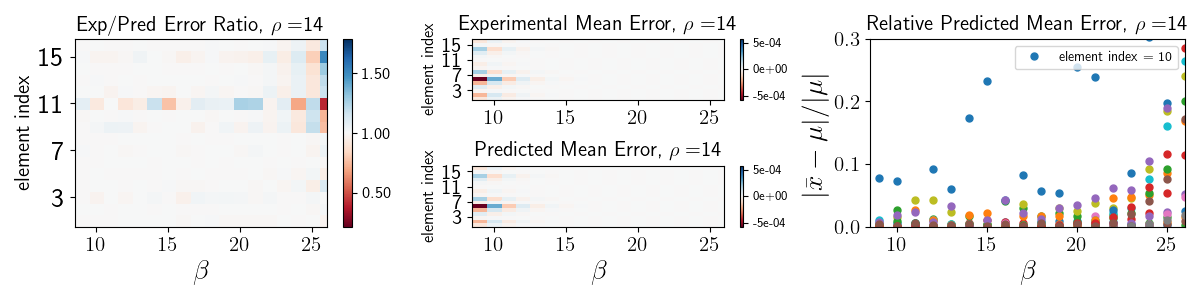}
		\label{fig:rho142d}
	\end{subfigure}
	\caption{2-d Simulated Example: Each row depicts the ratio, a side-by-side comparison, and the relative error of the experimental and \teal{predicted error bias} error for different $\rho$ values, where $\rho$ is the \teal{exponent} range of values in a block defined by \cref{eqn:rho}. \teal{This figure uses the original biased variant of ZFP.}} 
	\label{fig:2drelative error}
\end{figure}

\begin{figure}[h!]
	\centering
	\begin{subfigure}[b]{\textwidth}
		\includegraphics[width=\textwidth]{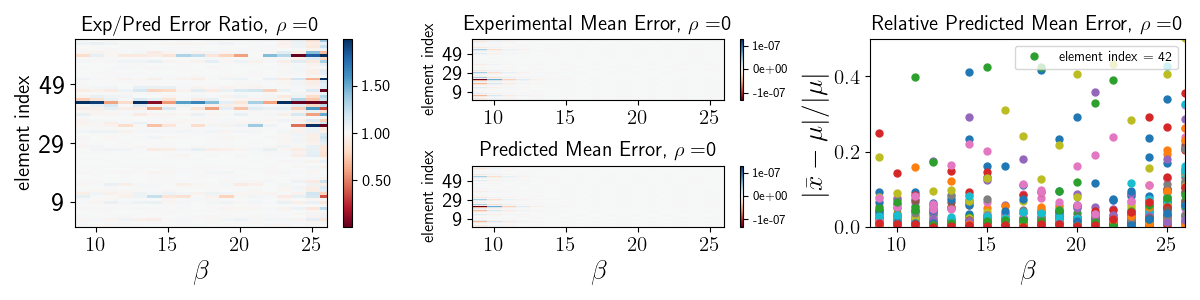}
		\label{fig:rho03d}
	\end{subfigure}
	\begin{subfigure}[b]{\textwidth}
		\includegraphics[width=\textwidth]{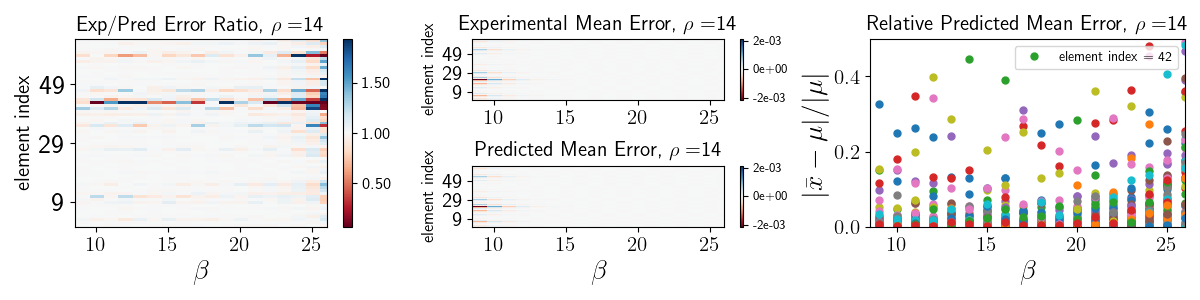}
		\label{fig:rho143d}
	\end{subfigure}
	
	\caption{3-d Simulated Example: Each row depicts the ratio, a side-by-side comparison, and the relative error of the experimental and \teal{predicted error bias} error for different $\rho$ values, where $\rho$ is the \teal{exponent} range of values in a block defined by \cref{eqn:rho}. \teal{This figure uses the original biased variant of ZFP.}} 
	\label{fig:3drelative error}
\end{figure}

 \begin{figure}
	\centering
	 	\begin{subfigure}[b]{\textwidth}
		 		\includegraphics[width=\textwidth]{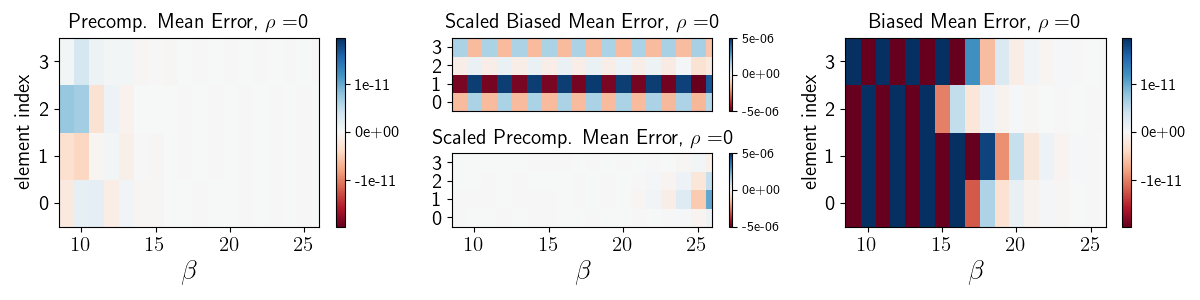}
		 	\end{subfigure}
	
\caption{1-D Simulated {\it Precompression} Rounding Example: The left and right figures depict the unbiased and biased experimental mean error, respectively, using {\it precompression} and the original variant, while the middle figure shows a side-by-side comparison of the unbiased and biased scaled experimental mean error by $\beta$ for $\rho = 0$, where $\rho$ is the exponent range of values in a block defined by \cref{eqn:rho}.} 
	\label{fig:precomp_new} 	
\end{figure}

\begin{figure}
	\centering
		\begin{subfigure}[b]{\textwidth}
			\includegraphics[width=\textwidth]{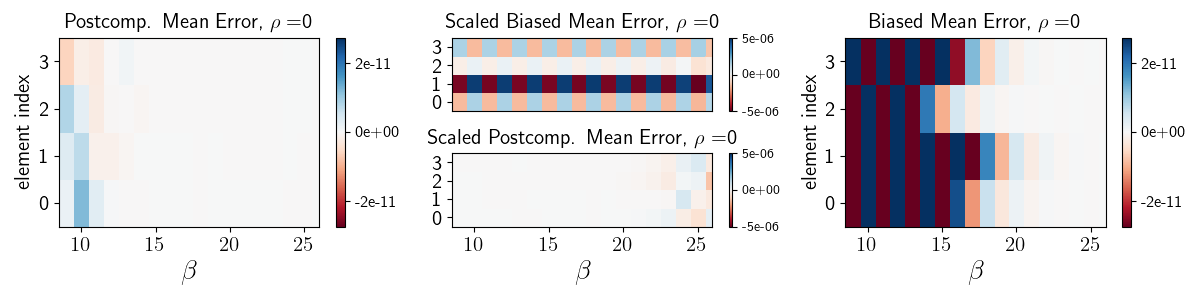}
			\end{subfigure}
	
	\caption{1-d Simulated {\it Postcompression} Rounding Example: \teal{The left figure depicts the unbiased experimental mean error using {\it postcompression} rounding, while the right figure shows a side-by-side comparison of the unbiased and biased scaled experimental mean error by $\beta$ for $\rho = 0$},  where $\rho$ is the \teal{exponent} range of values in a block defined by \cref{eqn:rho}. }
	\label{fig:postcomp_new}
\end{figure}

\begin{figure}[h!]
	\centering
	\begin{subfigure}[b]{\textwidth}
		\includegraphics[width=\textwidth]{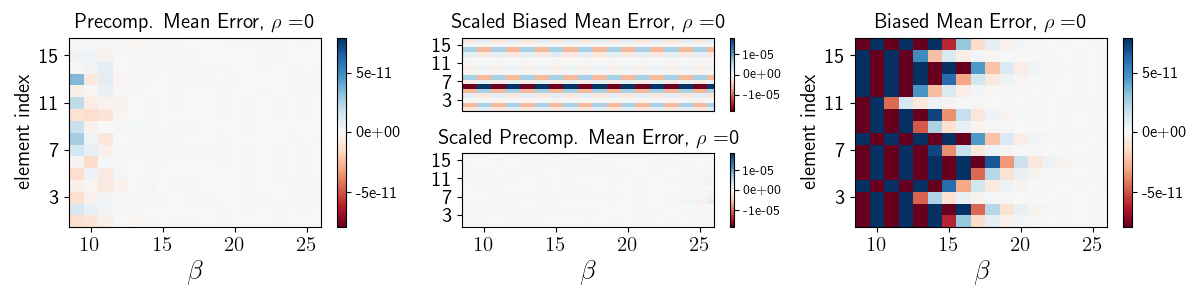}
		\label{fig:precompression_2rho0}
	\end{subfigure}
	\begin{subfigure}[b]{\textwidth}
		\includegraphics[width=\textwidth]{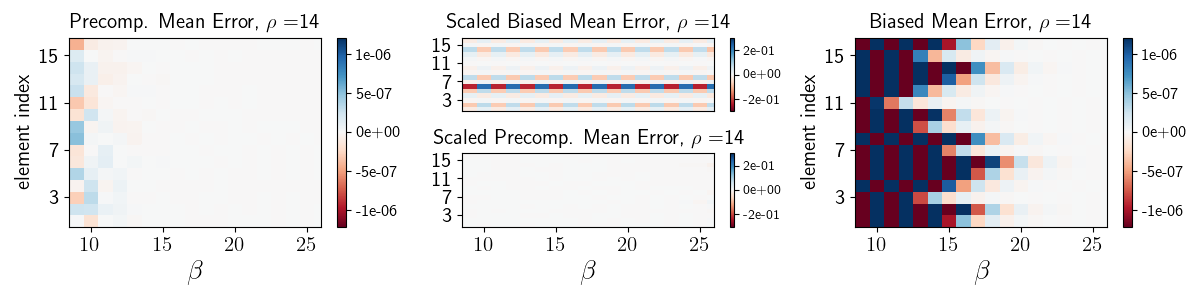}
		\label{fig:precompression_2rho14}
	\end{subfigure}
	
	\caption{2-d Simulated Precompression Example: For each row, \teal{the left and right figures depict the unbiased and biased experimental mean error, respectively, using {\it precompression} and the original variant. The middle figure shows a side-by-side comparison of the unbiased and biased scaled experimental mean error by $\beta$ for} different $\rho$ values, where $\rho$ is the \teal{exponent} range of values in a block defined by \cref{eqn:rho}.}
	\label{fig:precompression2d}
\end{figure}

\begin{figure}[h!]
	\centering
	\begin{subfigure}[b]{\textwidth}
		\includegraphics[width=\textwidth]{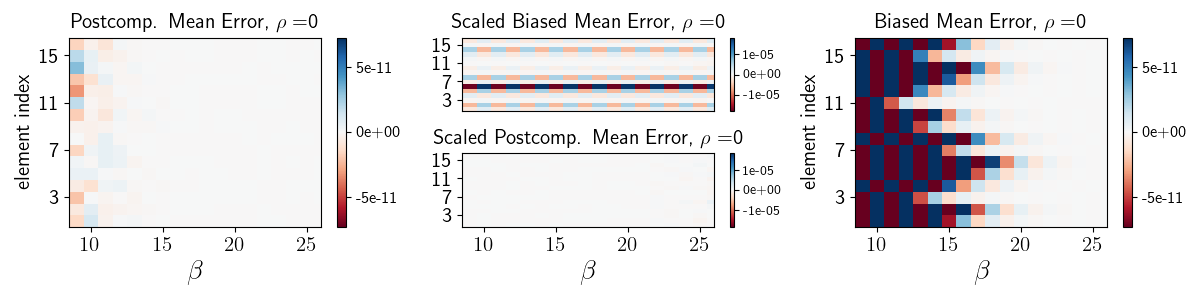}
		\label{fig:postcompression_2rho0}
	\end{subfigure}
	\begin{subfigure}[b]{\textwidth}
		\includegraphics[width=\textwidth]{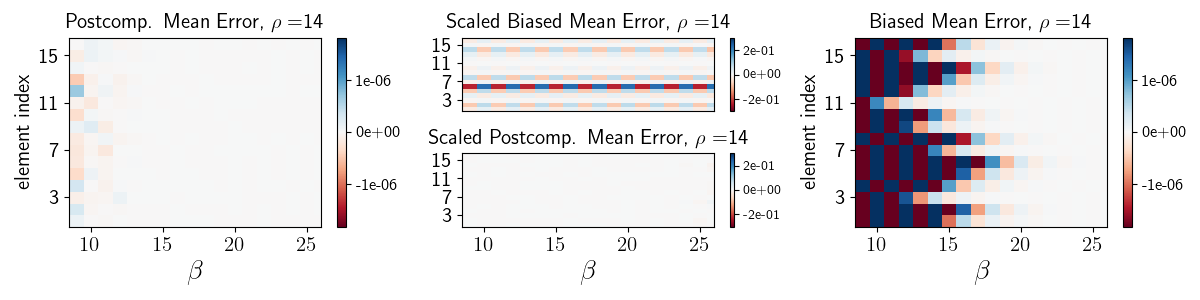}
		\label{fig:postcompression_ratio_2rho0}
	\end{subfigure}
	
	\caption{2-d Simulated Postcompression Example: For each row, \teal{the left and right figures depict the unbiased and biased experimental mean error, respectively, using {\it postcompression} and the original variant. The middle figure shows a side-by-side comparison of the unbiased and biased scaled experimental mean error by $\beta$ for} different $\rho$ values, where $\rho$ is the \teal{exponent} range of values in a block defined by \cref{eqn:rho}.}
	\label{fig:postcompression2d}
\end{figure}

\begin{figure}[h!]
	\centering
	\begin{subfigure}[b]{\textwidth}
		\includegraphics[width=\textwidth]{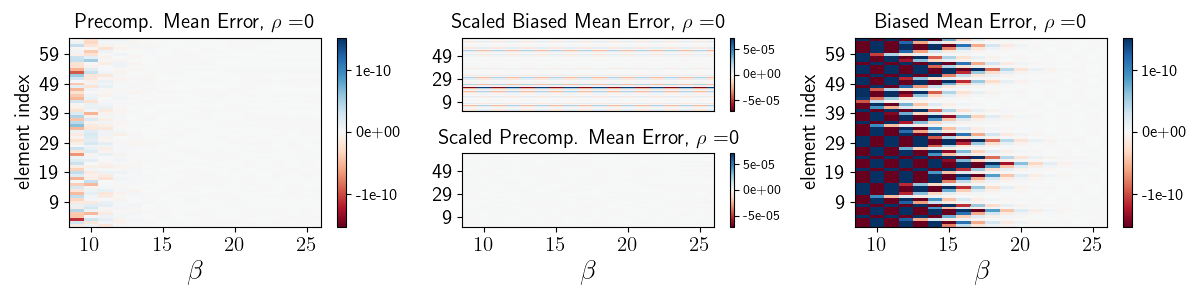}
		\label{fig:precompression_3rho0}
	\end{subfigure}
	\begin{subfigure}[b]{\textwidth}
		\includegraphics[width=\textwidth]{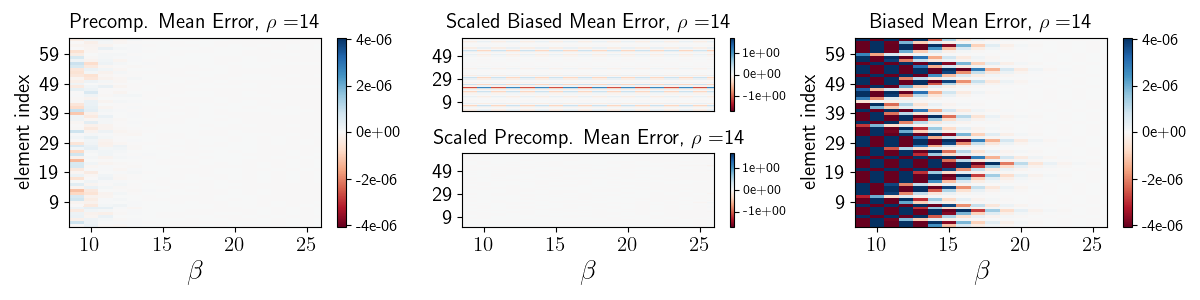}
		\label{fig:precompression_ratio_3rho0}
	\end{subfigure}
	
	\caption{3-d Simulated Precompression Example: For each row, \teal{the left and right figures depict the unbiased and biased experimental mean error, respectively, using {\it precompression} and the original variant. The middle figure shows a side-by-side comparison of the unbiased and biased scaled experimental mean error by $\beta$ for} different $\rho$ values, where $\rho$ is the \teal{exponent} range of values in a block defined by \cref{eqn:rho}.}
	\label{fig:precompression3d}
\end{figure}

\begin{figure}[h!]
	\centering
	\begin{subfigure}[b]{\textwidth}
		\includegraphics[width=\textwidth]{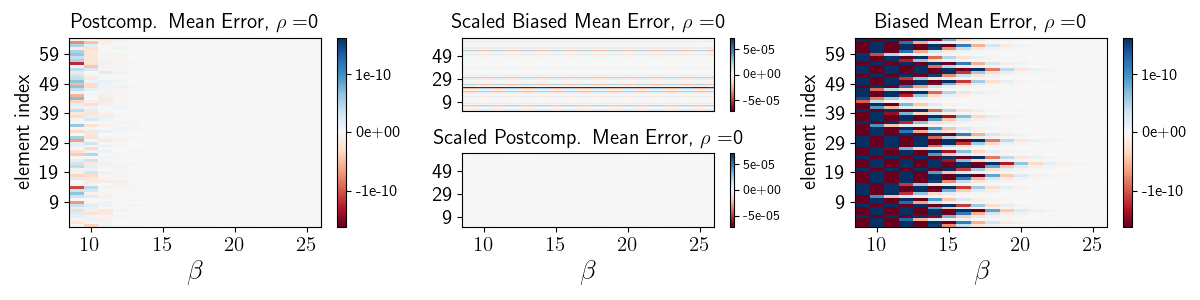}
		\label{fig:postcompression_3rho0}
	\end{subfigure}
	\begin{subfigure}[b]{\textwidth}
		\includegraphics[width=\textwidth]{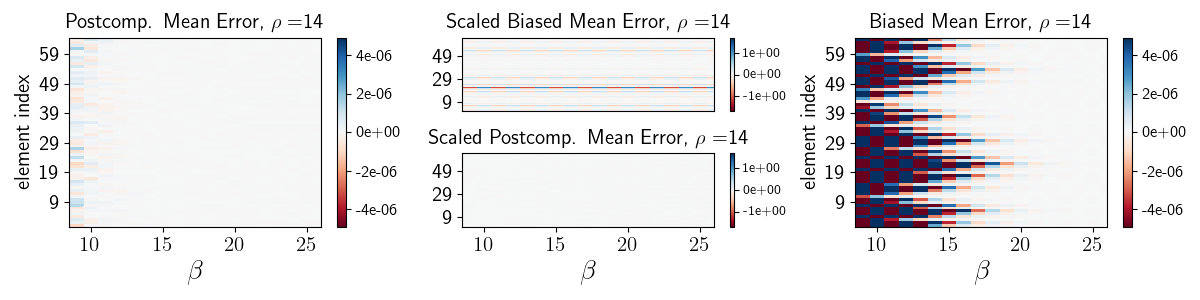}
		\label{fig:postcompression_ratio_3rho0}
	\end{subfigure}
	
	\caption{3-d Simulated Postcompression Example: For each row, \teal{the left and right figures depict the unbiased and biased experimental mean error, respectively, using {\it postcompression} and the original variant. The middle figure shows a side-by-side comparison of the unbiased and biased scaled experimental mean error by $\beta$ for} different $\rho$ values, where $\rho$ is the \teal{exponent} range of values in a block defined by \cref{eqn:rho}.}
	
	\label{fig:postcompression3d}
\end{figure}

\end{document}